\documentclass[a4paper,11pt]{article}
\usepackage[a4paper, %
left=2.8cm, %
right=2.8cm, %
top=2cm, %
bottom=2cm, %
heightrounded]{geometry}
\usepackage{amsmath}
\usepackage{amsthm}
\usepackage{amsfonts}
\usepackage{amssymb}
\usepackage{stmaryrd}
\usepackage{mathrsfs}
\usepackage{enumerate}
\usepackage{sectsty}

\sectionfont{\centering}

\author{Vincent DEVINCK}

\newtheorem{Theo}{Theorem}[section]
\newtheorem{Lem}[Theo]{Lemma}
\newtheorem{Prop}[Theo]{Proposition}
\newtheorem{Cor}[Theo]{Corollary}
\newtheorem{Fac}[Theo]{Fact}

\theoremstyle{definition}
\newtheorem{Def}[Theo]{Definition}
\newtheorem{Ex}[Theo]{Example}
\newtheorem{Ass}[Theo]{Assumption}
\newtheorem{Mot}[Theo]{Motivation}
\newtheorem{Que}[Theo]{Question}
\newtheorem{Rem}[Theo]{Remark}

\numberwithin{equation}{section}

\date{}
\title{\textbf{Strongly mixing operators on Hilbert spaces and speed of mixing}}
\begin{document}
\maketitle

\begin{abstract}
\noindent We investigate the subject of speed of mixing for operators on infinite dimensional Hilbert spaces which are strongly mixing with respect to a nondegenerate Gaussian measure. We prove that there is no way to find a uniform speed of mixing for all square-integrable functions. We give classes of regular functions for which the sequence of correlations decreases to zero with speed $n^{-\alpha}$ when the eigenvectors associated to unimodular eigenvalues of the operator are parametrized by an $\alpha$-H\"olderian $\mathbb{T}$-eigenvector field.
\end{abstract}

\noindent Mathematics Subject Classification (2010) codes: 37A05, 37A25, 58D20, 65F15

\section{Introduction}

In this paper, we will be interested in the dynamics of a bounded linear operator $T$ acting on a complex separable infinite dimensional Hilbert space $\mathcal{H}$ from the measure-theoretic point of view, and more precisely in the strong mixing property of $T$. The study of strongly mixing operators on infinite dimensional spaces was begun in 2006 by Bayart and Grivaux in \cite{BG1} where they give conditions on $T$ so that $T$ admits a nondegenerate Gaussian measure $m$ with respect to which it defines a strongly mixing transformation on $(\mathcal{H},m)$. In a second work, they extended this result when the underlying space is a complex separable Banach space of type different from $2$ (see \cite{BG2}). The idea of studying a linear operator from the measure-theoretic point of view is that measurable dynamics and topological dynamics are connected: if an operator $T$ acting on some complex separable infinite dimensional Banach space $X$ turns out to be ergodic with respect to some nondegenerate measure, then the operator is hypercyclic, that is there exists a vector $x$ in the space $X$ such that the orbit $\mathcal{O}rb(x,T)=\{T^nx\,;\,n\ge 0\}$ of $x$ under the action of $T$ is dense in $X$. Building from the work in \cite{BG1} and \cite{BG2}, our aim is to study the convergence to zero of the correlations in the definition of the strong mixing property.

All the definitions on ergodic theory can be found in \cite{W}. The first important concept is that of measure-preserving transformation.

\begin{Def}
Let $(X,\mathcal{B},m)$ be a probability space. We say that a measurable map $T:(X,\mathcal{B},m) \longrightarrow (X,\mathcal{B},m)$ is a \textit{measure-preserving transformation} if for any measurable set $A$ in $\mathcal{B}$, we have $m(T^{-1}(A))=m(A)$.
\end{Def}

\noindent To begin with, we should recall the fundamental notion of ergodicity.

\begin{Def}
Let $(X,\mathcal{B},m)$ be a probability space. We say that a measure-preserving transformation $T: (X,\mathcal{B},m) \longrightarrow (X,\mathcal{B},m)$ is \textit{ergodic} if one of the two equivalent conditions is satisfied:\\
$(i)$ for every measurable set $A$ in $\mathcal{B}$, if $T^{-1}(A)=A$, then $m(A)=0$ or $m(A)=1$;\\
$(ii)$ for every functions $f,g$ in $L^2(X,\mathcal{B},m)$, 
$$
\frac{1}{N}\sum_{n=0}^{N-1}\int_X f(T^nx)g(x)\,dm(x)\underset{N \to +\infty}{\longrightarrow}\int_X f\,dm \int_X g\,dm
$$
\end{Def}

\noindent We now define the central notion of the paper which is a stronger notion than ergodicity.

\begin{Def}
Let $(X,\mathcal{B},m)$ be a probability space. We say that a measure-preserving transformation $T: (X,\mathcal{B},m) \longrightarrow (X,\mathcal{B},m)$ is \textit{strongly mixing} if one of the two equivalent conditions is satisfied:\\
$(i)$ for every measurable sets $A,B$ in $\mathcal{B}$, $\displaystyle m(T^{-n}(A)\cap B)\underset{n\to +\infty}{\longrightarrow} m(A)m(B)$;\\
$(ii)$ for every functions $f,g$ in $L^2(X,\mathcal{B},m)$, $\displaystyle \int_X f(T^nx)g(x)\,dm(x)\underset{n \to +\infty}{\longrightarrow}\int_X f\,dm \int_X g\,dm$.
\end{Def}
\noindent By the speed of mixing, we refer to the speed with which the \textit{correlation of order n between} $f$ \textit{and} $g$
\begin{equation}\label{expressioncorrelation}
\mathcal{I}_n(f,g):= \int_X f(T^nx)g(x)\,dm(x)-\int_X f\,dm \int_X g\,dm
\end{equation}
converges to zero as $n$ goes to infinity, where $f$ and $g$ belong to $L^2(X,\mathcal{B},m)$.

The fundamental notions of ergodicity, weak mixing (see the definition in \cite{BG1}) or strong mixing in ergodic theory are very studied when we deal with transformations on compact sets (for more on the compact setting, see \cite{KB}). The notion of speed of mixing when the transformation is strongly mixing on a compact set $K$ is systematically studied and the way of computation depends on the structure of the compact set $K$ (see for instance \cite{D1}, \cite{D2}, \cite{Y1} or \cite{Y2}). It is proved in these papers that, in the situations studied here, the sequences of correlations decrease to zero with exponential speed when dealing with some classes of regular functions.

Troughout this paper, $\mathcal{H}$ is a complex separable infinite dimensional Hilbert space, with Borel $\sigma$-algebra $\mathcal{B}$ generated by the bounded real linear functionals $\mathfrak{Re}\langle x,\cdot \rangle: \mathcal{H}\longrightarrow \mathbb{R}$, where $x$ is a vector of $\mathcal{H}$. It is equipped  with a scalar product $\langle u,v\rangle$ which is assumed to be linear with respect to $v$ and conjugate-linear with respect to $u$ and we adopt the convention that all the scalar products which appear in the paper have this property. The algebra of bounded linear operators on the Hilbert space $\mathcal{H}$ is denoted by $\mathcal{B}(\mathcal{H})$. Finally, $\mathbb{T}$ stands for the set of all complex numbers of modulus $1$ and $\mu$ is the normalized Lebesgue measure on $\mathbb{T}$, that is $d\mu=\frac{d\theta}{2\pi}$.

The starting point of our investigation here is Theorem $3.29$ of \cite{BG1} which says that for a bounded linear operator $T$ on $\mathcal{H}$ whose eigenvectors associated to unimodular eigenvalues are $\mu$-spanning (Definition \ref{eigenvector}), there exists a Gaussian measure $m$ on $\mathcal{H}$ (Definition \ref{gaussianmeasure}) for which $T:(\mathcal{H},\mathcal{B},m)\longrightarrow (\mathcal{H},\mathcal{B},m)$ is a strongly mixing transformation. It is the object of Section $2$ to recall the main definitions and ideas around this result which will be useful in our work. In particular, the parametrization of these eigenvectors by $\mathbb{T}$-eigenvector fields (Definition \ref{eigenvectorfield}) in a regular way (Assumption \ref{assumption}) gives us the first basic result about speed of mixing (Proposition \ref{functional}) which is the kind of result we would like to generalize to a broader class of functions than bounded linear functionals on $\mathcal{H}$. We also give examples of strongly mixing operators on Hilbert spaces where the $\mathbb{T}$-eigenvector field is directly given in a regular way (Examples $2.11$, $2.12$, $2.13$ and $2.14$) and where Assumption \ref{assumption} is satisfied. From the examples of $\alpha$-H\"olderian $\mathbb{T}$-eigenvector fields arises a natural question (Question \ref{question}): can we find a Hilbert space and a bounded linear operator on it which admits a $\mathbb{T}$-eigenvector field which is H\"olderian for a fixed H\"older exponent? We give a positive answer to it (Theorem \ref{reponse}). From now on, let $T$ denote a bounded linear operator on the Hilbert space $\mathcal{H}$ whose eigenvectors associated to unimodular eigenvalues are $\mu$-spanning (Definition \ref{eigenvector}) and are parametrized by a $\mathbb{T}$-eigenvector field $E$ (Definition \ref{eigenvectorfield}) which is $\alpha$-H\"olderian ($\alpha\in (0,1]$).

Our main problem is to study the convergence to zero of the sequence of correlations \eqref{expressioncorrelation} for any functions $f$ and $g$ in $L^2(\mathcal{H},\mathcal{B},m)$. A first natural question is to wonder if there exists a sequence of positive real numbers $(s_n)_{n\in\mathbb{N}}$ going to zero as $n$ goes to infinity such that 
\begin{equation}\label{vitesse}
\vert \mathcal{I}_n(f,g)\vert \le C_{f,g}\, s_n\ \ \ \ \ \mathrm{for\ any\ functions}\ f\ \mathrm{and}\ g\ \mathrm{in}\ L^2(\mathcal{H},\mathcal{B},m),
\end{equation}
where the positive constant $C_{f,g}$ only depends on $f$ and $g$. In this case, we will say that the correlation $\mathcal{I}_n(f,g)$ goes to zero with speed $s_n$.

\indent In Section $3$, we give an answer to the problem \eqref{vitesse} and we prove that there is no way to have a speed of mixing in the whole space $L^2(\mathcal{H},\mathcal{B},m)$ (Theorem \ref{non}).\\
\indent Sections $4$ and $5$ are devoted to study the problem \eqref{vitesse} for functions $f$ and $g$ which belong to some classes of regular functions in $L^2(\mathcal{H},\mathcal{B},m)$. In Section $4$, we explain a way of computation of the correlations by considering an orthogonal decomposition of the space $L^2(\mathcal{H},\mathcal{B},m)$ which is given by the theory of Fock spaces. We then make some smoothness assumption on our functions (condition \eqref{intfinite}) and we prove that the Fourier coefficients of such a function $f$ have an integral representation (Lemma \ref{coefficient})  which yields naturally a sequence of useful multilinear forms associated to the components of $f$ in its decomposition as Wiener chaos \eqref{expansion}. After proving several useful estimates involving these multilinear forms, we prove the main theorem of the paper (Theorem \ref{theoremefinal}) on the rate of mixing for some classes of regular functions. More precisely, we define two classes $\mathcal{X}$ and $\mathcal{Y}$ of regular functions in $L^2(\mathcal{H},\mathcal{B},m)$ having the property that for any $f\in \mathcal{X}$ and $g\in \mathcal{Y}$ there exists a positive constant $C_{f,g}$ such that for any $n\ge 1$, $\vert \mathcal{I}_n(f,g)\vert \le C_{f,g}\,n^{-\alpha}$. We finally give some applications of this result by exhibiting concrete functions which belong to our classes $\mathcal{X}$ and $\mathcal{Y}$.

\section{A known result about strongly mixing operators}
We present here a result which is for us the starting point for the question of speed of mixing. It gives a condition for an operator to admit a measure for which the operator is strongly mixing. Before stating the theorem, we need to recall some definitions and facts about Gaussian measures on complex Hilbert spaces. All the definitions and facts on Gaussian measures stated below can be found in one of the references \cite{BM} or \cite{J}. Even if the notion of Gaussian measure can naturally be developped in the Banach space setting, we will stay in the Hilbert case since our method to get a speed of mixing is purely Hilbertian (for the definitions in the Banach case, see for instance \cite{BM}). Furthermore, Bayart and Matheron proved in a recent paper (\cite{BM0}) that the result of this section we are going to present (Theorem \ref{mixing}) is also true when the space is not a Hilbert space with the weaker conclusion that the operator is a weakly mixing transformation.

\subsection{Gaussian measures on Hilbert spaces}
First, we have to introduce complex Gaussian distribution. For any $\sigma>0$, let us denote by $\gamma_\sigma$ the centred Gaussian measure on $\mathbb{R}$ with variance $\sigma^2$, that is
\begin{equation*}
d\gamma_\sigma=\frac{1}{\sigma\sqrt{2\pi}}\,e^{-t^2/2\sigma^2}\,dt.
\end{equation*}
\begin{Def}
Let $(\Omega, \mathcal{F}, \mathbb{P})$ be a probability space and $f : (\Omega,\mathcal{F},\mathbb{P}) \longrightarrow \mathbb{C}$ be a complex-valued measurable function. Then we say that $f$ has \textit{complex symmetric Gaussian distribution} if either $f$ is almost surely equal to zero or the real and imaginary parts $\mathfrak{Re} \, f$ and $\mathfrak{Im} \, f$ have independent centred Gaussian distribution with the same variance.
\end{Def}
\noindent In other words, a nonzero random variable $f$ has a complex Gaussian distribution if and only if its distribution is $\gamma_\sigma\otimes\gamma_\sigma$ for some $\sigma>0$.\\
\noindent It is important to note that if $f$ has complex symmetric Gaussian distribution then so does $\lambda f$ for any complex number $\lambda$. In particular, $f$ and $\lambda f$ have the same distribution when $\lambda$ is a complex number of modulus $1$.
\begin{Def}\label{gaussianmeasure}
A \textit{Gaussian measure} on $\mathcal{H}$ is a probability measure $m$ on $\mathcal{H}$ such that for every vector $x$ of $\mathcal{H}$, the bounded linear functional $\langle x,\cdot \rangle : y \longmapsto \langle x,y \rangle$ has complex symmetric Gaussian distribution when considered as a random variable on $(\mathcal{H},\mathcal{B},m)$.
\end{Def}
\noindent For a random variable $f:(\mathcal{H},\mathcal{B},m)\longrightarrow \mathbb{C}$, we denote by $\mathrm{var}_m(f)$ its variance with respect to $m$, that is
\begin{equation*}
\mathrm{var}_m(f):=\int_\mathcal{H}\vert f(z)\vert^2\,dm(z)-\bigg\vert \int_\mathcal{H} f(z)\,dm(z) \bigg\vert^2
\end{equation*}
and in the special case where $f=\mathfrak{Re}\langle x,\cdot \rangle$ is a bounded real linear functional, we denote the variance of $f$ by
\begin{equation*}
\sigma_x^2:=\int_\mathcal{H} (\mathfrak{Re}\langle x,z\rangle)^2\,dm(z).
\end{equation*}
\noindent A fundamental result is that a Gaussian measure $m$ on $\mathcal{H}$ has finite moments of all orders. In particular, the quantity
\begin{equation}\label{moment}
\int_{\mathcal{H}}\vert\vert x \vert\vert^2\,dm(x)
\end{equation}
is always finite.\\
In order to study the properties of a Gaussian measure $m$, we introduce its \textit{covariance operator} $R_m$ which is defined on $\mathcal{H}$ by the following equation:
\begin{equation*}
\langle R_m x,y\rangle = \int_{\mathcal{H}}\langle x,z\rangle\overline{\langle y,z\rangle} \,dm(z)
\end{equation*}
for every vectors $x,y$ of $\mathcal{H}$. According to \eqref{moment}, $R_m$ is a bounded linear operator which is self-adjoint and positive. Furthermore, $R_m$ is of trace class. In fact, the Gaussian measure $m$ is completely determined by its covariance operator $R_m$: if $R$ is a bounded linear self-adjoint positive operator which is of trace class, then there is a unique Gaussian measure on $\mathcal{H}$ whose covariance operator is $R$ (see \cite{BM}, Chapter $5$).\\
\noindent In the sequel, we need some factorization of the covariance operator $R$ of a Gaussian measure $m$.
Since $R$ is a positive operator, it admits a square root, that is there exists a \textit{unique} pair $(\tilde{\mathcal{H}},K)$ consisting in a separable Hilbert space $\tilde{\mathcal{H}}$ and a bounded linear operator $K: \tilde{\mathcal{H}}\longrightarrow \mathcal{H}$ such that $R=KK^{*}$. By the uniqueness we mean that if $\hat{\mathcal{H}}$ is another Hilbert space and $K_1 : \hat{\mathcal{H}}\to \mathcal{H}$ is such that $R=K_1K_1^{*}$, then there is an isometry $V: \tilde{H}\to \hat{\mathcal{H}}$ such that $K_1=KV^{*}$.

\subsection{The result}
It is now time to state the result which provides strongly mixing operators with respect to a Gaussian measure. We need to point out the main ideas of the proof which will be fundamental in the sequel. The condition for the operator to be strongly mixing with respect to a Gaussian measure is to have sufficiently many eigenvectors associated to unimodular eigenvalues. All the definitions and facts of the present section can be found in \cite{BG1}, \cite{BG2} (or the book \cite{BM} for a summary).
\begin{Def}\label{eigenvector}
Let $T\in \mathcal{B}(\mathcal{H})$. We say that the eigenvectors of $T$ associated to unimodular eigenvalues are $\mu$-\textit{spanning} if for every $\mu$-measurable subset $A$ of $\mathbb{T}$ such that $\mu(A)=1$, the eigenspaces $\text{Ker}(T-\lambda)$, $\lambda\in A$, span a dense subset of $\mathcal{H}$, where $\mu$ is the normalized Lebesgue measure on $\mathbb{T}$.
\end{Def}

\noindent The result that we are going to start with is the following (see for instance \cite{BG1}) and is due to Bayart and Grivaux.

\begin{Theo}\label{mixing}
Let $T\in\mathcal{B}(\mathcal{H})$. If the eigenvectors of $T$ associated to unimodular eigenvalues are $\mu$-spanning, then there exists a nondegenerate invariant Gaussian measure $m$ on $\mathcal{H}$ such that $T : (\mathcal{H},\mathcal{B},m)\longrightarrow (\mathcal{H},\mathcal{B},m)$ is a strongly mixing transformation.
\end{Theo}
Since a Gaussian measure is completely determined by its covariance operator, the idea of the proof of this theorem is to construct directly the covariance operator $R$. Moreover, the positivity of a covariance operator shows that we only need to build its square root $K:\tilde{\mathcal{H}}\longrightarrow \mathcal{H}$, where $\tilde{\mathcal{H}}$ is a separable Hilbert space to determine. When $R$ is a covariance operator, we will denote by $m$ the probability measure associated to the operator $R$ (or equivalently to $K$). Recall that the measure $m$ is said to be \textit{nondegenerate} if $m(U)>0$ for every nonempty open subset $U$ of $\mathcal{H}$. Then, the following fact, which is proved in \cite{BM}, collects the main steps of the proof of Theorem \ref{mixing}.
\begin{Fac}\label{fact}
Let $T\in \mathcal{B}(\mathcal{H})$.\\
$(i)$ If $K$ is a Hilbert-Schmidt operator, then $m$ defines a Gaussian measure on $\mathcal{H}$ $($in fact, this is also a necessary condition$)$.\\
We now assume that $m$ is a Gaussian measure. Then we have:\\
$(ii)$ the probability measure $m$ is nondegenerate if and only if the operator $K$ has dense range;\\
$(iii)$ the probability measure $m$ is $T$-invariant if and only if there exists a co-isometry $V:\tilde{\mathcal{H}}\longrightarrow \tilde{\mathcal{H}}$ such that the intertwining equation 
\begin{equation}\label{intertwining}
TK=KV
\end{equation}
is satisfied;\\
$(iv)$ when the above conditions are realized, the operator $T$ is strongly mixing with respect to the Gaussian measure $m$ if and only if for every vectors $x, y$ in $\mathcal{H}$,
\begin{equation*}
\langle RT^{*n}x,y\rangle\underset{n\to+\infty}{\longrightarrow}0.
\end{equation*}
\end{Fac}
With this result in hands, it remains to exhibit the operators $K : \tilde{\mathcal{H}}\longrightarrow \mathcal{H}$ and $V : \tilde{\mathcal{H}}\longrightarrow \tilde{\mathcal{H}}$ which make the intertwining equation \eqref{intertwining} true. They are defined by using the eigenvectors of $T$ associated to the unimodular eigenvalues. We first need to introduce some terminology.

\begin{Def}\label{eigenvectorfield}
Let $T\in \mathcal{B}(\mathcal{H})$. A bounded map $E:\mathbb{T}\longrightarrow \mathcal{H}$ such that $E(\lambda)$ belongs to $\text{Ker}(T-\lambda)$ for every $\lambda$ in $\mathbb{T}$ is called a $\mathbb{T}$-eigenvector field for $T$.
\end{Def}

The definition of the operator $K$ comes from a parametrization of the eigenvectors of $T$ associated to unimodular eigenvalues which is based on the following fact (see \cite{BG1}, Lemma $3.17$).
\begin{Fac}\label{parametrisation}
There exists a countable family $(E_i)_{i\in I}$ of $\mathbb{T}$-eigenvector fields for $T$ such that $\mathrm{Ker}(T-\lambda)=\overline{\mathrm{span}}^\mathcal{H}\big[E_i(\lambda)\, ;\, i\in I \big]$ for every $\lambda$ in $\mathbb{T}$. 
\end{Fac}
For our purpose, we can assume that we have a unique $\mathbb{T}$-eigenvector field $E$ for $T$ (see \cite{BM} for the general setting). Then, under the assumption that the eigenvectors of $T$ associated to the unimodular eigenvalues are $\mu$-spanning, the following operators $K:\tilde{\mathcal{H}}\longrightarrow \mathcal{H}$ and $V:\tilde{\mathcal{H}} \longrightarrow \tilde{\mathcal{H}}$ defined on the separable Hilbert space $\tilde{\mathcal{H}}=L^2(\mathbb{T},\mu)$ as
\begin{equation*}
Kf=\int_{\mathbb{T}}f(\lambda)E(\lambda)\,d\mu(\lambda)\ \ \ \text{and}\ \ \ Vf(\lambda)=\lambda f(\lambda)
\end{equation*}
satisfy Fact \ref{fact}. Moreover, the adjoint operator $K^{*} : \mathcal{H} \longrightarrow L^2(\mathbb{T},\mu)$ of $K$ is given by
\begin{equation*}
K^{*}x=\overline{\langle x,E(\cdot)\rangle}
\end{equation*}
for every vector $x$ of $\mathcal{H}$. In order to prove that $T$ is strongly mixing with respect to the Gaussian measure $m$ associated to $K$, it is crucial to see that $\langle RT^{*n}x,y\rangle$ is a Fourier coefficient. Indeed, it is proved in Lemma $3.23$ of \cite{BG1} that for every vectors $x, y$ of $\mathcal{H}$ we have the following integral representation:
\begin{equation}\label{representation}
\langle RT^{*n}x,y\rangle=\int_{\mathbb{T}}\lambda^{n}\langle x, E(\lambda)\rangle\overline{\langle y,E(\lambda)\rangle}\,d\mu(\lambda)=\widehat{\mu_{x,y}}(n)
\end{equation}
where 
\begin{equation*}
d\mu_{x,y}(\lambda)=\langle x, E(\lambda)\rangle\overline{\langle y, E(\lambda)\rangle}\,d\mu(\lambda),
\end{equation*}
which is a consequence of the intertwining equation \eqref{intertwining}.
It is obvious that the normalized Lebesgue measure on $\mathbb{T}$ is a Rajchman measure (that is the sequence $(\hat{\mu}(n))_{n\in \mathbb{Z}}$ of its Fourier coefficients tends to zero as $\vert n \vert$ goes to infinity). Since the measure $\mu_{x,y}$ is absolutely continuous with respect to $\mu$, it is also a Rajchman measure (see \cite{Ke}) and the strong mixing property follows.\\
We now want to give another proof of the convergence to zero of the sequence $(\langle RT^{*n}x,y\rangle)_{n\in \mathbb{N}}$ by making some regularity assumption on the $\mathbb{T}$-eigenvector field $E$. The additional assumption will provide a speed of convergence of this sequence and we will see in the examples that this condition of smoothness arises naturally.

\begin{Ass}\label{assumption}
There exists a real number $\alpha$ in $(0,1]$ such that the $\mathbb{T}$-eigenvector field $E$ is $\alpha$-H\"olderian, that is there exists a constant $C(E)>0$ such that
\begin{equation*}
\big\vert\big\vert E(e^{i\theta})-E(e^{i\theta'})\big\vert\big\vert\le C(E)\, \vert \theta-\theta' \vert^{\alpha}
\end{equation*}
for any $\theta, \theta'$ in $[0,2\pi)$.
\end{Ass}

The next result on the convergence to zero of the sequence of correlations $(\langle RT^{*n}x,y\rangle)_{n\in \mathbb{N}}$ explains the initial motivation of the paper.

\begin{Prop}\label{functional}
 There exists a positive constant $C(E,\alpha)$ which only depends on $E$ and $\alpha$ such that, for every vectors $x, y$ of $\mathcal{H}$, we have
\begin{equation*}
\vert \langle RT^{*n}x,y \rangle\vert \le \frac{C(E,\alpha)\,\vert\vert x \vert\vert \, \vert\vert y \vert\vert}{n^{\alpha}}
\end{equation*}
for any positive integer $n$.
\end{Prop}
\begin{proof}
To do this, we use a very classical argument which can be found in \cite{K} (Chapter $1$). 
For $x,y\in \mathcal{H}$, we consider the function $f_{x,y}:\theta\longmapsto \langle x,E(e^{i\theta})\rangle\overline{\langle y,E(e^{i\theta})\rangle}$. By a change of variable in the integral representation \eqref{representation}, we find that
\begin{equation*}
\langle RT^{*n}x,y \rangle=\frac{1}{2}\int_0^{2\pi}\Big(f_{x,y}(\theta)-f_{x,y}\Big(\theta+\frac{\pi}{n}\Big)\Big)e^{in\theta}\,\frac{d\theta}{2\pi}.
\end{equation*}
Since $E$ is $\alpha$-H\"olderian with H\"older constant $C(E)$, the function $f_{x,y}$ is also $\alpha$-H\"olderian and \begin{equation*}
\vert\vert f_{x,y}(\theta)-f_{x,y}(\theta')\vert\vert\le 2\,C(E)\,\vert\vert x\vert\vert\,\vert\vert y\vert\vert\,\vert \theta-\theta'\vert^\alpha
\end{equation*}
for any $\theta, \theta'$ in $[0,2\pi)$. Then the conclusion easily follows with the constant $C(E,\alpha):=C(E)\,\pi^{\alpha}$.
\end{proof}

\begin{Mot}
This proposition shows that, by taking $f=\langle x,\cdot\rangle$ and $g=\overline{\langle y,\cdot\rangle}$, the sequence of correlations $(\mathcal{I}_n(f,g))_{n\in \mathbb{N}}$ goes to zero with speed $n^{-\alpha}$ since
\begin{equation*}
\big\vert \mathcal{I}_n(f,g) \big\vert \le \frac{C(E,\alpha)\,\vert\vert x \vert\vert\, \vert\vert y \vert\vert}{n^{\alpha}}
\end{equation*}
for any positive integer $n$. Thus, it is natural to wonder if there exists a uniform rate of mixing in $L^2(\mathcal{H},\mathcal{B},m)$. For instance, have we got a convergence to zero of the sequence of correlations $(\mathcal{I}_n(f,g))_{n\in \mathbb{N}}$ with speed $n^{-\alpha}$ for any functions $f,g$ in $L^2(\mathcal{H},\mathcal{B},m)$?
\end{Mot}

\subsection{Examples}
In order to illustrate the previous subsection, let us give examples which show that the $\mu$-spanning condition is rather easy to check in general. In practice, the $\mathbb{T}$-eigenvector field is usually directly given with a regular parametrization. We will also construct a bounded linear operator on a Hilbert space which admits an $\alpha$-H\"olderian $\mathbb{T}$-eigenvector field for a fixed H\"older exponent $\alpha$. We denote by $(e_n)_{n\ge 0}$ the canonical basis of $\ell_2(\mathbb{Z}_+)$.

\begin{Ex}\label{analytic}
Let $w$ be a complex number such that $\vert w \vert>1$ and $B$ be the classical backward shift on $\ell_2(\mathbb{Z}_+)$, that is $Be_0=0$ and $Be_n=e_{n-1}$ when $n$ is a positive integer. The eigenvectors of $w B$ associated to the unimodular eigenvalues are $E(\lambda):=\sum_{n\ge 0}\big(\frac{\lambda}{w}\big)^ne_n$, where $\lambda$ belongs to $\mathbb{T}$. It is proved in \cite{BG1} (Example $3.3$) that $(E(\lambda))_{\lambda\in \mathbb{T}}$ is $\mu$-spanning. Furthermore, the map $E:\mathbb{T}\longrightarrow \ell_2(\mathbb{Z}_+)$ is a $\mathbb{T}$-eigenvector field for $wB$ which is a vector-valued analytic function.
\end{Ex}

\begin{Ex}
We deal with an operator which has been introduced by Kalisch in \cite{Kal}: it is the so-called \textit{Kalisch-type operator} $T$ defined on $L^2([0,2\pi])$ by the formula
\begin{equation*}
Tf(\theta)=e^{i\theta}f(\theta)-\int_{0}^{\theta}ie^{it}f(t)\,dt
\end{equation*}
for any $f$ in $L^2([0,2\pi])$.
For any $\alpha$ in $[0,2\pi)$, let $E(e^{i\alpha}):=\mathbf{1}_{(\alpha,2\pi)}$. Then $E(e^{i\alpha})$ is an eigenvector of $T$ associated to the unimodular eigenvalue $e^{i\alpha}$ and it is proved in \cite{BG2} (Example $3.11$) that $(E(e^{i\alpha}))_{\alpha\in [0,2\pi)}$ is $\mu$-spanning. In particular, the map $E:\mathbb{T}\longrightarrow L^2([0,2\pi])$ is a $\mathbb{T}$-eigenvector field for $T$ which is $\frac{1}{2}$-H\"olderian since for every $0\le \alpha < \beta <2\pi$,
\begin{equation*}
\big\vert\big\vert E(e^{i\alpha})-E(e^{i\beta})\big\vert\big\vert_2 =\bigg(\int_{0}^{2\pi}\big\vert E(e^{i\alpha})(\theta)-E(e^{i\beta})(\theta)\big\vert^2\,\frac{d\theta}{2\pi}\bigg)^{1/2} = (\beta-\alpha)^{1/2}.
\end{equation*}
\end{Ex}

\begin{Ex}\label{example}
For a bounded sequence of positive real numbers $\textbf{w}=(w_n)_{n\ge 1}$, we define the weighted backward shift $B_\textbf{w}$ on $\ell_2(\mathbb{Z}_+)$ by $B_\textbf{w} e_0=0$ and $B_\textbf{w} e_n=w_n e_{n-1}$ when $n$ is a positive integer. For the particular sequence $\textbf{w}$ where $w_1=1$ and $w_n=\frac{n}{n-1}$ when $n\ge 2$, the weighted backward shift $B_\textbf{w}$ admits a $\mu$-spanning $\mathbb{T}$-eigenvector field which is $\frac{1}{2}$-H\"olderian.
\end{Ex}

\begin{proof}
It is already known from \cite{BG1} (Example $3.21$) that the $\mathbb{T}$-eigenvector field $E: \mathbb{T}\longrightarrow \ell_2(\mathbb{Z}_+)$ which is defined by
\begin{equation*}
E(\lambda):=\sum_{n\ge 0}\frac{\lambda^n}{w_1\dots w_n}e_n=e_0+\sum_{n\ge 1}\frac{\lambda^n}{n}e_n
\end{equation*}
is $\mu$-spanning. Now, for every $\lambda, \xi$ in $\mathbb{T}$, we have
\begin{equation*}
\vert\vert E(\lambda)-E(\xi)\vert\vert_2^2=\sum_{n\ge 1}\frac{\vert \lambda^n-\xi^n\vert^2}{n^2}=\sum_{n\ge 1}\frac{\vert  (\lambda \overline{\xi})^n-1\vert^2}{n^2}\cdot
\end{equation*}
Then, if we introduce the $2\pi$-periodic function $f: [0,2\pi[ \longrightarrow \mathbb{R}$ by setting
\begin{equation*}
f(\theta):=\sum_{n\ge 1}\frac{\vert e^{in\theta}-1 \vert^2}{n^2}=2\bigg(\frac{\pi^2}{6}-\sum_{n\ge 1}\frac{\cos (n\theta)}{n^2}\bigg),
\end{equation*}
we can check that $f(\theta)=\pi \theta-\frac{\theta^2}{2}$ for every $\theta$ in $[0,2\pi[$. In particular, $f(\theta)$ behaves like $\pi \theta$ when $\theta$ is closed to zero. We deduce from this that $f$ is a Lipschitz function and then that $E$ is $\frac{1}{2}$-H\"olderian.
\end{proof}

\begin{Ex}\label{lipschitz}
In the same way as in the previous example, we consider the weighted sequence $\textbf{w}$ such that $w_1=1$ and $w_n=\big(\frac{n}{n-1}\big)^{\kappa}$ for $n\ge 2$, where $\kappa >\frac{3}{2}$. The weighted backward shift $B_\textbf{w}$ admits a $\mu$-spanning $\mathbb{T}$-eigenvector field which is a Lipschitz function.
\end{Ex}

\begin{proof}
We define the $\mathbb{T}$-eigenvector field $E$ for $B_\textbf{w}$ as in the previous proof:
\begin{equation*}
E(\lambda):=\sum_{n\ge 0}\frac{\lambda^n}{w_1\dots w_n}e_n=e_0+\sum_{n\ge 1}\frac{\lambda^n}{n^{\kappa}}
\end{equation*}
and it is already known from \cite{BG1} (Example $3.21$) that $E$ is $\mu$-spanning. Moreover, for every $\lambda, \xi$ in $\mathbb{T}$, we have
\begin{eqnarray*}
\vert\vert E(\lambda)-E(\xi) \vert\vert_2^2&=&\sum_{n\ge 1}\frac{\vert \lambda^n-\xi^n\vert^2}{n^{2\kappa}}=\vert\lambda-\xi\vert^2\sum_{n \ge 1}\frac{\big\vert\lambda^{n-1}\xi^0+\lambda^{n-2}\xi^1+\dots+\lambda^{0}\xi^{n-1}\big\vert^2}{n^{2\kappa}}\\
&\le & \bigg(\sum_{n\ge 1}\frac{1}{n^{2\kappa-2}}\bigg)\vert\lambda-\xi\vert^2
\end{eqnarray*}
and we conclude that $E$ is a Lipschitz function since the series $\sum_{n\ge 1}n^{-2\kappa+2}$ is convergent by definition of $\kappa$.
\end{proof}
\noindent From the above examples arises the next natural question.

\begin{Que}\label{question}
If $\alpha$ is any fixed number in $(0,1]$, can we find a separable complex Hilbert space and a bounded linear operator on this space which admits a $\mu$-spanning $\mathbb{T}$-eigenvector field which is \textit{exactly} $\alpha$-H\"olderian (that is $\alpha$-H\"olderian and not $\beta$-H\"olderian for any $\beta>\alpha$)?
\end{Que}

\noindent Refining Example \ref{lipschitz}, we can solve this problem. This positive answer is essentially due to the referee (originally, it was a partial answer).

\begin{Theo}\label{reponse}
Let $\alpha$ be a real number in $(0,1]$. There exists a sequence of positive real numbers $\textbf{w}=(w_n)_{n\ge 1}$ such that the weighted backward shift $B_\textbf{w}\in \mathcal{B}(\ell_2(\mathbb{Z}_+))$ admits a $\mathbb{T}$-eigenvector field which is exactly $\alpha$-H\"olderian and $\mu$-spanning.
\end{Theo}

\begin{proof}
For $\alpha=1$, the problem has already been studied in Example \ref{lipschitz}. Let $\alpha \in (0,1)$ and $\kappa=\alpha+\frac{1}{2}.$ Our goal is to prove that there exists a weighted backward shift on $\ell_2(\mathbb{Z}_+)$ which admits a $\mathbb{T}$-eigenvector field which is exactly $\alpha$-H\"olderian and $\mu$-spanning. As in Example \ref{lipschitz}, we consider the weighted sequence $\mathbf{w}$ such that $w_1=1$ and $w_n=\big(\frac{n}{n-1}\big)^{\kappa}$ for $n\ge 2$ and we define the $\mathbb{T}$-eigenvector field $E$ for $B_\textbf{w}$ by setting
$$
E(\lambda):=\sum_{n\ge 0}\frac{\lambda^n}{w_1\dots w_n}e_n=e_0+\sum_{n\ge 1}\frac{\lambda^n}{n^{\kappa}}\cdot
$$
It is already known from \cite{BG1} (Example $3.21$) that $E$ is $\mu$-spanning. We prove that $E$ is not better than $\alpha$-H\"olderian. Let $N$ be an even positive integer and $\lambda_N=e^{i\pi/N}$. Then there exists a positive constant $C$ (which not depends on $N$) such that $\vert \lambda_N^k-1\vert\ge C>0$ for any $k\in \{N,\dots, 3N/2\}$. Then we have
\begin{align*}
\vert\vert E(\lambda_N)-E(1)\vert\vert^2_2 &\ge \sum_{n=N}^{3N/2}\frac{\vert \lambda_N^n -1\vert^2}{n^{2\kappa}}\\
&\ge C^2\sum_{n=N}^{3N/2}\frac{1}{n^{2\alpha+1}}\\
&\ge C^2\int_{N}^{3N/2+1}\frac{dt}{t^{2\alpha+1}}=\frac{C^2}{2\alpha}\bigg(\frac{1}{N^{2\alpha}}-\frac{1}{(3N/2+1)^{2\alpha}}\bigg)
\end{align*}
and we conclude that there exists $C_\alpha>0$ (which only depends on $\alpha$) such that
$$
\vert\vert E(\lambda_N)-E(1)\vert\vert^2_2\ge \frac{C_\alpha}{N^{2\alpha}}\ge \frac{C_\alpha}{2^\alpha \pi^{2\alpha}}\vert \lambda_N-1\vert^{2\alpha}.
$$
This proves that the $\mathbb{T}$-eigenvector field $E$ is not better than $\alpha$-H\"olderian. We now prove that $E$ is $\alpha$-H\"olderian. We just need to prove that $E$ is $\alpha$-H\"olderian at $1$ (since $\vert \lambda^n-\xi^n\vert=\vert (\lambda\overline{\xi})^n-1\vert$ for every $\lambda,\xi\in \mathbb{T}$). We write $\lfloor x\rfloor$ to denote the integer part of the real number $x$. Let $\lambda\in \mathbb{T}\setminus\{1\}$ (close to $1$) and we put $N=\Big\lfloor \frac{1}{\vert \lambda-1\vert}\Big\rfloor$. Then we have
\begin{align*}
\vert\vert E(\lambda)-E(1)\vert\vert^2 &=\sum_{n=1}^N\frac{\vert \lambda^n-1\vert^2}{n^{2\alpha+1}}+\sum_{n=N+1}^{+\infty}\frac{\vert \lambda^n-1\vert^2}{n^{2\alpha+1}}\\
&:=\mathcal{S}_1(\lambda)+\mathcal{S}_2(\lambda)
\end{align*}
where we denote the first sum by $\mathcal{S}_1(\lambda)$ and the second one by $\mathcal{S}_2(\lambda)$. We estimate $\mathcal{S}_1(\lambda)$ by using the inequality $\vert \lambda^n-1\vert\le n\vert \lambda-1\vert$:
\begin{align*}
\mathcal{S}_1(\lambda)\le \vert \lambda-1\vert^2\sum_{n=1}^N \frac{1}{n^{2\alpha-1}}
\end{align*}
and 
\begin{align*}
\sum_{n=1}^N \frac{1}{n^{2\alpha-1}}&=(1-2\alpha)\sum_{n=1}^N\int_1^n\frac{dt}{t^{2\alpha}}+N\\
&=(1-2\alpha)\int_1^N\bigg(\sum_{t<n\le N}1\bigg)\frac{dt}{t^{2\alpha}}+N\\
&=(1-2\alpha)\int_1^N (N-\lfloor t\rfloor)\frac{dt}{t^{2\alpha}}+N.
\end{align*}
Putting $\{t\}$ for the fractional part of $t$ (that is $\{t\}=t-\lfloor t\rfloor$), we get
\begin{align*}
\sum_{n=1}^N \frac{1}{n^{2\alpha-1}}&=N\bigg(\frac{1}{N^{2\alpha-1}}-1\bigg)-\frac{1-2\alpha}{2-2\alpha}\bigg(\frac{1}{N^{2\alpha-2}}-1\bigg)+(1-2\alpha)\int_1^N\frac{\{t\}}{t^{2\alpha-1}}\,dt+N\\
&\le\frac{1}{(2-2\alpha)N^{2\alpha-2}}+\bigg\vert \frac{1-2\alpha}{2-2\alpha}\bigg\vert+\bigg\vert \frac{1-2\alpha}{2-2\alpha}\bigg\vert\bigg(\frac{1}{N^{2\alpha-2}}+1\bigg)
\end{align*}
It easily follows that there exists a positive constant $D_\alpha$ such that
\begin{align*}
\mathcal{S}_1(\lambda)\le D_\alpha \vert\lambda-1\vert^{2\alpha}.
\end{align*}
We then estimate $\mathcal{S}_2(\lambda)$ by using the inequality $\vert \lambda-1\vert\le 2$ and we get
\begin{align*}
\mathcal{S}_2(\lambda)&\le 4\sum_{n=N+1}^{+\infty}\frac{1}{n^{2\alpha+1}}\le 4\int_N^{+\infty}\frac{dt}{t^{2\alpha+1}}=\frac{2}{\alpha N^{2\alpha}}
\end{align*}
and there exists a positive constant $E_\alpha$ such that $\mathcal{S}_2(\lambda)\le E_\alpha \vert\lambda-1\vert^{2\alpha}$. Finally, we conclude that the $\mathbb{T}$-eigenvector field $E$ is $\alpha$-H\"olderian, which concludes the proof of Theorem \ref{reponse}.
\end{proof}

\section{A negative result}
The aim of this section is to show that, given a bounded linear operator $T$ on $\mathcal{H}$ whose eigenvectors associated to unimodular eigenvalues are parametrized by a $\mu$-spanning $\mathbb{T}$-eigenvector field, there is no uniform rate of decrease of the sequence of correlations for every functions in $L^{2}(\mathcal{H},\mathcal{B},m)$.\\
\noindent Since the covariance operator $R$ is a positive trace class operator, the Hilbert space $\mathcal{H}$ has an orthonormal basis $(e_n)_{n\in \mathbb{N}}$ consisting of eigenvectors of $R$. Thus $Re_n=\lambda_n e_n$ where $\lambda_n \ge 0$, $\text{tr} R=\sum_{n\ge 1}\lambda_n<+\infty$ and one can easily shows that $\lambda_n=2\,\sigma_n^2$, where $\sigma_n^2$ is the variance of the Gaussian random variable $\mathfrak{Re}\langle e_n,\cdot\rangle$ with respect to the measure $m$. The following property of this basis is fundamental for the sequel.
\begin{Prop}\label{basis}
The sequence of random variables $(\langle e_k,\cdot\rangle)_{k\in \mathbb{N}}$ is orthogonal in the space $L^2(\mathcal{H},\mathcal{B},m)$. In fact, the complex Gaussian variables $\langle e_k,\cdot \rangle$ are independent.
\end{Prop}
\begin{proof}
Since $(e_k)_{k\in \mathbb{N}}$ is an orthogonal sequence of eigenvectors of $R$, $\langle Re_k,e_\ell\rangle=0$ for every distinct positive integers $k$ and $\ell$. Hence, the first statement holds by definition of $R$. Furthermore, the independence comes from the fact that for any distinct positive integers $k$ and $\ell$, the real Gaussian variables $\mathfrak{Re}\langle e_k,\cdot\rangle$ and $\mathfrak{Im}\langle e_k,\cdot\rangle$ are independent from $\mathfrak{Re}\langle e_\ell,\cdot\rangle$ and $\mathfrak{Im}\langle e_\ell,\cdot\rangle$ since they are orthogonal real Gaussian variables. Indeed, the covariances of these real Gaussian variables are computed in the following lemma which is a consequence of the rotation invariance of a Gaussian measure (see \cite{BG2}, Section $3$).
\begin{Lem}\label{rotations}
For every vectors $x,y$ of $\mathcal{H}$, we have
\begin{equation}\label{rotation1}
 \big\langle \mathfrak{Re}\langle x,\cdot\rangle,\mathfrak{Re}\langle y,\cdot \rangle\big\rangle_{L^2(m)}=\big\langle \mathfrak{Im}\langle x,\cdot\rangle,\mathfrak{Im}\langle y,\cdot \rangle\big\rangle_{L^2(m)}=\frac{1}{2}\mathfrak{Re}\langle Rx,y\rangle
\end{equation}
and
\begin{equation}\label{rotation2}
\big\langle \mathfrak{Im}\langle x,\cdot\rangle,\mathfrak{Re}\langle y,\cdot \rangle\big\rangle_{L^2(m)}=-\big\langle \mathfrak{Re}\langle x,\cdot\rangle,\mathfrak{Im}\langle y,\cdot \rangle\big\rangle_{L^2(m)}=\frac{1}{2}\mathfrak{Im}\langle Rx,y\rangle.
\end{equation}
\end{Lem}
\end{proof}
\noindent The integral representation of the correlations \eqref{representation} gives us the corresponding result in $L^2(\mathbb{T},\mu)$.
\begin{Cor}\label{orthogonality}
The sequence of functions $\big(\langle e_k,E(\cdot)\rangle\big)_{k\in \mathbb{N}}$ is orthogonal in $L^2(\mathbb{T},\mu)$
 and 
\begin{equation*}
\int_{\mathbb{T}}\big\vert \langle e_k,E(\lambda)\rangle\big\vert^2\,d\mu(\lambda)=\int_{\mathcal{H}}\big\vert \langle e_k,x\rangle \big\vert^2\,dm(x)=2\sigma_k^2
\end{equation*}
for any positive integer $k$.
\end{Cor}
\noindent We now introduce the \textit{complex Gaussian space} 
\begin{equation*}
\mathcal{G}_{\mathbb{C}}:=\overline{\text{span}}^{L^2(\mathcal{H},\mathcal{B},m)}\big[\langle e_k,\cdot \rangle\, ;\, k\in \mathbb{N}\big].
\end{equation*}
This subspace is called \textit{Gaussian} in the sense that any function in $\mathcal{G}_{\mathbb{C}}$ has complex symmetric Gaussian distribution. We can now prove the main result of this section.
\begin{Theo}\label{non}
Let $T\in \mathcal{B}(\mathcal{H})$ be a bounded linear operator on $\mathcal{H}$ whose eigenvectors associated to unimodular eigenvalues are parametrized by a $\mu$-spanning $\mathbb{T}$-eigenvector field $E$. Then, for every null sequence $(s_n)_{n\in \mathbb{N}}$ of positive real numbers, there exists a function $f$ in $\mathcal{G}_{\mathbb{C}}$ such that
\begin{equation*}
\big\vert \mathcal{I}_n(\overline{f},f) \big\vert \ge s_n
\end{equation*}
for any positive integer $n$, where $\overline{f}$ denotes the function $x\longmapsto \overline{f(x)}$.
\end{Theo}
\begin{proof}
Since the random variables $\langle e_k,\cdot \rangle$ are orthogonal in the space $L^2(\mathcal{H},\mathcal{B},m)$ by Proposition \ref{basis}, for every function $f$ in  $\mathcal{G}_{\mathbb{C}}$ we can find a sequence of complex numbers $(a_k)_{k\in \mathbb{N}}$ such that
\begin{equation*}
f=\sum_{k\ge 1}a_k\langle e_k,\cdot \rangle\quad \text{with} \quad \sum_{k= 1}^{+\infty}\vert a_k \vert^2\sigma_k^2<+\infty.
\end{equation*}
Now the random variable $f$ is centered and then the integral representation \eqref{representation} gives us
\begin{align*}
\mathcal{I}_n(\overline{f},f)&=\sum_{(k,\ell)\in\mathbb{N}^2}\overline{a_k}a_\ell\int_{\mathcal{H}}\overline{\langle e_k,T^nx\rangle}\langle e_\ell,x\rangle\,dm(x)=\sum_{(k,\ell)\in \mathbb{N}^2}\overline{a_k}a_\ell \overline{\langle RT^{*n}e_k,e_\ell\rangle} \\
&=\sum_{(k,\ell)\in \mathbb{N}^2}\overline{a_k}a_\ell\int_{\mathbb{T}}\lambda^{-n}\overline{\langle e_k,E(\lambda)\rangle}\langle e_\ell,E(\lambda)\rangle\,d\mu(\lambda)=\big\langle V^n(f\circ E),f\circ E\big\rangle_{L^2(\mathbb{T},\mu)},
\end{align*}
that is $\big\{\mathcal{I}_n(\overline{f},f)\,;\, n\in \mathbb{N}\big\}$ is the weak orbit of the vector $f\circ E$ under the action of the operator $V$ of multiplication by the variable $\lambda$ on $L^2(\mathbb{T},\mu)$. Then we consider the closed subspace of $L^2(\mathbb{T},\mu)$: 
\begin{equation*}
\mathcal{E}=\overline{\text{span}}^{L^2(\mathbb{T},\mu)}\big[\langle e_k, E(\cdot)\rangle \, ; \, k\in \mathbb{N}\big].
\end{equation*}
Since $E(\lambda)$ is an eigenvector of $T$ associated to the eigenvalue $\lambda$, the subspace $\mathcal{E}$ of $L^2(\mathbb{T},\mu)$ is $V$-invariant. At this stage, we apply a result of \cite{BMu} due to Badea and M\"uller which deals with the speed of convergence to zero of the weak orbits $(\langle S^nx,y\rangle)_{n\in \mathbb{N}}$ of an operator $S$ such that $S^n\longrightarrow 0$ in the weak operator topology. In particular, it is proved in \cite{BMu} that if $S$ is a bounded linear operator on a complex Hilbert space $H$ with spectral radius equal to $1$ such that $S^n\longrightarrow 0$ in the weak operator topology, then for any sequence $(s_n)_{n\ge 1}$ of positive numbers which decreases to zero, we can find a vector $x$ in $H$ such that $\vert\langle S^nx,x\rangle\vert \ge s_n$ for any positive integer $n$. Applying this result here, we get that for any sequence $(s_n)_{n\ge 1}$ of positive numbers which decreases to zero there exists a function $f_E$ in $\mathcal{E}$ such that $\vert\langle V^nf_E,f_E\rangle_{L^2(\mathbb{T},\mu)}\vert\ge s_n$ for any positive integer $n$. We expand $f_E$ as
\begin{equation*}
f_E=\sum_{k\ge 1}a_k\langle e_k,E(\cdot)\rangle\ \ \ \mathrm{with}\ \ \ \sum_{k=1}^{+\infty}\vert a_k\vert^2\, \sigma_k^2<+\infty,
\end{equation*}
and we conclude that $f=\sum_{k\ge 1}a_k\langle e_k,\cdot\rangle$ is a function in $\mathcal{G}_{\mathbb{C}}$ which satisfied the conclusion of Theorem \ref{non}.
\end{proof}

\begin{Rem}\label{plusieurs}
In our work, we consider bounded linear operators which admit only one $\mathbb{T}$-eigenvector field $E$ which is $\mu$-spanning. In a more general situation, the $\mathbb{T}$-eigenvectors of the bounded linear operator $T\in \mathcal{B}(\mathcal{H})$ are parametrized by a countable family of $\mathbb{T}$-eigenvector fields $(E_i)_{i\in I}$ (see Fact \ref{parametrisation}). Then the operator $K$ is defined on the Hilbert space $\bigoplus_{i\in I}L^2(\mathbb{T},\mu)$ by 
\begin{equation*}
K(\oplus_{i\in I} f_i)=\sum_{i\in I} \alpha_i K_{E_i}(f_i)\ \ \ \mathrm{where}\ \ \ K_{E_i}(f_i)=\int_{\mathbb{T}}f_i(\lambda)E_i(\lambda)\,d\mu(\lambda)
\end{equation*}
and where $(\alpha_i)_{i\in I}$ is a sequence of positive numbers such that $\sum_{i\in I}\alpha_i^2\vert\vert E_i \vert\vert_2^2<\infty$ where $\vert\vert E_i \vert\vert_2^2=\int_\mathbb{T}\vert\vert E_i(\lambda)\vert\vert^2\,d\mu(\lambda)$ (and we put $R:=KK^{*}=\sum_{i\in I}^{} \alpha_i^2 K_{E_i}K_{E_i}^*$). In this case, it readily follows from the proof of Theorem \ref{non} that there is no uniform rate of decrease in this situation too.
\end{Rem}

This result shows that there is no uniform rate of decrease of the correlations in the whole space $L^2(\mathcal{H},\mathcal{B},m)$. The rest of the paper is devoted to find a speed of mixing for classes of regular functions of $L^2_{\mathbb{R}}(\mathcal{H},\mathcal{B},m):=\{f:\mathcal{H}\longrightarrow \mathbb{R}\ ;\ f\in L^2(\mathcal{H},\mathcal{B},m)\}$. In a first step, we will need to compute the correlations $\mathcal{I}_n(P,Q)$ where $P$ and $Q$ are real polynomials in several variables. We can do this by using the Fock space associated to $L^2_{\mathbb{R}}(\mathcal{H},\mathcal{B},m)$.

\section{Orthogonal decomposition of $L^2_{\mathbb{R}}(\mathcal{H},\mathcal{B},m)$}
In this section, we aim to present a way to compute the correlations $\mathcal{I}_n(f,g)$ for arbitrary functions $f, g$ in $L^2_{\mathbb{R}}(\mathcal{H},\mathcal{B},m)$. In order to do this, we will first explain the construction of the Fock space over a Gaussian subspace of $L^2_{\mathbb{R}}(\mathcal{H},\mathcal{B},m)$ and we will establish some helpful properties of the orthogonal components of the Fock space before giving the general formula for the correlations.

\subsection{Fock space over a Gaussian space}
The theory of Fock spaces will allow us to compute the correlations for arbitrary real functions in $L^2(\mathcal{H},\mathcal{B},m)$ by using an orthogonal decomposition of this space. We begin by recalling some definitions and facts on Fock spaces that will be useful in the sequel; for a thorough account see \cite{J} or \cite{P}.\\
We denote by $\mathbb{Z}^*$ the set of integers different from zero. In the sequel, we denote by $(\mathfrak{e}_\ell)_{\ell\in \mathbb{Z}^*}$ the sequence of vectors of $\mathcal{H}$ defined by 
$$
\mathfrak{e}_\ell=e_\ell\qquad \mathrm{and}\qquad \mathfrak{e}_{-\ell}=ie_\ell
$$
for any positive integer $\ell$. Recall that for any positive integer $\ell$, $\sigma_\ell^2$ denotes the variance of the Gaussian random variables $\mathfrak{Re}\langle \mathfrak{e}_\ell,\cdot\rangle=\mathfrak{Re}\langle e_\ell,\cdot\rangle$ and $\mathfrak{Re}\langle \mathfrak{e}_{-\ell},\cdot\rangle=\mathfrak{Im}\langle e_\ell,\cdot\rangle$:
$$
\sigma_\ell^2=\int_\mathcal{H}(\mathfrak{Re}\langle \mathfrak{e}_\ell,x\rangle)^2\,dm(x)=\int_\mathcal{H}(\mathfrak{Re}\langle \mathfrak{e}_{-\ell},x\rangle)^2\,dm(x)
$$
and we put $\sigma_{-\ell}^2:=\sigma_\ell^2$. We also denote by $\mathcal{G}$ the real Gaussian space 
\begin{equation*}
\mathcal{G}=\overline{\text{span}}^{L^2_{\mathbb{R}}(\mathcal{H},\mathcal{B},m)}[\mathfrak{Re}\langle \mathfrak{e}_k, \cdot \rangle ; k\in \mathbb{Z}^*].
\end{equation*}
Since $\mathcal{B}$ is the $\sigma$-algebra generated by the functions in $\mathcal{G}$, we get by applying the Weierstrass Theorem:
\begin{equation*}
L^2_{\mathbb{R}}(\mathcal{H},\mathcal{B},m)=\overline{\text{span}}^{L^2_{\mathbb{R}}(\mathcal{H},\mathcal{B},m)}[g^k ; g\in \mathcal{G}, k\in \mathbb{Z}_+].
\end{equation*}
Let $\mathcal{G}^k$ denote the space of homogeneous polynomials of degree $k$ of elements of $\mathcal{G}$, with $\mathcal{G}^0=\mathbb{R}$. Then the spaces $\mathcal{G}^k$ are linearly independent (see \cite{P}, Chapter $8$, Lemma $2.3$) and we can orthonormalize them by the so-called \textit{Wick transform}.

\begin{Def}
The Wick tranform $:f:$ of a function $f$ belonging to one of the spaces $\mathcal{G}^k$ is defined in the following way:\\
$(i)$ if $f$ is constant, $:f:=f$;\\
$(ii)$ if $f\in \mathcal{G}^k$, $k\ge 1$, then $:f:=f-\mathcal{P}_k(f)$,
where $\mathcal{P}_k$ denotes the orthogonal projection onto the closure in $L^2_{\mathbb{R}}(\mathcal{H},\mathcal{B},m)$ of $\text{span} [\mathcal{G}^j\, ;\, 0\le j \le k-1]$.\\
We also define $:\mathcal{G}^{k}:$ to be the space $\{:f: \ ;\ f\in \mathcal{G}^k\}$.
\end{Def}
By definition of the Wick transform, we have an orthogonal decomposition of $L^2_{\mathbb{R}}(\mathcal{H},\mathcal{B},m)$ as
\begin{equation*}
L^2_{\mathbb{R}}(\mathcal{H},\mathcal{B},m)=\bigoplus_{k\ge 0}:\mathcal{G}^{k}:
\end{equation*}
and a function $f$ in $L^2_{\mathbb{R}}(\mathcal{H},\mathcal{B},m)$ can be decomposed into its so-called \textit{Wiener chaos decomposition}
\begin{equation}\label{expansion}
f=\sum_{k\ge 0}\mathcal{P}_{:\mathcal{G}^k:}f,
\end{equation}
where $\mathcal{P}_{:\mathcal{G}^k:}$ denotes the orthogonal projection onto the space $:\mathcal{G}^k:$.\\
Our aim is to identify the space $L^2_{\mathbb{R}}(\mathcal{H},\mathcal{B},m)$ with the Fock space over $\mathcal{G}$ by using this decomposition. We define the scalar product $\langle \cdot , \cdot \rangle_{\otimes}$ on the Hilbert tensor product $\bigotimes_k\mathcal{G}$ by setting, for every $g_1,\dots,g_k, h_1,\dots,h_k$ in $\mathcal{G}$,

\begin{equation*}
\langle g_1\otimes\dots\otimes g_k, h_1\otimes\dots\otimes h_k \rangle_{\otimes}=\langle g_1, h_1\rangle_{L^2(m)} \dots \langle g_k, h_k \rangle_{L^2(m)}.
\end{equation*}
We then introduce the space $\mathcal{G}_{\odot}^k$ which is the range of the projection
\begin{equation*}
 \textrm{Sym} : \bigotimes_k{\mathcal{G}}\longrightarrow \mathcal{G}_{\odot}^k
\end{equation*} 
 defined by, for every $f_1,\dots,f_k$ in $\mathcal{G}$,
\begin{equation}\label{sym}
\text{Sym}(f_1\otimes\dots\otimes f_k)=\frac{1}{k!}\sum_{\tau \in \mathfrak{S}_k}f_{\tau(1)}\otimes\dots\otimes f_{\tau(k)},
\end{equation}
where $\mathfrak{S}_k$ denotes the group of permutations of the set $\{1,\dots,k\}$.
For convenience, we endow $\mathcal{G}_{\odot}^k$ with a new scalar product $\langle \cdot,\cdot \rangle_{\odot}$ by setting
\begin{equation}\label{scalarproduct}
\langle f,g \rangle_{\odot}=k!\,\langle f,g \rangle_{\otimes}
\end{equation}
for every $f,g$ in $\mathcal{G}_{\odot}^k$.
\begin{Def}
The Fock space $\mathcal{F}(\mathcal{G})$ over $\mathcal{G}$ is defined by
\begin{equation*}
\mathcal{F}(\mathcal{G})=\bigoplus_{k\ge 0}\mathcal{G}_{\odot}^k
\end{equation*}
where the sum is an orthogonal direct sum and each $\mathcal{G}_{\odot}^k$ is endowed with the scalar product $\langle \cdot,\cdot \rangle_{\odot}$.
\end{Def}
\noindent The main interest of this is that the map
\begin{eqnarray}\label{isometrie}
\notag :\mathcal{G}^{k}:&\longrightarrow & \mathcal{G}_{\odot}^k\\
 :f_1\dots f_k:&\longmapsto & \text{Sym}(f_1\otimes\dots\otimes f_k)
\end{eqnarray}
extends uniquely to an isometry from $\big(:\mathcal{G}^{k}:,\langle \cdot,\cdot\rangle_{L^2(m)}\big)$ onto $\big(\mathcal{G}_{\odot}^k,\langle \cdot,\cdot\rangle_{\odot}\big)$. Hence, the orthogonal decomposition
\begin{equation*}
L_{\mathbb{R}}^2(\mathcal{H},\mathcal{B},m)=\bigoplus_{k\ge 0}:\mathcal{G}^{k}:
\end{equation*}
allows us to make the identification $L_{\mathbb{R}}^2(\mathcal{H},\mathcal{B},m)=\mathcal{F}(\mathcal{G})$.\\
\noindent In order to compute the correlations between two functions in $L^2_{\mathbb{R}}(\mathcal{H},\mathcal{B},m)$, we are now going to find  a helpful decomposition of the functions $\mathcal{P}_{:\mathcal{G}^k:}f$ which appear in \eqref{expansion}.

\subsection{Canonical decomposition in the spaces $:\mathcal{G}^k:$}
It is now time to understand more precisely the spaces $:\mathcal{G}^k:$ and to make some computations in them. We shall begin with the space $:\mathcal{G}^1:=\mathcal{G}$ which is of important interest in the rest of the paper. The following result ensues directly from Proposition \ref{basis} and Lemma \ref{rotations}.
\begin{Cor}\label{dec}
The sequence of random variables $\big(\mathfrak{Re}\langle \mathfrak{e}_k,\cdot\rangle\big)_{k\in \mathbb{Z}^*}$ is orthogonal in the space $L^2_{\mathbb{R}}(\mathcal{H},\mathcal{B},m)$.
\end{Cor}
\begin{Rem}\label{decompositiong}
According to Corollary \ref{dec}, every function $f$ in $\mathcal{G}$ can be written in a unique way as
\begin{equation*}
f=\sum_{k\in \mathbb{Z}^*}a_k\,\mathfrak{Re}\langle \mathfrak{e}_k,\cdot \rangle
\end{equation*}
where $(a_k)_{k\in \mathbb{Z}^*}$ is a sequence of real numbers such that $\sum_{k\in \mathbb{Z}^*}a_k^2\,\sigma_k^2<+\infty$.
\end{Rem}
It is now time to look more closely at the Wick transform of a polynomial of elements of $\mathcal{G}$.
\begin{Prop}\label{hermite}
For any vector $x$ of $\mathcal{H}$ and any positive integer $k$, there exists a $k$-tuple $(\alpha_0,\dots,\alpha_{k-1})$ of real numbers such that
\begin{equation*}
:(\mathfrak{Re}\langle x,\cdot\rangle)^k:=(\mathfrak{Re}\langle x,\cdot\rangle)^k+\alpha_{k-1}\,(\mathfrak{Re}\langle x,\cdot\rangle)^{k-1}+\dots+\alpha_1\, \mathfrak{Re}\langle x,\cdot\rangle+\alpha_0.
\end{equation*}
More precisely, if $H_k$ denotes the $k^{th}$ Hermite polynomial, that is
\begin{equation*}
H_k(t)=(-1)^k e^{t^2/2}\frac{d^k}{dt^k}e^{-t^2/2},
\end{equation*}
then we have 
\begin{equation}\label{hermitevariance}
:(\mathfrak{Re}\langle x,\cdot \rangle)^k:=\sigma_x^k\, H_k\bigg(\frac{\mathfrak{Re}\langle x,\cdot\rangle}{\sigma_x}\bigg).
\end{equation}
\end{Prop}
\begin{proof}
The Wick transform of $(\mathfrak{Re}\langle x,\cdot \rangle)^k$ does not depend of the Gaussian space which contains $\mathfrak{Re}\langle x,\cdot\rangle$ (see for instance \cite{J}, Theorem $3.4$). Then there exists a monic polynomial with real coefficients $Q_k$ such that
\begin{equation*} 
:(\mathfrak{Re}\langle x,\cdot \rangle)^k:=Q_k(\mathfrak{Re}\langle x,\cdot \rangle).
\end{equation*}
By definition of the Wick transform, we know that for every integer $j$ in $\{0,\dots,k-1\}$, we have
\begin{equation*}
\int_{\mathcal{H}}:(\mathfrak{Re}\langle x,\cdot\rangle)^k:(z)(\mathfrak{Re}\langle x,z\rangle)^j\,dm(z)=\int_{\mathcal{H}}Q_k(\mathfrak{Re}\langle x,z\rangle)(\mathfrak{Re}\langle x,z\rangle)^j\,dm(z)=0.
\end{equation*}
We now use the fact that the random variable $\mathfrak{Re}\langle x,\cdot \rangle$ has Gaussian distribution $\gamma_{\sigma_x}$ and we get
\begin{equation*}
\int_{\mathbb{R}}Q_k(\sigma_x\,s)\,s^j\,e^{-s^2/2}\,ds=0
\end{equation*}
for any $j$ in $\{0,\dots,k-1\}$. The conclusion follows from the definition of Hermite polynomials $(H_\ell)_{\ell\ge 0}$ which is the sequence of monic polynomials in the weighted space $L^2(\mathbb{R},e^{-s^2/2}\,ds)$ which orthogonalizes the polynomials $t^k$ in this space. Indeed, we proved that $H_k=\frac{Q_k(\sigma_x\,\cdot)}{\sigma_x^k}$, that is $Q_k=\sigma_x^k\,H_k(\frac{\cdot}{\sigma_x})$ and we finally find that
\begin{equation*}
:(\mathfrak{Re}\langle x,\cdot \rangle)^k:=Q_k(\mathfrak{Re}\langle x,\cdot  \rangle)=\sigma_x^k\, H_k\bigg(\frac{\mathfrak{Re}\langle x,\cdot \rangle}{\sigma_x}\bigg).
\end{equation*}
\end{proof}
We now want to give a canonical decompostion of a function belonging to the space $:\mathcal{G}^k:$ by using the properties of the orthonormal basis $(e_n)_{n\in \mathbb{N}}$. To do this, the following lemma will be useful.
\begin{Lem}\label{correlation0}
For every $k$-tuples $(j_1,\dots,j_k)$ and $(\ell_1,\dots,\ell_k)$ of integers different from zero, we have
\begin{align*}
\int_{\mathcal{H}}&:\mathfrak{Re}\langle \mathfrak{e}_{j_1},\cdot\rangle\dots\mathfrak{Re}\langle \mathfrak{e}_{j_k},\cdot\rangle:(x):\mathfrak{Re}\langle \mathfrak{e}_{\ell_1},\cdot\rangle\dots\mathfrak{Re}\langle \mathfrak{e}_{\ell_k},\cdot\rangle:(x)\,dm(x)\\
&=\sum_{\tau\in \mathfrak{S}_k}\big\langle \mathfrak{Re}\langle \mathfrak{e}_{j_1},\cdot\rangle,\mathfrak{Re}\langle \mathfrak{e}_{\ell_{\tau(1)}},\cdot\rangle\big\rangle_{L^2(m)}\dots\big\langle \mathfrak{Re}\langle \mathfrak{e}_{j_k},\cdot\rangle,\mathfrak{Re}\langle \mathfrak{e}_{\ell_{\tau(k)}},\cdot\rangle\big\rangle_{L^2(m)}.
\end{align*}
\end{Lem}
\begin{proof}
This is a consequence of the fact that the map between $\big(:\mathcal{G}^k:,\langle \cdot,\cdot\rangle_{L^2(m)}\big)$ and $\big(\mathcal{G}_{\odot}^k,\langle \cdot,\cdot \rangle_{\odot}\big)$ given by \eqref{isometrie} is an isometry. So
\begin{align*}
\mathcal{I}(j_1,\dots,j_k\,&;\, \ell_1,\dots,\ell_k)\\
&:=\int_{\mathcal{H}}:\mathfrak{Re}\langle \mathfrak{e}_{j_1},\cdot\rangle\dots\mathfrak{Re}\langle \mathfrak{e}_{j_k},\cdot\rangle:(x):\mathfrak{Re}\langle \mathfrak{e}_{\ell_1},\cdot\rangle\dots\mathfrak{Re}\langle \mathfrak{e}_{\ell_k},\cdot\rangle:(x)\,dm(x)\\
&=\big\langle \text{Sym}\big(\mathfrak{Re}\langle \mathfrak{e}_{j_1},\cdot\rangle\otimes\dots\otimes\mathfrak{Re}\langle \mathfrak{e}_{j_k},\cdot\rangle\big) , \text{Sym}\big(\mathfrak{Re}\langle \mathfrak{e}_{\ell_1},\cdot\rangle\otimes\dots\otimes\mathfrak{Re}\langle \mathfrak{e}_{\ell_k},\cdot\rangle\big)\big\rangle_{\odot}\\
&=k!\,\big\langle \text{Sym}\big(\mathfrak{Re}\langle  \mathfrak{e}_{j_1},\cdot\rangle\otimes\dots\otimes\mathfrak{Re}\langle \mathfrak{e}_{j_k},\cdot\rangle\big) , \text{Sym}\big(\mathfrak{Re}\langle \mathfrak{e}_{\ell_1},\cdot\rangle\otimes\dots\otimes\mathfrak{Re}\langle \mathfrak{e}_{\ell_k},\cdot\rangle\big)\big\rangle_{\otimes}
\end{align*}
where the last equality comes from \eqref{scalarproduct}. By using the expression \eqref{sym} of the function \text{Sym}, we get by definition of the scalar product $\displaystyle \langle \cdot, \cdot\rangle_{\otimes}$ that
\begin{align*}
\mathcal{I}(j_1,\dots,j_k\,&;\, \ell_1,\dots,\ell_k)\\
&=\frac{1}{k!}\sum_{\substack{\sigma\in \mathfrak{S}_k\\ \tau \in \mathfrak{S}_k}}\big\langle \mathfrak{Re}\langle \mathfrak{e}_{j_{\sigma(1)}},\cdot\rangle\otimes\dots\otimes\mathfrak{Re}\langle \mathfrak{e}_{j_{\sigma(k)}},\cdot\rangle, \mathfrak{Re}\langle \mathfrak{e}_{\ell_{\tau(1)}},\cdot\rangle\otimes\dots\otimes\mathfrak{Re}\langle \mathfrak{e}_{\ell_{\tau(k)}},\cdot\rangle\big\rangle_{\otimes}\\
&=\frac{1}{k!}\sum_{\substack{\sigma\in \mathfrak{S}_k\\ \tau \in \mathfrak{S}_k}}\big\langle \mathfrak{Re}\langle \mathfrak{e}_{j_{\sigma(1)}},\cdot\rangle,\mathfrak{Re}\langle \mathfrak{e}_{\ell_{\tau(1)}},\cdot\rangle\big\rangle_{L^2(m)}\dots \big\langle \mathfrak{Re}\langle \mathfrak{e}_{j_{\sigma(k)}},\cdot\rangle,\mathfrak{Re}\langle \mathfrak{e}_{\ell_{\tau(k)}},\cdot\rangle\big\rangle_{L^2(m)}\\
&=\sum_{\omega\in \mathfrak{S}_k}\big\langle \mathfrak{Re}\langle \mathfrak{e}_{j_1},\cdot\rangle,\mathfrak{Re}\langle \mathfrak{e}_{\ell_{\omega(1)}},\cdot\rangle\big\rangle_{L^2(m)}\dots\big\langle \mathfrak{Re}\langle \mathfrak{e}_{j_k},\cdot\rangle,\mathfrak{Re}\langle \mathfrak{e}_{\ell_{\omega(k)}},\cdot\rangle\big\rangle_{L^2(m)}.
\end{align*}
\end{proof}
\noindent This computation allows us to find an orthogonal basis of the space $:\mathcal{G}^k:$.
\begin{Prop}\label{basis2}
For any positive integer $k$, an orthogonal basis of the space $:\mathcal{G}^k:$ is given by the family
\begin{equation}\label{basis3}
\big(:\mathfrak{Re}\langle \mathfrak{e}_{j_1},\cdot\rangle\dots\mathfrak{Re}\langle \mathfrak{e}_{j_k},\cdot\rangle:\big)_{\substack{(j_1,\dots, j_k)\in (\mathbb{Z}^*)^k\\j_1\le\dots\le j_k}}.
\end{equation}
\end{Prop}
\begin{proof}
By definition of the space $:\mathcal{G}^k:$, the only thing we need to prove is that the sequence \eqref{basis3} is orthogonal. If $j_1\le \dots \le j_k$ and $\ell_1\le \dots \le \ell_k$ are two different $k$-tuples, then the orthogonality of the sequence $(\mathfrak{Re}\langle \mathfrak{e}_p,\cdot\rangle)_{p\in \mathbb{Z}^*}$ shows that 
\begin{equation*}
\big\langle \mathfrak{Re}\langle \mathfrak{e}_{j_1},\cdot\rangle,\mathfrak{Re}\langle \mathfrak{e}_{\ell_{\tau(1)}},\cdot\rangle\big\rangle_{L^2(m)}\dots\big\langle \mathfrak{Re}\langle \mathfrak{e}_{j_k},\cdot\rangle,\mathfrak{Re}\langle \mathfrak{e}_{\ell_{\tau(k)}},\cdot\rangle\big\rangle_{L^2(m)}=0
\end{equation*}
for every $\tau \in \mathfrak{S}_k$ and the conclusion follows from Lemma \ref{correlation0}.
\end{proof}

\noindent In order to give explicitly the expansion of a function of the space $:\mathcal{G}^k:$ with respect to the basis \eqref{basis3}, we also need to determine the variance with respect to the measure $m$ of an element of this basis.
\begin{Prop}\label{variance}
For every $k$-tuple $(j_1,\dots,j_r)$ of integers different from zero such that $j_1<\dots<j_r$, and for every $k$-tuple $(\ell_1,\dots, \ell_r)$ of positive integers, we have
\begin{equation*}
\mathrm{var}_{m}\big[:(\mathfrak{Re}\langle \mathfrak{e}_{j_1},\cdot \rangle)^{\ell_1}\dots(\mathfrak{Re}\langle \mathfrak{e}_{j_r},\cdot\rangle)^{\ell_r}:\big]=\ell_1!\dots\ell_r!\,\sigma_{j_1}^{2\ell_1}\dots\sigma_{j_r}^{2\ell_r}.
\end{equation*}
\end{Prop}
\begin{proof}
Since $j_1<\dots<j_r$, the random variables $\mathfrak{Re}\langle \mathfrak{e}_{j_1},\cdot\rangle$,$\dots$, $\mathfrak{Re}\langle \mathfrak{e}_{j_r},\cdot\rangle$ are orthogonal. The fact below is a particular case of a more general statement which can be found in \cite{J} (Chapter $3$, Theorem $3.20$) and which computes the Wick transform of a product of orthogonal functions.
\begin{Fac}\label{mutually}
Under the assumptions of Proposition \ref{variance}, we have
\begin{equation*}
:(\mathfrak{Re}\langle \mathfrak{e}_{j_1},\cdot \rangle)^{\ell_1}\dots(\mathfrak{Re}\langle \mathfrak{e}_{j_r},\cdot\rangle)^{\ell_r}:=:(\mathfrak{Re}\langle \mathfrak{e}_{j_1},\cdot \rangle)^{\ell_1}:\dots:(\mathfrak{Re}\langle \mathfrak{e}_{j_r},\cdot\rangle)^{\ell_r}:.
\end{equation*}
\end{Fac}
\noindent We know that the random variables $\mathfrak{Re}\langle \mathfrak{e}_i,\cdot \rangle$ are independent since they are orthogonal real Gaussian variables. Since the Wick transform of $(\mathfrak{Re}\langle x,\cdot\rangle)^p$ is a measurable function in the variable $\mathfrak{Re}\langle x,\cdot\rangle$ according to Proposition \ref{hermite}, the random variables $:(\mathfrak{Re}\langle \mathfrak{e}_{j_1},\cdot \rangle)^{\ell_1}:,\dots,:(\mathfrak{Re}\langle \mathfrak{e}_{j_r},\cdot \rangle)^{\ell_r}:$ are independent. Then we get
\begin{eqnarray*}
\text{var}_{m}\big[:(\mathfrak{Re}\langle \mathfrak{e}_{j_1},\cdot\rangle)^{\ell_1}\dots(\mathfrak{Re}\langle \mathfrak{e}_{j_r},\cdot\rangle)^{\ell_r}:\big]=\prod_{t=1}^{r}\text{var}_{m}\big[:(\mathfrak{Re}\langle \mathfrak{e}_{j_t},\cdot\rangle)^{\ell_t}:\big]=\prod_{t=1}^{r}\ell_t!\,\sigma_{j_t}^{2\ell_t},
\end{eqnarray*}
where the computation of each variance in the last equality follows directly to Lemma \ref{correlation0}.
\end{proof}
\noindent According to Proposition \ref{basis2} and Proposition \ref{variance}, we have the following decomposition of a function $f_k$ in $:\mathcal{G}^k:$.
\begin{Prop}
 A function $f_k$ which belongs to $:\mathcal{G}^k:$ can be written in a unique way as
\begin{equation}\label{developpement}
f_k=\sum_{\substack{(j_1,\dots, j_k)\in (\mathbb{Z}^*)^k\\j_1\le \dots \le j_k}}a_{j_1,\dots,j_{k}}^{(k)}:\mathfrak{Re}\langle \mathfrak{e}_{j_1},\cdot\rangle\dots \mathfrak{Re}\langle \mathfrak{e}_{j_{k}},\cdot\rangle :
\end{equation}
where the real numbers $a_{j_1,\dots,j_k}^{(k)}$ are given by the formula
\begin{equation}\label{expressioncoefficient}
a_{j_1,\dots, j_k}^{(k)}=\frac{\big\langle f_k,:\mathfrak{Re}\langle \mathfrak{e}_{j_1},\cdot\rangle\dots\mathfrak{Re}\langle \mathfrak{e}_{j_k},\cdot\rangle:\big\rangle_{L^2(m)}}{\mathrm{var}_m\big[:\mathfrak{Re}\langle \mathfrak{e}_{j_1},\cdot\rangle\dots \mathfrak{Re}\langle \mathfrak{e}_{j_k},\cdot\rangle:\big]}
\end{equation}
and satisfy the condition
\begin{equation}\label{condition}
\sum_{j_1\le\dots\le j_k}\big\vert a_{j_1,\dots,j_k}^{(k)}\big\vert^2\sigma_{j_1}^2\dots\sigma_{j_k}^2<+\infty.
\end{equation}
\end{Prop}
\begin{proof}
The decomposition directly follows from the orthogonality of the family 
$$
(:\mathfrak{Re}\langle \mathfrak{e}_{j_1},\cdot\rangle\dots \mathfrak{Re}\langle \mathfrak{e}_{j_k},\cdot \rangle:)_{j_1\le\dots\le j_k}
$$ 
and condition \eqref{condition} is a consequence of Proposition \ref{variance}.
\end{proof}
So, according to \eqref{expansion} and \eqref{developpement}, the computation of the correlations of two functions in $L^2_\mathbb{R}(\mathcal{H},\mathcal{B},m)$ can be reduced to the computation of the correlations between the Wick transforms of two homogeneous polynomials.
\subsection{Orthogonality and computation of the correlations}
It is now time to compute the correlations of two functions living in two different spaces $:\mathcal{G}^k:$ and then of two functions belonging to the same space $:\mathcal{G}^k:$.

\begin{Prop}\label{orthogonal}
For any vectors $x,y$ of $\mathcal{H}$ and any nonnegative integers $k, \ell$ such that $k\ne \ell$, we have
\begin{equation*}
\int_{\mathcal{H}}:(\mathfrak{Re}\langle x,\cdot\rangle)^k:(T^nz):(\mathfrak{Re}\langle y,\cdot \rangle)^\ell:(z)\, dm(z)=0
\end{equation*}
for any nonnegative integer $n$.
\end{Prop}
\begin{proof}
With the notations of Proposition \ref{hermite}, $:(\mathfrak{Re}\langle x,\cdot\rangle)^k:=Q_k(\mathfrak{Re}\langle x,\cdot \rangle)$. Then, we have
\begin{eqnarray*}
:(\mathfrak{Re}\langle x,\cdot\rangle)^k:(T^nz)=Q_k(\mathfrak{Re}\langle x,\cdot \rangle)(T^nz)=Q_k(\mathfrak{Re}\langle T^{*n}x,\cdot\rangle)(z)
\end{eqnarray*}
that is
\begin{equation}\label{interversion}
:(\mathfrak{Re}\langle x,\cdot\rangle)^k:\circ\ T^n=:(\mathfrak{Re}\langle T^{*n}x,\cdot \rangle)^k:
\end{equation}
by definition of the polynomial $Q_k$. The conclusion follows from the orthogonality of the spaces $:\mathcal{G}^j:$. 
\end{proof}
\noindent From this we can easily deduce the general case.
\begin{Cor}\label{nul}
For every functions $f_k$ in $:\mathcal{G}^k:$ and $g_\ell$ in $:\mathcal{G}^{\ell}:$ such that $k\ne \ell$, we have $\mathcal{I}_n(f_k,g_\ell)=0$ for any nonnegative integer $n$.
\end{Cor}
\begin{proof}
According to decomposition \eqref{developpement}, it suffices to check that for any tuples $(j_1,\dots,j_k)$ and $(m_1,\dots,m_\ell)$ of integers different from zero such that $j_1\le \dots\le j_k$ and $m_1\le\dots\le m_\ell$, we have
\begin{equation*}
\int_{\mathcal{H}}:\mathfrak{Re}\langle \mathfrak{e}_{j_1},\cdot\rangle\dots\mathfrak{Re}\langle \mathfrak{e}_{j_k},\cdot\rangle:(T^nx):\mathfrak{Re}\langle \mathfrak{e}_{m_1},\cdot\rangle\dots\mathfrak{Re}\langle \mathfrak{e}_{m_\ell},\cdot\rangle:(x)\,dm(x)=0
\end{equation*}
for any nonnegative integer $n$. But the orthogonality follows from Proposition \ref{orthogonal} and from the multilinear identity
\begin{equation}\label{multilineaire}
p!\,\prod_{j=1}^{p}x_j=\sum_{r=1}^{p}(-1)^{p-r}\sum_{j_1<\dots<j_r}(x_{j_1}+\dots+x_{j_r})^p
\end{equation}
which holds true for any elements $x_1,\dots, x_p$ of a unit commutative ring (see for instance \cite{P}, Chapter $1$). Indeed, one can rewrite each product $\mathfrak{Re}\langle x_1,\cdot\rangle\dots\mathfrak{Re}\langle x_k,\cdot\rangle$ as
\begin{equation*}
\frac{1}{k!}\sum_{r=1}^k(-1)^{k-r}\sum_{j_1<\dots< j_r}\big(\mathfrak{Re}\langle x_{j_1}+\dots +x_{j_r},\cdot \rangle\big)^k
\end{equation*}
and the conclusion follows from the linearity of the Wick transform on the space $\mathcal{G}^k$ and Proposition \ref{orthogonal}.
\end{proof}

\begin{Rem}\label{interversion1}
By using the multilinear identity \eqref{multilineaire}, the linearity of the Wick transform on the space $\mathcal{G}^k$ and \eqref{interversion}, we can also prove that for every $k$-tuple $(j_1,\dots,j_k)$ of integers different from zero,
\begin{equation*}
:\mathfrak{Re}\langle \mathfrak{e}_{j_1},\cdot\rangle \dots \mathfrak{Re}\langle \mathfrak{e}_{j_k},\cdot\rangle:\circ\, T^n=:\mathfrak{Re}\langle T^{*n}\mathfrak{e}_{j_1},\cdot\rangle \dots \mathfrak{Re}\langle T^{*n}\mathfrak{e}_{j_k},\cdot\rangle:
\end{equation*}
which will be useful in the rest of the paper.
\end{Rem}
A consequence of Corollary \ref{nul} is that it suffices to know the values of the correlations between two functions which belong to the same space $:\mathcal{G}^k:$. Indeed, if we consider two functions $f$ and $g$ in $L^2_{\mathbb{R}}(\mathcal{H},\mathcal{B},m)$ with Wiener chaos decompositions $f=\sum_{k\ge 0} f_k$ and $g=\sum_{\ell\ge 0}g_\ell$ as in \eqref{expansion} (with $f_k:=\mathcal{P}_{:\mathcal{G}^k:}f$ and $g_\ell:=\mathcal{P}_{:\mathcal{G}^\ell:}g$), then the $n^{th}$ correlation between $f$ and $g$ becomes
\begin{equation*}
\mathcal{I}_n(f,g)=\sum_{k\ge 1}\mathcal{I}_n(f_k,g_k),
\end{equation*}
where the sum begins at $1$ since $\int_\mathcal{H} f\,dm=f_0$ and $\int_\mathcal{H} g\,dm=g_0$. In fact, it suffices by decomposition \eqref{developpement} to compute the correlations when the functions are Wick transforms of homogeneous polynomials of the same degree.

\begin{Prop}\label{correlationpoly}
For every $k$-tuples $(j_1,\dots,j_k)$ and $(\ell_1,\dots,\ell_k)$ of integers different from zero, we have
\begin{align*}
\int_{\mathcal{H}}:\mathfrak{Re}\langle \mathfrak{e}_{j_1},\cdot\rangle\dots\mathfrak{Re}\langle \mathfrak{e}_{j_k},\cdot\rangle: & (T^nx):\mathfrak{Re}\langle \mathfrak{e}_{\ell_1},\cdot\rangle\dots\mathfrak{Re}\langle \mathfrak{e}_{\ell_k},\cdot\rangle:(x)\,dm(x)\\
&=\sigma_{\ell_1}^2\dots\sigma_{\ell_k}^2\sum_{\tau\in \mathfrak{S}_k}\mathfrak{Re}\langle \mathfrak{e}_{j_1},T^n \mathfrak{e}_{\ell_{\tau(1)}}\rangle\dots\mathfrak{Re}\langle \mathfrak{e}_{j_k},T^n \mathfrak{e}_{\ell_{\tau(k)}}\rangle
\end{align*}
for any nonnegative integer $n$.
\end{Prop}
\begin{proof}
The beginning of the proof is the same as that of the proof of Lemma \ref{correlation0}.
We use the isometry between $\big(:\mathcal{G}^k:,\langle \cdot,\cdot\rangle_{L^2(m)}\big)$ and $\big(\mathcal{G}_{\odot}^k,\langle \cdot,\cdot \rangle_{\odot}\big)$, which is given by \eqref{isometrie}, and Remark \ref{interversion1}:
\begin{align*}
\mathcal{I}(j_1,\dots,j_k\,&;\, \ell_1,\dots,\ell_k)\\
&:=\int_{\mathcal{H}}:\mathfrak{Re}\langle \mathfrak{e}_{j_1},\cdot\rangle\dots\mathfrak{Re}\langle \mathfrak{e}_{j_k},\cdot\rangle:(T^nx):\mathfrak{Re}\langle \mathfrak{e}_{\ell_1},\cdot\rangle\dots\mathfrak{Re}\langle \mathfrak{e}_{\ell_k},\cdot\rangle:(x)\,dm(x)\\
&=\int_{\mathcal{H}}:\mathfrak{Re}\langle T^{*n}\mathfrak{e}_{j_1},\cdot\rangle\dots\mathfrak{Re}\langle T^{*n}\mathfrak{e}_{j_k},\cdot\rangle:(x):\mathfrak{Re}\langle \mathfrak{e}_{\ell_1},\cdot\rangle\dots\mathfrak{Re}\langle \mathfrak{e}_{\ell_k},\cdot\rangle:(x)\,dm(x).
\end{align*}
By replacing $\mathfrak{e}_{j_1},\dots,\mathfrak{e}_{j_k}$ by $T^{*n}\mathfrak{e}_{j_1},\dots,T^{*n}\mathfrak{e}_{j_k}$ in the proof of Lemma \ref{correlation0}, we find that
\begin{align*}
\mathcal{I}(j_1,\dots,j_k\,&;\, \ell_1,\dots,\ell_k)\\
&=\sum_{\tau\in \mathfrak{S}_k}\big\langle \mathfrak{Re}\langle T^{*n}\mathfrak{e}_{j_1},\cdot\rangle,\mathfrak{Re}\langle \mathfrak{e}_{\ell_{\tau(1)}},\cdot\rangle\big\rangle_{L^2(m)}\dots\big\langle \mathfrak{Re}\langle T^{*n}\mathfrak{e}_{j_k},\cdot\rangle,\mathfrak{Re}\langle \mathfrak{e}_{\ell_{\tau(k)}},\cdot\rangle\big\rangle_{L^2(m)}
\end{align*}
and the fact above yields the desired conclusion.
\begin{Fac}\label{transformation}
For every integers $i$ and $j$ different from zero, we have
\begin{equation*}
\big\langle\mathfrak{Re}\langle T^{*n}\mathfrak{e}_i,\cdot \rangle,\mathfrak{Re}\langle \mathfrak{e}_j,\cdot\rangle\big\rangle_{L^2(m)}=\sigma_j^2\, \mathfrak{Re}\langle \mathfrak{e}_i,T^n \mathfrak{e}_j\rangle.
\end{equation*}
\end{Fac}

\begin{proof}
We already know from \eqref{rotation1} that 
\begin{equation*}
\big\langle\mathfrak{Re}\langle T^{*n}\mathfrak{e}_i,\cdot \rangle,\mathfrak{Re}\langle \mathfrak{e}_j,\cdot\rangle\big\rangle_{L^2(m)}=\frac{1}{2}\mathfrak{Re}\langle RT^{*n}\mathfrak{e}_i, \mathfrak{e}_j\rangle.
\end{equation*}
The covariance operator is self-adjoint and $\mathfrak{e}_j$ is an eigenvector of $R$ corresponding to the eigenvalue $2\,\sigma_j^2$. Then the result follows readily.
\end{proof}
\end{proof}
By using Proposition \ref{correlationpoly}, we are now able to compute the correlations between two arbitrary functions in $L^2_{\mathbb{R}}(\mathcal{H},\mathcal{B},m)$. Since we have no uniform rate of decrease in $L^2_{\mathbb{R}}(\mathcal{H},\mathcal{B},m)$ according to Theorem \ref{non}, we need to make some assumptions of regularity on our functions. In the next section, we study some of these regularity assumptions and we show that the correlations decrease to zero with speed $n^{-\alpha}$ when we consider square-integrable real-valued functions which satisfy these conditions.

\section{Speed of mixing}
We consider here a bounded linear operator $T$ on $\mathcal{H}$ whose eigenvectors associated to unimodular eigenvalues are parametrized by a $\mu$-spanning $\mathbb{T}$-eigenvector field $E$ which is assumed to be $\alpha$-H\"olderian as in Assumption \ref{assumption}.
We already know from Section $3$ that there is no hope to find a uniform speed of mixing in the whole space $L^2(\mathcal{H},\mathcal{B},m)$. Then a natural problem is to find some classes of functions of $L^2_\mathbb{R}(\mathcal{H},\mathcal{B},m)$ for which the correlations decrease to zero with some speed of mixing. We will exhibit classes of functions for which the speed of mixing is exactly $n^{-\alpha}$.
\subsection{Speed of mixing for functions in a finite number of variables}
Our first result, which requires no regularity on the functions, gives a speed of mixing by considering functions of a finite number of variables.
\begin{Theo}
Let $N$ be a positive integer and 
\begin{equation*}
f=\phi\big(\mathfrak{Re}\langle \mathfrak{e}_{-N},\cdot \rangle,\dots,\mathfrak{Re}\langle \mathfrak{e}_N,\cdot\rangle\big),\ g=\psi\big(\mathfrak{Re}\langle \mathfrak{e}_{-N},\cdot \rangle,\dots,\mathfrak{Re}\langle \mathfrak{e}_N,\cdot\rangle\big)
\end{equation*}
be two real-valued functions which belong to $L^2_\mathbb{R}(\mathcal{H},\mathcal{B},m)$, where $\phi,\psi : \mathbb{R}^{2N}\longrightarrow \mathbb{R}$. Then there exists a positive constant $C_N$, which only depends on $N$ and $\sigma_1,\dots,\sigma_N$, such that
\begin{equation*}
\big\vert\mathcal{I}_n(f,g)\big\vert\le \frac{C_N}{n^{\alpha}}\,\vert\vert f \vert\vert_{L^2(m)}\,\vert\vert g \vert\vert_{L^2(m)}
\end{equation*}
for any positive integer $n$.
\end{Theo}
\begin{proof}
In this proof, we deal with the Gaussian space generated by the random variables\newline 
$\mathfrak{Re}\langle \mathfrak{e}_\ell,\cdot\rangle$ where $\ell\in \{-N,\dots, N\}\setminus\{0\}$, that is
\begin{equation*}
\mathcal{G}_N:=\mathrm{span}\big[\mathfrak{Re}\langle \mathfrak{e}_{-N},\cdot\rangle,\dots, \mathfrak{Re}\langle \mathfrak{e}_N,\cdot\rangle\big].
\end{equation*}
We expand the functions $f$and $g$ as in \eqref{developpement} by using this Gaussian space and we find that
\begin{equation*}
f=\sum_{(i_{-N},\dots,i_N)\in (\mathbb{Z}_+)^{2N}}f_{i_{-N},\dots,i_N}:(\mathfrak{Re}\langle \mathfrak{e}_{-N},\cdot\rangle)^{i_{-N}}\dots (\mathfrak{Re}\langle \mathfrak{e}_N,\cdot\rangle)^{i_N}:
\end{equation*}
where
\begin{equation*}
\vert\vert f \vert\vert_{L^2(m)}^2=\sum_{(i_{-N},\dots,i_N)\in (\mathbb{Z}_+)^{2N}}\vert f_{i_{-N},\dots, i_N}\vert^2\, i_{-N}!\,\sigma_{-N}^{2i_{-N}}\dots i_N!\,\sigma_{N}^{2i_N}<+\infty,
\end{equation*}
and
\begin{equation*}
g=\sum_{(j_{-N},\dots,j_N)\in (\mathbb{Z}_+)^{2N}}g_{j_{-N},\dots, j_N}:(\mathfrak{Re}\langle \mathfrak{e}_{-N},\cdot\rangle)^{j_{-N}}\dots (\mathfrak{Re}\langle \mathfrak{e}_N,\cdot\rangle)^{j_N}:
\end{equation*}
where
\begin{equation*}
\vert\vert g \vert\vert_{L^2(m)}^2=\sum_{(j_1,\dots,j_N)\in (\mathbb{Z}_+)^{2N}}\vert g_{j_{-N},\dots,j_N}\vert^2\, j_{-N}!\,\sigma_{-N}^{2j_{-N}}\dots j_N!\,\sigma_{N}^{2j_N}<+\infty.
\end{equation*}
Then it follows from Corollary \ref{nul} that
\begin{align*}
\mathcal{I}_n(f,g)=&\sum_{\ell=1}^{+\infty}\sum_{\substack{(i_{-N},\dots,i_N)\in (\mathbb{Z}+)^{2N}\\ i_{-N}+\dots +i_N=\ell}}\sum_{\substack{(j_{-N},\dots,j_N)\in (\mathbb{Z}+)^{2N}\\ j_{-N}+\dots +j_N=\ell}}\,f_{i_{-N},\dots,i_N}\,g_{j_{-N},\dots, j_N} \\
&\times\mathcal{I}_n\big(:(\mathfrak{Re}\langle \mathfrak{e}_{-N},\cdot\rangle)^{i_{-N}}\dots (\mathfrak{Re}\langle \mathfrak{e}_N,\cdot\rangle)^{i_N}:,:(\mathfrak{Re}\langle \mathfrak{e}_{-N},\cdot\rangle)^{j_{-N}}\dots (\mathfrak{Re}\langle \mathfrak{e}_N,\cdot\rangle)^{j_N}:\big).
\end{align*}
We now expand the correlation 
$$
\mathcal{I}_n(:(\mathfrak{Re}\langle \mathfrak{e}_{-N},\cdot\rangle)^{i_{-N}}\dots (\mathfrak{Re}\langle \mathfrak{e}_N,\cdot\rangle)^{i_N}:,:(\mathfrak{Re}\langle \mathfrak{e}_{-N},\cdot\rangle)^{j_{-N}}\dots (\mathfrak{Re}\langle \mathfrak{e}_N,\cdot\rangle)^{j_N}:)
$$ 
by using the proof of Proposition \ref{correlationpoly} (before Fact \ref{transformation}). The triangle inequality and the proof of Proposition \ref{functional} show that the absolute value of this correlation is less than $\ell!\,\frac{C(E,\alpha)^\ell}{n^{\ell\alpha}}$. There exists a positive constant $C$  such that $C^{-\ell}\le \sigma_{-N}^{i_{-N}}\dots\sigma_N^{i_N}$ for any nonnegative integers $i_{-N},\dots,i_N$ such that $i_{-N}+\dots+i_N=\ell$ and then
\begin{align*}
\big\vert\mathcal{I}_n(f,g)\big\vert\le\sum_{\ell=1}^{+\infty}\frac{\ell!\,[C(E,\alpha)C^2]^{\ell}}{n^{\ell\alpha}}&\sum_{\substack{(i_{-N},\dots,i_N)\in (\mathbb{Z}+)^{2N}\\ i_{-N}+\dots +i_N=\ell}}\frac{\vert f_{i_{-N},\dots,i_N}\vert\,\sqrt{i_{-N}!\dots i_N!}\,\sigma_{-N}^{i_{-N}}\dots \sigma_N^{i_N}}{\sqrt{i_{-N}!\dots i_N!}}\\
&\times \sum_{\substack{(j_{-N},\dots,j_N)\in (\mathbb{Z}+)^{2N}\\ j_{-N}+\dots +j_N=\ell}}\frac{\vert g_{j_{-N},\dots,j_N}\vert\,\sqrt{j_{-N}!\dots j_N!}\,\sigma_{-N}^{j_{-N}}\dots \sigma_N^{j_N}}{\sqrt{j_{-N}!\dots j_N!}}\cdot
\end{align*}
By applying the Cauchy-Schwarz inequality, we find that
$$
\sum_{\substack{(i_{-N},\dots,i_N)\in (\mathbb{Z}+)^{2N}\\ i_{-N}+\dots +i_N=\ell}}\frac{\vert f_{i_{-N},\dots,i_N}\vert\,\sqrt{i_{-N}!\dots i_N!}\,\sigma_{-N}^{i_{-N}}\dots \sigma_N^{i_N}}{\sqrt{i_{-N}!\dots i_N!}}
$$
is less or equal than
$$
\bigg(\sum_{\substack{(i_{-N},\dots,i_N)\in (\mathbb{Z}+)^{2N}\\ i_{-N}+\dots +i_N=\ell}}\frac{1}{i_{-N}!\dots i_N!}\bigg)^{1/2}\vert\vert f \vert\vert_{L^2(m)}
=\sqrt{\frac{(2N)^{\ell}}{\ell!}}\,\vert\vert f \vert\vert_{L^2(m)}.
$$
If we do the same thing with the sum corresponding to $g$, it follows that
\begin{equation*}
\big\vert\mathcal{I}_n(f,g)\big\vert\le\bigg(\sum_{\ell=1}^{+\infty}\frac{[2C(E,\alpha)C^2 N]^{\ell}}{n^{\ell\alpha}}\bigg)\,\vert\vert f \vert\vert_{L^2(m)}\,\vert\vert g \vert\vert_{L^2(m)}. 
\end{equation*}
Finally, we can find a constant $C_N>0$ such that $\big\vert\mathcal{I}_n(f,g)\big\vert\le \frac{C_N}{n^{\alpha}}\vert\vert f \vert\vert_{L^2(m)}\,\vert\vert g \vert\vert_{L^2(m)}$ for any positive integer $n$, which concludes the proof.
\end{proof}

We now deal with more general functions on which we will have to impose some condition of smoothness in order to still have a speed of mixing. More precisely, in the estimation of $\mathcal{I}_n(f,g)$, we will consider a large class of infinitely differentiable functions for $f$ which satisfy some integrability condition and for $g$, we will deal with a more restrictive class of functions which contains the class of polynomial functions. It is the object of the next section to define these classes of functions.

\subsection{Regularity and Fourier coefficients}
In this section, we consider an infinitely differentiable function $f:\mathcal{H}\longrightarrow \mathbb{R}$ and we make the assumption that
\begin{equation}\label{intfinite}
\int_\mathcal{H} \big\vert\big\vert D^k f(x) \big\vert\big\vert \, dm(x)<+\infty\ \ \ \ \ \ \ \mathrm{for\ any\ nonnegative\ integer\ }k,
\end{equation}
where the $k$-linear form $D^kf$ is the $k^{th}$ derivative of the function $f$. Recall that the norm $\vert\vert \cdot\vert\vert$ of a bounded $k$-linear form $\phi$ is defined by
\begin{equation}\label{normemultilineaire}
\vert\vert \phi \vert\vert=\sup_{\vert\vert x_1 \vert\vert\le 1,\dots, \vert\vert x_k\vert\vert\le 1}\vert\phi(x_1,\dots, x_k)\vert.
\end{equation}
We write our function as in \eqref{expansion}, that is $f=\sum_{k\ge 0}f_k$, where $f_k:=\mathcal{P}_{:\mathcal{G}^k:}f$ is given by the expansion \eqref{developpement}. Our first task is to get informations on the coefficients $a_{j_1,\dots,j_k}^{(k)}$ of $f_k$ in this expansion.
\begin{Lem}\label{coefficient}
For any positive integer $r$ and any $r$-tuples $(j_1,\dots,j_r)$ and $(\ell_1,\dots,\ell_r)$ of integers such that $j_1<\dots<j_r$ and $\ell_1+\dots+\ell_r=k$, we have
\begin{equation*}
a_{\underbrace{j_1,\dots,j_1}_{\ell_1\ \mathrm{times}},\dots,\underbrace{j_r,\dots,j_r}_{\ell_r\ \mathrm{times}}}^{(k)}=\frac{1}{\ell_1!\dots\ell_r!}\int_{\mathcal{H}}D^kf(x)(\underbrace{\mathfrak{e}_{j_1},\dots,\mathfrak{e}_{j_1}}_{\ell_1\ \mathrm{times}},\dots,\underbrace{\mathfrak{e}_{j_r},\dots, \mathfrak{e}_{j_r}}_{\ell_r\ \mathrm{times}})\,dm(x).
\end{equation*}
\end{Lem}

\begin{proof}
Recall that the coefficient $a_{j_1,\dots,j_k}^{(k)}$ is given in \eqref{expressioncoefficient} by
\begin{equation*}
a_{\underbrace{j_1,\dots,j_1}_{\ell_1\ \text{times}},\dots,\underbrace{j_r,\dots,j_r}_{\ell_r\ \text{times}}}^{(k)}=\frac{\big\langle f, :(\mathfrak{Re}\langle \mathfrak{e}_{j_1},\cdot\rangle)^{\ell_1}\dots (\mathfrak{Re}\langle \mathfrak{e}_{j_r},\cdot\rangle)^{\ell_r}:\big\rangle_{L^2(m)}}{\text{var}_{m}\big[:(\mathfrak{Re}\langle \mathfrak{e}_{j_1},\cdot\rangle)^{\ell_1}\dots (\mathfrak{Re}\langle \mathfrak{e}_{j_r},\cdot\rangle)^{\ell_r}:\big]}
\end{equation*}
and we already know from Proposition \ref{variance} that
\begin{equation*}
\text{var}_{m}\big[:(\mathfrak{Re}\langle \mathfrak{e}_{j_1},\cdot\rangle)^{\ell_1}\dots (\mathfrak{Re}\langle \mathfrak{e}_{j_r},\cdot\rangle)^{\ell_r}:\big]=\ell_1!\dots\ell_r!\,\sigma_{j_1}^{2\ell_1}\dots\sigma_{j_r}^{2\ell_r}.
\end{equation*}
In this proof, we assume that all the integers $j_1,\dots, j_r$ are positive, that is to say $0<j_1<\dots<j_r$. The computations are the same when one of the integer is negative. We now compute $\langle f, :(\mathfrak{Re}\langle e_{j_1},\cdot\rangle)^{\ell_1}\dots (\mathfrak{Re}\langle e_{j_r},\cdot\rangle)^{\ell_r}:\rangle_{L^2(m)}$ by using integrations by parts with respect to the Gaussian measure $m$ on $\mathcal{H}$. To do this, we decompose the Hilbert space $\mathcal{H}$ as follows: 
\begin{equation*}
\mathcal{H}=\bigoplus_{t=1}^{r+1}\mathcal{H}_t
\end{equation*}
where $\mathcal{H}_1=\text{span}[e_j\, ;\,1\le j\le j_1]$, $\mathcal{H}_2=\text{span}[e_j\, ;\, j_1<j\le j_2]$,..., $\mathcal{H}_r=\text{span}[e_j\, ;\, j_{r-1}<j\le j_r]$ and $\mathcal{H}_{r+1}=\overline{\text{span}}^{\mathcal{H}}[e_j\, ;\, j>j_r]$. Hence, we can write our Gaussian measure $m$ as a finite product of Gaussian measures, that is
\begin{equation*}
m=\bigotimes_{t=1}^{r+1}m_t
\end{equation*}
where $m_1$ is the distribution of the Gaussian vector $(\langle e_1,\cdot \rangle,\dots,\langle e_{j_1},\cdot\rangle)$, $m_2$ is the distribution of\newline 
$(\langle e_{j_1+1},\cdot\rangle,\dots,\langle e_{j_2},\cdot\rangle)$,..., $m_{r+1}$ is the distribution of $(\langle e_{i},\cdot\rangle)_{i>j_r}$. Then we can write that the scalar product $\big\langle f, :(\mathfrak{Re}\langle e_{j_1},\cdot\rangle)^{\ell_1}\dots (\mathfrak{Re}\langle e_{j_{r}},\cdot\rangle)^{\ell_{r}}:\big\rangle_{L^2(m)}$ is equal to
\begin{align*}
&\int_{\mathcal{H}}f(x):(\mathfrak{Re}\langle e_{j_1},\cdot\rangle)^{\ell_1}\dots (\mathfrak{Re}\langle e_{j_{r}},\cdot\rangle)^{\ell_{r}}:(x)\,dm(x)\\
&=\int_{\mathcal{H}_1}\dots\int_{\mathcal{H}_{r+1}}f(x):(\mathfrak{Re}\langle e_{j_1},\cdot\rangle)^{\ell_1}\dots (\mathfrak{Re}\langle e_{j_{r}},\cdot\rangle)^{\ell_{r}}:(x)\,dm_1(x_1)\dots dm_{r+1}(x_{r+1})\\
&=\int_{\mathcal{H}_1}\dots\int_{\mathcal{H}_{r+1}}f(x):(\mathfrak{Re}\langle e_{j_1},\cdot\rangle)^{\ell_1}:(x_1)\dots :(\mathfrak{Re}\langle e_{j_{r}},\cdot\rangle)^{\ell_{r}}:(x_{r})\,dm_1(x_1)\dots dm_{r+1}(x_{r+1}),
\end{align*}
where the last equality comes from Fact \ref{mutually}. We now fix $(x_2,\dots,x_{r+1})$ in $\mathcal{H}_2\times\dots\times \mathcal{H}_{r+1}$ and we compute
\begin{equation*}
\alpha_{j_1,\ell_1}:=\int_{\mathcal{H}_1}f(x):(\mathfrak{Re}\langle e_{j_1},\cdot\rangle)^{\ell_1}:(x_1)\,dm_1(x_1)
=\int_{\mathcal{H}_1}f(x_1+\zeta):(\mathfrak{Re}\langle e_{j_1},\cdot\rangle)^{\ell_1}:(x_1)\,dm_1(x_1)
\end{equation*}
where $\zeta:=x_2+\dots+x_{r+1}$. But we know from \eqref{hermitevariance} that $:(\mathfrak{Re}\langle e_{j_1},\cdot\rangle)^{\ell_1}:=\sigma_{j_1}^{\ell_1}\,H_{\ell_1}\Big(\frac{\mathfrak{Re}\langle e_{j_1},\cdot\rangle}{\sigma_{j_1}}\Big)$. Then, since $m_1$ is the distribution of the Gaussian vector $(\langle e_1,\cdot\rangle,\dots,\langle e_{j_1},\cdot\rangle)$, we get
\begin{align*}
\alpha_{j_1,\ell_1}=\sigma_{j_1}^{\ell_1}\int_{\mathbb{R}^{2j_1}}f\bigg(\sum_{s=1}^{j_1-1}t_se_s+i\sum_{s=1}^{j_1}t_s^{'}e_s+&t_{j_1}e_{j_1}+\zeta\bigg)H_{\ell_1}\Big(\frac{t_{j_1}}{\sigma_{j_1}}\Big)\\
&\times d(\gamma_{\sigma_1}\otimes\gamma_{\sigma_1})(t_1,t_1^{'})\dots d(\gamma_{\sigma_{j_1}}\otimes\gamma_{\sigma_{j_1}})(t_{j_1},t_{j_1}^{'}).
\end{align*}
We now fix $(t_1,\dots,t_{j_1-1},t_1^{'},\dots,t_{j_1-1}^{'},t_{j_1}^{'})$ in $\mathbb{R}^{2j_1-1}$ and we put $\displaystyle \omega:=\sum_{s=1}^{j_1-1}t_se_s+i\sum_{s=1}^{j_1}t_s^{'}e_s+\zeta$. The only integral we really need to compute is
\begin{align*}
\mathcal{I}:=\int_{\mathbb{R}}f(te_{j_1}+\omega)H_{\ell_1}\Big(\frac{t}{\sigma_{j_1}}\Big)\,d\gamma_{j_1}(t)&=\int_{\mathbb{R}}f(te_{j_1}+\omega)H_{\ell_1}\Big(\frac{t}{\sigma_{j_1}}\Big)e^{-t^2/2\sigma_{j_1}^2}\,\frac{dt}{\sigma_{j_1}\sqrt{2\pi}}\\
&=\int_{\mathbb{R}}f(\sigma_{j_1}se_{j_1}+\omega)H_{\ell_1}(s)e^{-s^2/2}\frac{ds}{\sqrt{2\pi}}\cdot
\end{align*}
By considering the expression of the Hermite polynomial $H_{\ell_1}$ which is given in Proposition \ref{hermite} and integrating by parts $\ell_1$ times, we have
\begin{align*}
\mathcal{I}=(-1)^{\ell_1}\int_{\mathbb{R}}f(\sigma_{j_1}se_{j_1}+\omega)\frac{d^{\ell_1}}{ds^{\ell_1}}&e^{-s^2/2}\,\frac{ds}{\sqrt{2\pi}}\\
&=\sigma_{j_1}^{\ell_1}\int_{\mathbb{R}}D^{\ell_1}f(\sigma_{j_1}se_{j_1}+\omega)(\underbrace{e_{j_1},\dots,e_{j_1}}_{\ell_1\ \text{times}})e^{-s^2/2}\,\frac{ds}{\sqrt{2\pi}}\\
&=\sigma_{j_1}^{\ell_1}\int_{\mathbb{R}}D^{\ell_1}f(te_{j_1}+\omega)(\underbrace{e_{j_1},\dots,e_{j_1}}_{\ell_1\ \text{times}})e^{-t^2/2\sigma_{j_1}^2}\,\frac{dt}{\sigma_{j_1}\sqrt{2\pi}}\cdot
\end{align*} 
So we conclude that
\begin{align*}
\alpha_{j_1,\ell_1}=\sigma_{j_1}^{2\ell_1}\int_{\mathbb{R}^{2j_1}}D^{\ell_1}f\bigg(\sum_{s=1}^{j_1-1}t_se_s +i\sum_{s=1}^{j_1}&t_s^{'}e_s +t_{j_1}e_{j_1}+\zeta\bigg)(e_{j_1},\dots,e_{j_1})\\
&\times d(\gamma_{\sigma_1}\otimes\gamma_{\sigma_1})(t_1,t_1^{'})\dots d(\gamma_{\sigma_{j_1}}\otimes\gamma_{\sigma_{j_1}})(t_{j_1},t_{j_1}^{'}),
\end{align*}
and then the scalar product $\big\langle f, :(\mathfrak{Re}\langle e_{j_1},\cdot\rangle)^{\ell_1}\dots (\mathfrak{Re}\langle  e_{j_{r}},\cdot\rangle)^{\ell_{r}}:\big\rangle_{L^2(m)}$ is equal to
\begin{align*}
\sigma_{j_1}^{2\ell_1}\int_{\mathcal{H}_1}\dots &\int_{\mathcal{H}_{r+1}} D^{\ell_1}f(x)(e_{j_1},\dots, e_{j_1})\\
& \times :(\mathfrak{Re}\langle e_{j_2},\cdot\rangle)^{\ell_2}:(x_2)\dots:(\mathfrak{Re}\langle e_{j_{r}},\cdot\rangle)^{\ell_{r}}:(x_{r})\,dm_1(x_1)\dots dm_{r+1}(x_{r+1}).
\end{align*}
Secondly, we do the same thing for 
\begin{equation*}
\int_{\mathcal{H}_2}D^{\ell_1}f(x)(e_{j_1},\dots, e_{j_1}):(\mathfrak{Re}\langle e_{j_2},\cdot\rangle)^{\ell_2}:(x_2)\,dm_2(x_2)
\end{equation*}
and we find that this integral is equal to
\begin{equation*}
\sigma_{j_2}^{2\ell_2}\int_{\mathcal{H}_2}D^{\ell_1+\ell_2}f(x)(\underbrace{e_{j_1},\dots,e_{j_1}}_{\ell_1\ \mathrm{times}},\underbrace{e_{j_2},\dots,e_{j_2}}_{\ell_2\ \mathrm{times}})\,dm_2(x_2).
\end{equation*}
At the $i^{th}$ step ($1\le i < r$), we easily obtain that
\begin{equation}
\frac{\big\langle f, :(\mathfrak{Re}\langle e_{j_1},\cdot\rangle)^{\ell_1}\dots (\mathfrak{Re}\langle e_{j_r},\cdot\rangle)^{\ell_r}:\big\rangle_{L^2(m)}}{\sigma_{j_1}^{2\ell_1}\dots\sigma_{j_i}^{2\ell_i}}
\label{quotient}
\end{equation}
is equal to
\begin{align*}
\int_{\mathcal{H}_1}\dots \int_{\mathcal{H}_{r+1}}D^{\ell_1+\dots+\ell_i}&f(x)(\underbrace{e_{j_1},\dots,e_{j_1}}_{\ell_1\ \text{times}},\dots,\underbrace{e_{j_i},\dots,e_{j_i}}_{\ell_i\ \text{times}})\\
&\times\prod_{t=i+1}^{r}:(\mathfrak{Re}\langle e_{j_{t}},\cdot\rangle)^{\ell_{t}}:(x_{t})\,dm_1(x_1)\dots dm_{r+1}(x_{r+1}),
\end{align*}
and finally, since $\ell_1+\dots+\ell_r=k$, we find that \eqref{quotient} is equal to
\begin{align*}
\int_{\mathcal{H}_1}\dots \int_{\mathcal{H}_{r+1}}& D^kf(x)(\underbrace{e_{j_1},\dots,e_{j_1}}_{\ell_1\ \text{times}},\dots,\underbrace{e_{j_r},\dots,e_{j_r}}_{\ell_r\ \text{times}})\,dm_1(x_1)\dots dm_{r+1}(x_{r+1})\\
&=\int_{\mathcal{H}}D^kf(x)(\underbrace{e_{j_1},\dots,e_{j_1}}_{\ell_1\ \text{times}},\dots,\underbrace{e_{j_r},\dots,e_{j_r}}_{\ell_r\ \text{times}})\,dm(x),
\end{align*}
which proves the lemma.
\end{proof}

\begin{Rem}
In the sequel, we denote by $\int_\mathcal{H}D^kf(x)\,dm(x)$ the $k$-linear form defined by 
$$
\int_\mathcal{H}D^kf(x)\,dm(x)(x_1,\dots, x_k):=\int_\mathcal{H}D^kf(x)(x_1,\dots, x_k)\,dm(x)
$$
for every vectors $x_1,\dots, x_k$ of $\mathcal{H}$.
\end{Rem}

Further on in the proof, we will need to write a function $f_k$ in $:\mathcal{G}^k:$ in a way which is a bit different from \eqref{developpement}. We now expand $f_k$ as
\begin{equation} \label{developpement2}
f_k=\sum_{(i_1,\dots, i_k)\in (\mathbb{Z}^*)^k}\alpha_{i_1,\dots,i_k}^{(k)}:\mathfrak{Re}\langle \mathfrak{e}_{i_1},\cdot \rangle\dots \mathfrak{Re}\langle \mathfrak{e}_{i_k},\cdot \rangle:
\end{equation}
where this sum is taken over all the $k$-tuples $(i_1,\dots,i_k)$ of integers different from zero. The difference with the first decomposition \eqref{developpement} is that each Wick transform 
$$
:\mathfrak{Re}\langle \mathfrak{e}_{i_1},\cdot \rangle\dots \mathfrak{Re}\langle \mathfrak{e}_{i_k},\cdot \rangle:
$$ 
appears several times and so that the coefficient $\alpha_{i_1,\dots,i_k}^{(k)}$ comes from all the coefficients $a_{i_{\sigma(1)},\dots,i_{\sigma(k)}}^{(k)}$, where $\sigma$ is a permutation in $\mathfrak{S}_k$ such that $i_{\sigma(1)}\le\dots\le i_{\sigma(k)}$. More precisely, the computation of $\alpha_{i_1,\dots,i_k}^{(k)}$ is given by the next proposition.
\begin{Prop}\label{nouveauxcoefficients}
For any $k$-tuple $(i_1,\dots,i_k)$ of integers different from zero, we can find some integers $r$, $j_1,\dots,j_r$ and $\ell_1,\dots,\ell_r$ such that the set $\{i_1,\dots,i_k\}$ is equal to the set $\{\underbrace{j_1,\dots,j_1}_{\ell_1\ \mathrm{times}},\dots,\underbrace{j_r,\dots,j_r}_{\ell_r\ \mathrm{times}}\}$ where $j_1<\dots<j_r$ and $\ell_1+\dots+\ell_r=k$. Furthermore,
\begin{equation}\label{coefficient2}
\alpha_{i_1,\dots, i_k}^{(k)}=\frac{\ell_1!\dots \ell_r!}{k!}\,a_{\underbrace{j_1,\dots,j_1}_{\ell_1\ \mathrm{times}},\dots,\underbrace{j_r,\dots, j_r}_{\ell_r\ \mathrm{times}}}^{(k)}.
\end{equation}
\end{Prop}
\begin{proof}
The only thing we really need to prove is \eqref{coefficient2}. It is based on the following combinatorial fact.
\begin{Fac}\label{faitcombinatoire}
The number of $k$-tuples $(i_1,\dots,i_k)$ in $(\mathbb{Z}^*)^{k}$ such that 
$$
\{i_1,\dots,i_{k}\}=\{\underbrace{j_1,\dots,j_1}_{\ell_1\ \mathrm{times}},\dots,\underbrace{j_r,\dots,j_r}_{\ell_r\ \mathrm{times}}\}
$$ 
is equal to $\displaystyle\frac{k!}{\ell_1!\dots \ell_r!}$.
\end{Fac}
\begin{proof}
We want to compute the number of $k$-tuples we can produce with $\ell_1$ integers $j_1$,..., $\ell_{r-1}$ integers $j_{r-1}$ and $\ell_r$ integers $j_r$, where $j_p\ne j_q$ when $p\ne q$. The number of different positions of the $\ell_1$ integers $j_1$ in a $k$-tuple is exactly $\binom{k}{\ell_1}$. Then, for the $\ell_2$ integers $j_2$, there are $k-\ell_1$ positions left, that is $\binom{k-\ell_1}{\ell_2}$ possibilities. At the end, for the $\ell_r$ integers $j_r$, we have only $\binom{k-\ell_1-\dots-\ell_{r-1}}{\ell_r}=1$ possibility. Hence, the number of $k$-tuples we have is
\begin{align*}
\binom{k}{\ell_1} \binom{k-\ell_1}{\ell_2}\dots &\binom{k-\ell_1-\dots- \ell_{r-1}}{\ell_r}\\
&=\frac{k!}{\ell_1!(k-\ell_1)!}\cdot\frac{(k-\ell_1)!}{\ell_2!(k-\ell_2)!}\cdots \frac{(k-\ell_1-\dots-\ell_{r-1})!}{\ell_r!(k-\ell_1-\dots-\ell_r)!}\\
&=\frac{k!}{\ell_1!\dots \ell_r!},
\end{align*}
since $\ell_1+\dots+\ell_r=k$.
\end{proof}
\noindent We can now conclude the proof of Proposition \ref{nouveauxcoefficients} since our new coefficient $\alpha_{i_1,\dots,i_k}^{(k)}$ is the coefficient $a_{\underbrace{j_1,\dots,j_1}_{\ell_1\ \mathrm{times}},\dots,\underbrace{j_r,\dots,j_r}_{\ell_r\ \mathrm{times}}}^{(k)}$ divided by $\displaystyle\frac{k!}{\ell_1!\dots \ell_r!}\cdot$
\end{proof}

The sequence of coefficients $\big(\alpha_{i_1,\dots,i_k}^{(k)}\big)_{(i_1,\dots,i_k)\in (\mathbb{Z}^*)^k}$ in the expansion \eqref{developpement2} satisfies a property of symmetry which is defined above.
\begin{Def}
A sequence of real numbers $(\alpha_{i_1,\dots,i_k})_{(i_1,\dots,i_k)\in (\mathbb{Z}^*)^k}$ is said to be \textit{symmetric} if for every permutation $\sigma$ in $\mathfrak{S}_k$, $\alpha_{i_{\sigma(1)},\dots,i_{\sigma(k)}}=\alpha_{i_1,\dots,i_k}$ for any $k$-tuple $(i_1,\dots, i_k)$ of integers different from zero.
\end{Def}
The reason we consider this new expansion \eqref{developpement2} is that the symmetry of the sequence of coefficients defined in \eqref{coefficient2} enables us to define a \textit{symmetric} $k$-linear form associated to $f_k$ by setting
\begin{eqnarray}\label{multilinearform}
\mathcal{B}_{f_k}:\ell_2(\mathbb{Z}^*,\mathbb{R})\times\dots\times\ell_2(\mathbb{Z}^*,\mathbb{R})&\longrightarrow& \ \ \ \ \ \ \ \ \ \ \mathbb{R}\\
\notag \big(\big(x_{i_1}^{(1)}\big)_{i_1\in \mathbb{Z}^*},\dots,\big(x_{i_k}^{(k)}\big)_{i_k\in \mathbb{Z}^*}\big)&\longmapsto& \sum_{(i_1,\dots, i_k)\in (\mathbb{Z}^*)^k}\alpha_{i_1,\dots, i_{k}}^{(k)}x_{i_1}^{(1)}\dots x_{i_k}^{(k)}.
\end{eqnarray}

\begin{Rem}
When the $k$-linear form $\mathcal{B}_{f_k}$ is well defined, that is to say when the series
$$
 \sum_{(i_1,\dots, i_k)\in (\mathbb{Z}^*)^k}\alpha_{i_1,\dots, i_{k}}^{(k)}x_{i_1}^{(1)}\dots x_{i_k}^{(k)}
$$
is convergent for every vectors $x^{(1)},\dots, x^{(k)}$ in $\ell_2(\mathbb{Z}^*,\mathbb{R})$, then it is easy to prove by using the uniform boundedness principle that this $k$-linear form is bounded.
\end{Rem}

We recall that the multilinear form $\mathcal{B}_{f_k}$ is continuous at zero if and only if it is bounded on $\mathcal{H}$, that is if there exists a positive constant $C$ such that for every $\big(x_{i_1}^{(1)}\big)_{i_1\in \mathbb{Z}^*},\dots,\big(x_{i_k}^{(k)}\big)_{i_k\in \mathbb{Z}^*}$ in $\ell_{2}(\mathbb{Z}^*,\mathbb{R})$, we have
\begin{equation*}\label{multilinearnorm}
\Big\vert\sum_{(i_1,\dots, i_k)\in (\mathbb{Z}^*)^k}\alpha_{i_1,\dots, i_k}^{(k)}x_{i_1}^{(1)}\dots x_{i_k}^{(k)} \Big\vert\le C\, \big\vert\big\vert x^{(1)}\big\vert\big\vert_{2}\,\dots \big\vert\big\vert x^{(k)}\big\vert\big\vert_{2},
\end{equation*}
where $\vert\vert x^{(i)}\vert\vert_2^2:=\sum_{p\in \mathbb{Z}^*}\vert x^{(i)}_p\vert^2$.
The important result of this section is the connection between the $k$-linear form $\mathcal{B}_{f_k}$ and our infinitely differentiable function $f$ which satisfies \eqref{intfinite}. We introduce the unitary operator $\vartheta$ which is defined by
\begin{align*}
\vartheta : \ell_2(\mathbb{Z}^*,\mathbb{R})&\longrightarrow \mathcal{H}\\
x& \longmapsto \sum_{j\in \mathbb{Z}^*}x_j\mathfrak{e}_j=\sum_{j=1}^{+\infty}(x_j+i x_{-j})e_j.
\end{align*}
The $k$-linear forms $\mathcal{B}_{f_k}$ may exist if $f$ is not infinitely differentiable. But in the case where our function $f$ is infinitely differentiable, we can give an integral representation for $\mathcal{B}_{f_k}$.

\begin{Theo}\label{bounded}
For any positive integer $k$, we have the following expression for the $k$-linear form $\mathcal{B}_{f_k}:$
\begin{equation*}
\mathcal{B}_{f_k}\big(x^{(1)},\dots, x^{(k)}\big) =\frac{1}{k!}\int_{\mathcal{H}}D^kf(x)\big(\vartheta\big(x^{(1)}\big),\dots, \vartheta\big(x^{(k)}\big)\big)\, dm(x)
\end{equation*}
for every $x^{(1)},\dots, x^{(k)}\in \ell_2(\mathbb{Z}^*,\mathbb{R})$. In particular, $\mathcal{B}_{f_k}$ is bounded and 
\begin{equation*}\label{norme}
\vert\vert \mathcal{B}_{f_k} \vert\vert = \frac{1}{k!}\bigg\vert\bigg\vert \int_{\mathcal{H}}D^kf(x)\, dm(x)\bigg\vert\bigg\vert
\end{equation*}
where the norm $\vert\vert \cdot \vert\vert$ has been defined in \eqref{normemultilineaire}.
\end{Theo}
\begin{proof}
It follows from Lemma \ref{coefficient} and \eqref{coefficient2} that for every $r$-tuples $(j_1,\dots,j_r)$ and $(\ell_1,\dots,\ell_r)$ of integers such that $j_1<\dots<j_r$ and $\ell_1+\dots+\ell_r=k$, we have
\begin{align*}
\alpha_{\underbrace{j_1,\dots,j_1}_{\ell_1\ \mathrm{times}},\dots,\underbrace{j_r,\dots,j_r}_{\ell_r\ \mathrm{times}}}^{(k)}&=\frac{\ell_1!\dots\ell_r!}{k!}\cdot\frac{1}{\ell_1!\dots\ell_r!}\int_{\mathcal{H}}D^kf(z)\,dm(z)(\underbrace{\mathfrak{e}_{j_1},\dots, \mathfrak{e}_{j_1}}_{\ell_1\ \mathrm{times}},\dots, \underbrace{\mathfrak{e}_{j_r},\dots, \mathfrak{e}_{j_r}}_{\ell_r\ \mathrm{times}})\\
&=\frac{1}{k!}\int_\mathcal{H} D^kf(z)\,dm(z)(\underbrace{\mathfrak{e}_{j_1},\dots, \mathfrak{e}_{j_1}}_{\ell_1\ \mathrm{times}},\dots, \underbrace{\mathfrak{e}_{j_r},\dots, \mathfrak{e}_{j_r}}_{\ell_r\ \mathrm{times}}).
\end{align*}
Since the sequence $\big(\alpha_{i_1,\dots,i_k}^{(k)}\big)_{(i_1,\dots,i_k)\in (\mathbb{Z}^*)^k}$ is symmetric, we can deduce that for any $(i_1,\dots,i_k)$ in $(\mathbb{Z}^*)^k$,
\begin{equation*}
\alpha_{i_1,\dots, i_k}^{(k)}=\frac{1}{k!}\int_\mathcal{H} D^kf(z)\,dm(z)(\mathfrak{e}_{i_1},\dots, \mathfrak{e}_{i_k}).
\end{equation*}
Then, for every $x^{(1)}=\big(x_{i_1}^{(1)}\big)_{i_1\in \mathbb{Z}^*},\dots, x^{(k)}=\big(x_{i_k}^{(k)}\big)_{i_k\in \mathbb{Z}^*}$ in $\ell_2(\mathbb{Z}^*,\mathbb{R})$, we have
\begin{eqnarray*}
\mathcal{B}_{f_k}(x^{(1)},\dots, x^{(k)})&=&\sum_{(i_1,\dots, i_k)\in (\mathbb{Z}^*)^k}\alpha_{i_1,\dots, i_k}^{(k)}\,x_{i_1}^{(1)}\dots x_{i_k}^{(k)}\\
&=&\frac{1}{k!}\sum_{(i_1,\dots, i_k)\in (\mathbb{Z}^*)^k}\int_{\mathcal{H}}D^kf(z)\,dm(z)(\mathfrak{e}_{i_1},\dots, \mathfrak{e}_{i_k})\,x_{i_1}^{(1)}\dots x_{i_k}^{(k)}\\
&=&\frac{1}{k!}\int_{\mathcal{H}}D^kf(z)\,dm(z)\big(\vartheta\big(x^{(1)}\big),\dots, \vartheta\big(x^{(k)}\big)\big)
\end{eqnarray*}
and the theorem is proved.
\end{proof}

With the conditions of regularity we presented in this subsection, we can prove a result about the speed of mixing for functions which satisfy the integral condition \eqref{intfinite}.

\subsection{Speed of mixing in the spaces $:\mathcal{G}^k:$}
Before stating the general result, some estimations are needed. The first of these, which is essentially based on Parseval's theorem, will give the rate of mixing term in the main theorem.

\begin{Lem}\label{norm1}
For every positive integer $n$, we have
\begin{equation}
0\le \mathcal{I}_n\big(\vert\vert \cdot \vert\vert^2,\vert\vert \cdot \vert\vert^2\big) \le \frac{C(E)^2\,\pi^{2\alpha}}{n^{2\alpha}}\,\vert\vert E \vert\vert_2^2,
\end{equation}
where by definition $\displaystyle \vert\vert E \vert\vert_2^2=\int_\mathbb{T}\vert\vert E(\lambda)\vert\vert^2\,d\mu(\lambda)$.
\end{Lem}
\begin{proof}
Recall that the correlation $\mathcal{I}_n(\vert\vert \cdot \vert\vert^2,\vert\vert \cdot \vert\vert^2)$ is defined by
\begin{equation*}
\mathcal{I}_n\big(\vert\vert \cdot \vert\vert^2,\vert\vert \cdot \vert\vert^2\big)=\int_\mathcal{H} \vert\vert T^nx\vert\vert^2\,\vert\vert x \vert\vert^2\,dm(x)-\bigg(\int_\mathcal{H} \vert\vert x \vert\vert^2\,dm(x)\bigg)^2
\end{equation*}
and that we can rewrite it in a more tractable way as
\begin{equation*}
\mathcal{I}_n\big(\vert\vert \cdot \vert\vert^2,\vert\vert \cdot \vert\vert^2\big)=\sum_{(k,\ell)\in \mathbb{N}^2}\int_{\mathcal{H}}\vert \langle e_k, T^nx \rangle \vert^2\,\vert \langle e_\ell, x \rangle \vert^2\, dm(x)-\bigg(\int_\mathcal{H} \vert\vert x \vert\vert^2\,dm(x)\bigg)^2.
\end{equation*}
We now compute each integral
\begin{equation*}
\int_{\mathcal{H}}\vert \langle e_k, T^nx \rangle \vert^2\,\vert \langle e_\ell, x \rangle \vert^2\, dm(x)
\end{equation*}
by splitting the two factors of the integrand term into real and imaginary parts. Since for every vector $x$ of $\mathcal{H}$, we have
\begin{equation}\label{wicktransform}
(\mathfrak{Re}\langle x, \cdot \rangle)^2=:(\mathfrak{Re}\langle x, \cdot \rangle)^2:+\,\sigma_x^2,
\end{equation}
we easily find, by using the isometry \eqref{isometrie} between $\big(:\mathcal{G}^2:,\langle \cdot,\cdot \rangle_{L^2(m)}\big)$ and $\big(\mathcal{G}_{\odot}^{2},\langle \cdot, \cdot \rangle_{\odot}\big)$ and \eqref{rotation1}, that
\begin{align}\label{integrale1}
\int_{\mathcal{H}} (\mathfrak{Re}\langle e_k, &T^nx \rangle)^2\,(\mathfrak{Re}\langle e_\ell, x \rangle)^2\, dm(x)= \frac{2!}{2^2}(\mathfrak{Re}\langle RT^{*n}e_k,e_\ell \rangle)^2+\sigma_k^2\sigma_\ell^2.
\end{align}
Furthermore, since $\mathfrak{Im}\langle e_\ell, \cdot \rangle=\mathfrak{Re}\langle i e_{\ell}, \cdot \rangle$, we deduce from \eqref{integrale1} that
\begin{align}\label{integrale2}
\notag \int_{\mathcal{H}}(\mathfrak{Re}\langle e_k, T^nx \rangle)^2\,(\mathfrak{Im}\langle e_\ell, x \rangle)^2\, dm(x)&=\frac{1}{2}(\mathfrak{Re}\langle RT^{*n}e_k, ie_\ell \rangle)^2+\sigma_k^2\sigma_\ell^2\\
&=\frac{1}{2}(\mathfrak{Im}\langle RT^{*n}e_k, e_\ell \rangle)^2+\sigma_k^2\sigma_\ell^2.
\end{align}
By using the same method, we get
\begin{equation}\label{integrale3}
\int_{\mathcal{H}}(\mathfrak{Im}\langle e_k, T^nx \rangle)^2\,(\mathfrak{Re}\langle e_\ell, x \rangle)^2\, dm(x)=\frac{1}{2}(\mathfrak{Im}\langle RT^{*n}e_k, e_\ell \rangle)^2+\sigma_k^2\sigma_\ell^2
\end{equation}
and
\begin{equation}\label{integrale4}
\int_{\mathcal{H}}(\mathfrak{Im}\langle e_k, T^nx \rangle)^2\,(\mathfrak{Im}\langle e_\ell, x \rangle)^2\, dm(x)=\frac{1}{2}(\mathfrak{Re}\langle RT^{*n}e_k, e_\ell \rangle)^2+\sigma_k^2\sigma_\ell^2.
\end{equation}
We can deduce from \eqref{integrale1}, \eqref{integrale2}, \eqref{integrale3} and \eqref{integrale4} that
\begin{equation*}
\int_{\mathcal{H}}\vert \langle e_k, T^nx \rangle \vert^2\,\vert \langle e_\ell, x \rangle \vert^2\, dm(x)=\vert \langle RT^{*n} e_k, e_\ell \rangle \vert^2+4\,\sigma_k^2\sigma_\ell^2.
\end{equation*}
Finally, since $\int_{\mathcal{H}}\vert\vert \cdot \vert\vert^2\, dm=2\sum_{j\ge 1}\sigma_j^2$, we get
\begin{equation}\label{ecriture}
\mathcal{I}_n\big(\vert\vert \cdot \vert\vert^2,\vert\vert \cdot \vert\vert^2\big)=\sum_{(k,\ell) \in \mathbb{N}^2}\vert \langle RT^{*n}e_k,e_\ell \rangle \vert^2.
\end{equation}
We now need the integral representation \eqref{representation} of $\langle RT^{*n}e_k,e_\ell \rangle$: for any positive integers $k$ and $\ell$, we have
\begin{align*}
\langle RT^{*n}e_k,e_\ell \rangle =\int_{\mathbb{T}}\lambda^{n} \langle e_k,E(\lambda) \rangle\overline{\langle e_\ell , E(\lambda) \rangle} \, d\mu(\lambda)
=\Big\langle e_k,\int_{\mathbb{T}}\lambda^{n}\overline{\langle e_\ell , E(\lambda)\rangle}  E(\lambda)\,d\mu(\lambda)  \Big\rangle,
\end{align*}
and by the Plancherel theorem, 
\begin{equation*}
\mathcal{I}_n\big(\vert\vert \cdot \vert\vert^2,\vert\vert \cdot \vert\vert^2\big)=\sum_{\ell=1}^{+\infty}\bigg\vert\bigg\vert \int_{\mathbb{T}}\lambda^{n}\overline{\langle e_\ell , E(\lambda) \rangle} E(\lambda)\,d\mu(\lambda)\bigg\vert\bigg\vert^2.
\end{equation*}
We need to estimate the Fourier coefficient $c_n=\int_{\mathbb{T}}\lambda^{n}\overline{\langle e_\ell , E(\lambda) \rangle} E(\lambda)\,d\mu(\lambda)$ which appears in this equation and the way to do this is the same as in the proof of Proposition \ref{functional}. For any $\theta$, we put $\theta_n:=\theta+\frac{\pi}{n}$ and we have
\begin{align*}
\vert\vert c_n \vert\vert^2 =\frac{1}{4}\bigg\vert\bigg\vert\int_{0}^{2\pi}e^{in\theta}\Big[\overline{\langle e_\ell, E(e^{i\theta})-E(e^{i\theta_n}) \rangle} E(e^{i\theta})+\overline{\langle e_\ell, E(e^{i\theta_n})\rangle}\big(E(e^{i\theta})-E(e^{i\theta_n})\big)\Big]\frac{d\theta}{2\pi}\bigg\vert\bigg\vert^2
\end{align*}
Then $\vert\vert c_n\vert\vert^2$ is less or equal than
\begin{align*}
\frac{1}{2}&\bigg(\bigg\vert\bigg\vert \int_{0}^{2\pi}\overline{\langle e_\ell, E(e^{i\theta})-E(e^{i\theta_n}) \rangle} E(e^{i\theta})\,\frac{d\theta}{2\pi} \bigg\vert\bigg\vert^2+\bigg\vert\bigg\vert \int_{0}^{2\pi}\overline{\langle e_\ell, E(e^{i\theta_n})\rangle}\big(E(e^{i\theta})-E(e^{i\theta_n})\big)\frac{d\theta}{2\pi} \bigg\vert\bigg\vert^2\bigg)\\
&\le \frac{1}{2}\bigg(\int_{0}^{2\pi}\big\vert\langle e_\ell, E(e^{i\theta})-E(e^{i\theta_n}) \rangle\big\vert^2\,\frac{d\theta}{2\pi}\,\vert\vert E \vert\vert_2^2\\
&\qquad \qquad \qquad \qquad\qquad\qquad+\int_{0}^{2\pi}\big\vert \langle e_\ell, E(e^{i\theta_n})\big\rangle \big\vert^2\,\big\vert\big\vert E(e^{i\theta})-E(e^{i\theta_n})\big\vert\big\vert^2\,\frac{d\theta}{2\pi}\bigg)
\end{align*}
where the last inequality is a consequence of the Cauchy-Schwarz inequality. Since the $\mathbb{T}$-eigenvector field $E$ is $\alpha$-H\"olderian with H\"older constant $C(E)$, we conclude that the correlation $\mathcal{I}_n\big(\vert\vert \cdot \vert\vert^2,\vert\vert \cdot \vert\vert^2\big)$ is less or equal that
\begin{align*}
\frac{1}{2}\int_{0}^{2\pi}\big\vert\big\vert E(e^{i\theta})-E(e^{i\theta_n}) \big\vert\big\vert^2\,\frac{d\theta}{2\pi}\,\vert\vert E \vert\vert_2^2+\frac{C(E)^2\,\pi^{2\alpha}}{2n^{2\alpha}}\,\vert\vert E \vert\vert_2^2
\le \frac{C(E)^2\,\pi^{2\alpha}}{n^{2\alpha}}\,\vert\vert E \vert\vert_2^2.
\end{align*}
\end{proof}
\noindent This lemma shows that the sequence of correlations decreases to zero with speed $n^{-2\alpha}$ if we consider the square norm function. Moreover, in the same way as in Fact \ref{transformation}, we have $\langle RT^{*n}e_k,e_\ell\rangle=2\,\sigma_\ell^2\,\langle e_k,T^ne_\ell\rangle$ for any positive integers $k,\ell$. Then we deduce from \eqref{ecriture} and the conclusion of Lemma \ref{norm1} that the sum $\sum_{k\ge 1}\sigma_k^4\,\vert\vert T^ne_k \vert\vert^2$ tends to zero as $n$ goes to infinity. More precisely, we have the following corollary.
\begin{Cor}\label{ecriture2}
For any positive integer $n$, we have
\begin{equation*}
\sum_{k=1}^{+\infty}\sigma_k^4\,\vert\vert T^ne_k\vert\vert^2\le \frac{C(E)^2\,\pi^{2\alpha}}{4 n^{2\alpha}}\,\vert\vert E \vert\vert_2^2.
\end{equation*}
\end{Cor}
\noindent The next lemma says that if we replace $\sigma_k^4$ by $\sigma_k^2$ in the sum which appears in Corollary \ref{ecriture2}, then this sum remains uniformly bounded in $n$.

\begin{Lem}\label{seriemajoree}
For any positive integer $n$, the series $\sum_{k\ge 1}\sigma_k^2\, \vert\vert T^n e_k \vert\vert^2$ is convergent. More precisely, for any $n\ge 1$,
\begin{equation}\label{finitesum1}
\sum_{k=1}^{+\infty}\sigma_k^2\, \vert\vert T^n e_k \vert\vert^2\le \frac{\vert\vert E \vert\vert_2^2}{2}\cdot
\end{equation}
\end{Lem}
\begin{proof}
Since $Re_k=2\, \sigma_k^2\,e_k$ for any positive integer $k$, the intertwining equation $TK=KV$ gives us
\begin{equation*}
2\sigma_k^2\, T^ne_k =T^n KK^{*}e_k=KV^nK^{*}e_k=\int_{\mathbb{T}}\lambda^{n}\overline{\langle e_k,E(\lambda)\rangle}E(\lambda)\,d\mu(\lambda).
\end{equation*}
Then, for any positive integer $k$, we define the function $\omega_k:=\overline{\frac{\langle e_k, E(\cdot)\rangle}{\sigma_k\sqrt{2}}}$ and we get by applying the Parseval theorem in $(\mathcal{H}, \langle \cdot,\cdot \rangle)$ that
\begin{align*}
2\sum_{k=1}^{+\infty}\sigma_k^2\, \vert\vert T^n e_k \vert\vert^2&=\sum_{k=1}^{+\infty}\bigg\vert\bigg\vert \int_{\mathbb{T}}\lambda^n\frac{\overline{\langle e_k,E(\lambda)\rangle}}{\sigma_k\sqrt{2}}E(\lambda)\,d\mu(\lambda) \bigg\vert\bigg\vert^2 \\
&=\sum_{k=1}^{+\infty}\bigg\vert\bigg\vert \sum_{p=1}^{+\infty}\int_{\mathbb{T}}\lambda^n\omega_k(\lambda) \langle e_p, E(\lambda)\rangle\, d\mu(\lambda)\, e_p
 \bigg\vert\bigg\vert^2\\
 &=\sum_{(k,p)\in \mathbb{N}^2}\bigg\vert\int_{\mathbb{T}}\lambda^n \omega_k(\lambda) \langle e_p, E(\lambda)\rangle\, d\mu(\lambda)\bigg\vert^2.
\end{align*}
By Corollary \ref{orthogonality}, the sequence $\big(\overline{\omega_k}\big)_{k\in \mathbb{N}}$ is orthogonal in $L^2(\mathbb{T},\mu)$ and for any positive integer $k$, the norm of $\omega_k$ in this space is equal to $1$. Hence, if we denote by $\Phi_{n,p}\in L^2(\mathbb{T},\mu)$ the function $\lambda \mapsto \overline{\lambda^n \langle e_p, E(\lambda)\rangle}$, the Bessel theorem in $L^2(\mathbb{T}, \mu)$ gives us
\begin{align*}
\sum_{k=1}^{+\infty}\bigg\vert \int_{\mathbb{T}}\omega_k(\lambda)\overline{\Phi_{n,p}(\lambda)}\, d\mu(\lambda)\bigg\vert^2&\le \vert\vert \Phi_{n,p}\vert\vert_{L^2(\mathbb{T},\mu)}^2\\
&=\int_{\mathbb{T}}\big\vert \lambda^n\langle e_p, E(\lambda)\rangle \big\vert^2\, d\mu(\lambda)\\
&=\int_{\mathbb{T}}\big\vert \langle e_p, E(\lambda)\rangle \big\vert^2\, d\mu(\lambda).
\end{align*}
Hence, we can conclude that $2\sum_{k\ge 1}^{}\sigma_{k}^2\, \vert\vert T^ne_k \vert\vert^2\le \vert\vert E \vert\vert_{2}^2$ for any positive integer $n$.
\end{proof}

To finish, we need an estimate of the moments of the measure $m$.

\begin{Prop}\label{moments}
For any positive integer $k$, we have the following estimate:
\begin{equation*}
\int_{\mathcal{H}}\vert\vert x \vert\vert^{2k}\, dm(x)\le k!\, \vert\vert E \vert\vert_2^{2k}.
\end{equation*}
\end{Prop}
\begin{proof}
The case $k=1$ is important for the rest of the proof. It is a consequence of Corollary \ref{orthogonality}:
\begin{equation*}
\int_{\mathcal{H}}\vert\vert x \vert\vert^2\, dm(x)=2\sum_{j=1}^{+\infty}\sigma_{j}^2=\vert\vert E \vert \vert_{2}^2.
\end{equation*}
We now fix a positive integer $k$. We expand our integral as
\begin{equation*}
\int_{\mathcal{H}}\vert\vert x \vert\vert^{2k}\,dm(x)=\sum_{(j_1,\dots,j_k)\in \mathbb{N}^k}\int_{\mathcal{H}}\vert\langle e_{j_1},x\rangle\vert^2\dots \vert \langle e_{j_k},x\rangle\vert^2\, dm(x)
\end{equation*}
and we need to estimate the integrals
\begin{equation*}
\int_{\mathcal{H}}\vert\langle e_{j_1},x\rangle\vert^2\dots \vert \langle e_{j_k},x \rangle\vert^2\, dm(x).
\end{equation*}
Each integral can be written in the form
\begin{equation*}
\int_{\mathcal{H}}\vert\langle e_{i_1},x\rangle\vert^{2\ell_1}\dots\vert\langle  e_{i_r},x \rangle\vert^{2\ell_r}\, dm(x)
\end{equation*}
where $r\in\{1,\dots,k\}$, $(\ell_1,\dots,\ell_r)\in \mathbb{N}^r$ with $\ell_1+\dots+\ell_r=k$ and $i_1<\dots<i_r$. It can be proved that 
\begin{equation*}
\int_\mathcal{H} \vert\langle y,x\rangle\vert^{2i}\,dm(x)=i!\bigg(\int_\mathcal{H}\vert\langle y,x\rangle\vert^2\,dm(x)\bigg)^i=i!\,2^i\,\sigma_y^{2i}
\end{equation*}
for every vector $y$ of $\mathcal{H}$ and any nonnegative integer $i$ (see for instance \cite{J}, Chapter $1$). But the random variables $\langle e_k,\cdot\rangle$ are independent by Proposition \ref{basis}. So we deduce that
\begin{align*}
\int_{\mathcal{H}}\vert\langle e_{i_1},x\rangle\vert^{2\ell_1}\dots\vert\langle  e_{i_r},x \rangle\vert^{2\ell_r}\, dm(x)=\prod_{t=1}^{r}\int_{\mathcal{H}}\vert \langle e_{i_t},x\rangle \vert^{2\ell_t}\,dm(x)&=\prod_{t=1}^{r}\big(\ell_t!\, 2^{\ell_t}\sigma_{i_t}^{2\ell_t}\big)\\
&=\Bigg(\prod_{t=1}^{r}\ell_t!\Bigg)2^k\sigma_{i_1}^{2\ell_1}\dots\sigma_{i_r}^{2\ell_r},
\end{align*} 
since $\ell_1+\dots+\ell_r=k$. We now use the inequality $i!\,j!\le (i+j)!$, which is easily seen to be true for any nonnegative integers $i$ and $j$, and we get
\begin{equation*}
\int_{\mathcal{H}}\vert\langle e_{i_1},x\rangle\vert^{2\ell_1}\dots\vert\langle e_{i_r},x \rangle\vert^{2\ell_r}\, dm(x)\le k!\,2^k\, \sigma_{i_1}^{2\ell_1}\dots\sigma_{i_r}^{2\ell_r}.
\end{equation*}
Eventually, we find that
\begin{equation*}
\int_{\mathcal{H}}\vert\vert x \vert\vert^{2k}\,dm(x)\le k!\, 2^k \sum_{(j_1,\dots, j_k)\in \mathbb{N}^k}\sigma_{j_1}^2\dots\sigma_{j_k}^2=k!\,\bigg(2\sum_{j=1}^{+\infty}\sigma_j^2\bigg)^k=k!\, \vert\vert E \vert\vert_2^{2k}.
\end{equation*}
\end{proof}

By using the estimate of Proposition \ref{moments}, we can prove that, given a function $f$ in $L^2(\mathcal{H},\mathcal{B},m)$ such that the multilinear forms $\mathcal{B}_{f_k}$ are bounded (where $f=\sum_{k\ge 0}f_k$ is the Wiener chaos decomposition of $f$), then the series $\sum_{(i_1,\dots,i_k)\in (\mathbb{Z}^*)^k} \big\vert\alpha_{i_1,\dots,i_k}^{(k)}\big\vert^2 \sigma_{i_1}^2\dots \sigma_{i_{k-1}}^2$ is convergent. In particular, it is the case when the function $f$ is an infinitely differentiable real-valued function on $\mathcal{H}$ which satisfies condition \eqref{intfinite}.
\begin{Cor}\label{sommeconvergente}
Let $f\in L^2_\mathbb{R}(\mathcal{H},\mathcal{B},m)$ where $f_k$ is written as in \eqref{developpement2} such that the multilinear forms $\mathcal{B}_{f_k}$ are bounded. Then the series $\sum_{(i_1,\dots,i_k)\in (\mathbb{Z}^*)^k}\big\vert \alpha_{i_1,\dots,i_k}^{(k)}\big\vert^2\, \sigma_{i_1}^2\dots\sigma_{i_{k-1}}^2$ is convergent for any positive integer $k$. More precisely, we have the following estimate:
\begin{equation}\label{finite2}
\sum_{(i_1,\dots, i_k)\in (\mathbb{Z}^*)^k}\big\vert \alpha_{i_1,\dots, i_k}^{(k)}\big\vert^2\,\sigma_{i_1}^2\dots\sigma_{i_{k-1}}^2\le \vert\vert \mathcal{B}_{f_k}\vert\vert^2\,\vert\vert E \vert\vert_2^{2(k-1)}.
\end{equation}
\end{Cor}
\begin{proof}
In order to get this estimate, we consider the quantities
\begin{equation}\label{expression}
\mathcal{S}_k:=\sum_{i_k\in \mathbb{Z}^*}\int_{\mathcal{H}}\bigg(\sum_{(i_1,\dots, i_{k-1})\in (\mathbb{Z}^*)^{k-1}}\alpha_{i_1,\dots, i_k}^{(k)}:\mathfrak{Re}\langle \mathfrak{e}_{i_1},\cdot\rangle \dots \mathfrak{Re}\langle \mathfrak{e}_{i_{k-1}},\cdot\rangle :(x)\bigg)^2\,dm(x).
\end{equation}
First, we give an upper bound of $\mathcal{S}_k$ and secondly we compute explicitly this quantity. Recall that for a positive integer $j$ and for a function $g$ in $\mathcal{G}^{j}$, the Wick transform of $g$ is defined by $:g:=(Id-\mathcal{P}_{j})g$, where $Id$ is the identity operator and $\mathcal{P}_{j}$ denotes the orthogonal projection onto $\displaystyle \overline{\text{span}}^{L^2_{\mathbb{R}}(\mathcal{H},\mathcal{B},m)}\big[\mathcal{G}^i\, ; \, 0\le i \le j-1\big]$. Since the functions $\mathfrak{Re}\langle \mathfrak{e}_{i_1},\cdot\rangle\dots\mathfrak{Re}\langle \mathfrak{e}_{i_{k-1}},\cdot\rangle$ belong to $\mathcal{G}^{k-1}$, we have
\begin{align*}
\mathcal{S}_k&=\sum_{i_k\in \mathbb{Z}^*}\bigg\vert\bigg\vert (Id-\mathcal{P}_{k-1})\Big( \sum_{(i_1,\dots, i_{k-1})\in (\mathbb{Z}^*)^{k-1}} \alpha_{i_1,\dots, i_k}^{(k)}\,\mathfrak{Re}\langle \mathfrak{e}_{i_1},\cdot\rangle \dots \mathfrak{Re}\langle \mathfrak{e}_{i_{k-1}},\cdot\rangle\Big) \bigg\vert\bigg\vert^2_{L^2(m)}\\
&\le \sum_{i_k\in \mathbb{Z}^*}\bigg\vert\bigg\vert  \sum_{(i_1,\dots, i_{k-1})\in (\mathbb{Z}^*)^{k-1}} \alpha_{i_1,\dots, i_k}^{(k)}\,\mathfrak{Re}\langle \mathfrak{e}_{i_1},\cdot\rangle \dots \mathfrak{Re}\langle \mathfrak{e}_{i_{k-1}},\cdot\rangle \bigg\vert\bigg\vert^2_{L^2(m)}.
\end{align*}
Now, the upper bound comes from the boundedness of the $k$-linear form $\mathcal{B}_{f_k}$. Indeed, for every vectors $x,z$ of $\mathcal{H}$, we know that
\begin{equation*}
\Big\vert\sum_{(i_1,\dots, i_k)\in (\mathbb{Z}^*)^k}\alpha_{i_1,\dots, i_k}^{(k)}\,\mathfrak{Re}\langle \mathfrak{e}_{i_1},x\rangle\dots \mathfrak{Re}\langle \mathfrak{e}_{i_{k-1}},x\rangle\mathfrak{Re}\langle \mathfrak{e}_{i_k},z\rangle\Big\vert\le \vert\vert \mathcal{B}_{f_k}\vert\vert\, \vert\vert x \vert\vert^{k-1}\, \vert\vert z \vert\vert.
\end{equation*}
By taking the supremum over all the vectors $z$ in the closed unit ball of $\mathcal{H}$, we get
\begin{equation*}
\sum_{i_k\in \mathbb{Z}^*}\Big(\sum_{(i_1,\dots, i_{k-1})\in (\mathbb{Z}^*)^{k-1}}\alpha_{i_1,\dots, i_k}^{(k)}\,\mathfrak{Re}\langle \mathfrak{e}_{i_1},x\rangle\dots \mathfrak{Re}\langle \mathfrak{e}_{i_{k-1}},x\rangle\Big)^2\le \vert\vert \mathcal{B}_{f_k} \vert\vert^2\, \vert\vert x \vert\vert^{2(k-1)}.
\end{equation*}
Then, we deduce from the beginning of the proof that
\begin{equation}\label{premier}
\mathcal{S}_k\le \vert\vert \mathcal{B}_{f_k}\vert\vert^2 \int_{\mathcal{H}}\vert\vert x \vert\vert^{2(k-1)}\, dm(x)\le (k-1)!\,\vert\vert 
\mathcal{B}_{f_k}\vert\vert ^2\, \vert\vert E \vert\vert_{2}^{2(k-1)},
\end{equation}
where the last inequality results from Proposition \ref{moments}. The second part of the proof consists in the computation of $\mathcal{S}_k$. We expand $\mathcal{S}_k$ as
\begin{align}\label{integrales}
\notag \mathcal{S}_k=\sum_{i_k\in \mathbb{Z}^*}&\sum_{\substack{(i_1,\dots,i_{k-1})\in (\mathbb{Z}^*)^{k-1}\\(j_1,\dots,j_{k-1})\in (\mathbb{Z}^*)^{k-1}}}\alpha_{i_1,\dots, i_{k-1},i_k}^{(k)}\,\alpha_{j_1,\dots, j_{k-1},i_k}^{(k)}\\
 &\times \int_{\mathcal{H}}:\mathfrak{Re}\langle \mathfrak{e}_{i_1},\cdot\rangle\dots\mathfrak{Re}\langle \mathfrak{e}_{i_{k-1}},\cdot\rangle:(x):\mathfrak{Re}\langle \mathfrak{e}_{j_1},\cdot\rangle\dots\mathfrak{Re}\langle \mathfrak{e}_{j_{k-1}},\cdot\rangle:(x)\,dm(x)
\end{align}
 and the computation of the integrals \eqref{integrales} is given by the following combinatorial fact.
\begin{Fac}\label{fait}
The integral 
\begin{align*}
\int_{\mathcal{H}}:\mathfrak{Re}\langle \mathfrak{e}_{i_1},\cdot\rangle\dots\mathfrak{Re}\langle \mathfrak{e}_{i_{k-1}},\cdot\rangle:(x):\mathfrak{Re}\langle \mathfrak{e}_{j_1},\cdot\rangle\dots\mathfrak{Re}\langle \mathfrak{e}_{j_{k-1}},\cdot\rangle:(x)\,dm(x)
\end{align*}
denoted by $\mathcal{I}(i_1,\dots,i_{k-1}\,;\, j_1,\dots,j_{k-1})$ is nonzero if and only if there exists a permutation $\tau$ in $\mathfrak{S}_{k-1}$ such that for every integer $\ell$ in $\{1,\dots,k-1\}$, $i_{\ell}=j_{\tau(\ell)}$.
\end{Fac}
\begin{proof}
This result is a consequence of Proposition \ref{correlationpoly}. Indeed, for $n=0$,  this proposition gives us that
\begin{equation*}
\mathcal{I}(i_1,\dots,i_{k-1}\,;\, j_1,\dots,j_{k-1})=\sigma_{j_1}^2\dots\sigma_{j_{k-1}}^2\sum_{\tau\in \mathfrak{S}_{k-1}}\mathfrak{Re}\langle \mathfrak{e}_{i_1}, \mathfrak{e}_{j_{\tau(1)}}\rangle\dots\mathfrak{Re}\langle \mathfrak{e}_{i_{k-1}}, \mathfrak{e}_{j_{\tau(k-1)}}\rangle.
\end{equation*}
Then, the integral is nonzero if and only if there exists a permutation $\tau$ in $\mathfrak{S}_{k-1}$ such that 
\begin{equation*}
\mathfrak{Re}\langle \mathfrak{e}_{i_1}, \mathfrak{e}_{j_{\tau(1)}}\rangle\dots\mathfrak{Re}\langle \mathfrak{e}_{i_{k-1}}, \mathfrak{e}_{j_{\tau(k-1)}}\rangle\ne 0 
\end{equation*}
since each term in the sum above is equal to $0$ or $1$ by the orthogonality of the sequence $(e_\ell)_{\ell \in \mathbb{N}}$. This means that for every integer $\ell$ in the set $\{1,\dots,k-1 \}$, $i_\ell=j_{\tau(\ell)}$.
\end{proof}
\noindent We now proceed with the proof of Corollary \ref{sommeconvergente}: Fact \ref{fait} above allows us to rewrite $\mathcal{S}_k$ as
\begin{align*}
\mathcal{S}_k=\sum_{i_k\in \mathbb{Z}^*}&\sum_{(i_1,\dots, i_{k-1})\in (\mathbb{Z}^*)^{k-1}}\sum_{\substack{(j_1,\dots,j_{k-1})\in (\mathbb{Z}^*)^{k-1}\\ \{i_1,\dots,i_{k-1}\}=\{j_1,\dots,j_{k-1}\}}}\big\vert\alpha_{i_1,\dots,i_{k-1},i_k}^{(k)}\big\vert^2\\
&\times\int_{\mathcal{H}}:\mathfrak{Re}\langle \mathfrak{e}_{i_1},\cdot\rangle\dots\mathfrak{Re}\langle \mathfrak{e}_{i_{k-1}},\cdot\rangle:(x):\mathfrak{Re}\langle \mathfrak{e}_{j_1},\cdot\rangle\dots\mathfrak{Re}\langle \mathfrak{e}_{j_{k-1}},\cdot\rangle:(x)\,dm(x)
\end{align*}
since the sequence $\big(\alpha_{i_1,\dots,i_k}^{(k)}\big)_{(i_1,\dots, i_k)\in (\mathbb{Z}^*)^k}$ is symmetric. A $(k-1)$-tuple of integers different from zero $(i_1,\dots,i_{k-1})$ can be written as $\displaystyle \{i_1,\dots,i_{k-1}\}=\{\underbrace{\ell_1,\dots,\ell_1}_{p_1\ \text{times}},\dots,\underbrace{\ell_r,\dots,\ell_r}_{p_r\ \text{times}}\}$ where $1\le r \le k-1$, $p_1+\dots+p_r=k-1$ and with $\ell_p\ne \ell_q$ when $p \ne q$. Then for any $(j_1,\dots,j_{k-1})$ in $(\mathbb{Z}^*)^{k-1}$ such that $\{i_1,\dots,i_{k-1}\}=\{j_1,\dots,j_{k-1}\}$, we know from Proposition \ref{variance} that
\begin{eqnarray*}
\mathcal{I}(i_1,\dots,i_{k-1}\,;\,j_1,\dots,j_{k-1})&=&\mathcal{I}(\underbrace{\ell_1,\dots,\ell_1}_{p_1\ \text{times}},\dots,\underbrace{\ell_r,\dots,\ell_r}_{p_r\ \text{times}}\,;\,\underbrace{\ell_1,\dots,\ell_1}_{p_1\ \text{times}},\dots,\underbrace{\ell_r,\dots,\ell_r}_{p_r\ \text{times}})\\
&=&\text{var}_m\big[:(\mathfrak{Re}\langle \mathfrak{e}_{\ell_1},\cdot\rangle)^{p_1}\dots(\mathfrak{Re}\langle \mathfrak{e}_{\ell_r},\cdot\rangle)^{p_r}:\big]\\
&=&p_1!\dots p_r!\ \sigma_{\ell_1}^{2p_1}\dots\sigma_{\ell_r}^{2p_r}.
\end{eqnarray*}
But we know from Fact \ref{faitcombinatoire} that the number of $(k-1)$-tuples $(j_1,\dots,j_{k-1})$ such that $\{i_1,\dots, i_{k-1}\}=\{j_1,\dots, j_{k-1}\}$ is equal to $\frac{(k-1)!}{p_1!\dots p_r!}\cdot$ Then we deduce that
\begin{equation}\label{deuxieme}
\mathcal{S}_k=(k-1)!\, \sum_{(i_1,\dots,i_k)\in (\mathbb{Z}^*)^k}\big\vert\alpha_{i_1,\dots,i_k}^{(k)}\big\vert^2\, \sigma_{i_1}^{2}\dots\sigma_{i_{k-1}}^2,
\end{equation}
and the result follows readily from \eqref{premier} and \eqref{deuxieme}.
\end{proof}
At this stage, we are able to prove the main result which gives the rate of mixing in each space $:\mathcal{G}^k:$.
\begin{Theo}\label{correlations}
Let $f_k,g_k $ be two functions in the space $:\mathcal{G}^k:$ such that the $k$-linear forms $\mathcal{B}_{f_k}$ and $\mathcal{B}_{g_k}$ are bounded. Then, for any positive integer $n$, we have
\begin{equation*}
\big\vert \mathcal{I}_n(f_k,g_k) \big\vert \le k!\,\frac{C(E)\,\pi^{\alpha}}{n^{\alpha}}\,\vert\vert E \vert\vert_2^{2k-1}\,\vert\vert \mathcal{B}_{f_k}\vert\vert\,\vert\vert \mathcal{B}_{g_k}\vert\vert.
\end{equation*}
\end{Theo}
\begin{proof}
 As usual, we write the functions $f_k$ and $g_k$ as in \eqref{developpement2} (with coefficients $\beta_{j_1,\dots, j_k}^{(k)}$ for $g_k$) and we have that the correlation $\mathcal{I}_n(f_k,g_k)$ is equal to
\begin{align*}
\sum_{\substack{(i_1,\dots, i_k)\in (\mathbb{Z}^*)^k\\ (j_1\,\dots, j_k)\in (\mathbb{Z}^*)^k}}&\alpha_{i_1,\dots, i_k}^{(k)}\,\beta_{j_1,\dots, j_k}^{(k)}\\
&\times \int_{\mathcal{H}}:\mathfrak{Re}\langle \mathfrak{e}_{i_1},\cdot\rangle\dots\mathfrak{Re}\langle \mathfrak{e}_{i_k},\cdot\rangle :(T^nx):\mathfrak{Re}\langle \mathfrak{e}_{j_1},\cdot\rangle\dots\mathfrak{Re}\langle \mathfrak{e}_{j_k},\cdot\rangle:(x)\,dm(x).
\end{align*}
The integrals above have been computed in Proposition \ref{correlationpoly}:
\begin{align*}
\mathcal{I}_n(f_k,g_k)=\sum_{\tau\in \mathfrak{S}_k}\sum_{\substack{(i_1,\dots, i_k)\in (\mathbb{Z}^*)^k\\ (j_1\,\dots, j_k)\in (\mathbb{Z}^*)^k}}\alpha_{i_1,\dots, i_k}^{(k)}\,\beta_{j_1,\dots, j_k}^{(k)}\,\sigma_{j_1}^2\dots\sigma_{j_k}^2\, \mathfrak{Re}\langle \mathfrak{e}_{i_1},T^n \mathfrak{e}_{j_{\tau(1)}}\rangle\dots \mathfrak{Re}\langle \mathfrak{e}_{i_k},T^n \mathfrak{e}_{j_{\tau(k)}}\rangle.
\end{align*}
Since the sequence $\big(\beta_{j_1,\dots,j_k}^{(k)}\big)_{(j_1,\dots,j_k)\in (\mathbb{Z}^*)^k}$ is symmetric, we obtain that $\mathcal{I}_n(f_k,g_k)$ is equal to
\begin{align*}
k!&\sum_{\substack{(i_1,\dots, i_k)\in (\mathbb{Z}^*)^k\\ (j_1\,\dots,j_k)\in (\mathbb{Z}^*)^k}}\alpha_{i_1,\dots, i_k}^{(k)}\,\beta_{j_1,\dots, j_k}^{(k)}\,\sigma_{j_1}^2\dots\sigma_{j_k}^2\, \mathfrak{Re}\langle \mathfrak{e}_{i_1},T^n \mathfrak{e}_{j_{1}}\rangle\dots \mathfrak{Re}\langle \mathfrak{e}_{i_k},T^n \mathfrak{e}_{j_{k}}\rangle\\
&=k!\sum_{(j_1,\dots, j_k)\in (\mathbb{Z}^*)^k}\beta_{j_1,\dots,j_k}^{(k)}\,\sigma_{j_1}^2\dots\sigma_{j_k}^2\sum_{(i_1,\dots, i_k)\in (\mathbb{Z}^*)^k}\alpha_{i_1,\dots, i_k}^{(k)}\mathfrak{Re}\langle \mathfrak{e}_{i_1},T^n \mathfrak{e}_{j_{1}}\rangle\dots \mathfrak{Re}\langle \mathfrak{e}_{i_k},T^n \mathfrak{e}_{j_{k}}\rangle.
\end{align*}
We now use the triangle inequality and the boundedness of the $k$-linear form $\mathcal{B}_{f_k}$:
\begin{align*}
&\big\vert \mathcal{I}_n(f_k ,g_k)\big\vert \\
&\le k!\sum_{(j_1,\dots, j_k)\in (\mathbb{Z}^*)^k}\big\vert\beta_{j_1,\dots,j_k}^{(k)}\big\vert\,\sigma_{j_1}^2\dots\sigma_{j_k}^2\Big\vert\sum_{(i_1,\dots, i_k)\in (\mathbb{Z}^*)^k}\alpha_{i_1,\dots, i_k}^{(k)}\mathfrak{Re}\langle \mathfrak{e}_{i_1},T^n \mathfrak{e}_{j_{1}}\rangle\dots \mathfrak{Re}\langle \mathfrak{e}_{i_k},T^n \mathfrak{e}_{j_{k}}\rangle\Big\vert\\
&\le k!\,\vert\vert \mathcal{B}_{f_k}\vert\vert\sum_{(j_1,\dots, j_k)\in (\mathbb{Z}^*)^k}\big\vert\beta_{j_1,\dots, j_k}^{(k)}\big\vert\,\sigma_{j_1}^2\dots\sigma_{j_k}^2\,\vert\vert T^n \mathfrak{e}_{j_1}\vert\vert\dots\vert\vert T^n \mathfrak{e}_{j_k}\vert\vert\\
&=k!\,\vert\vert \mathcal{B}_{f_k}\vert\vert\sum_{(j_1,\dots, j_k)\in (\mathbb{Z}^*)^k}\big(\big\vert\beta_{j_1,\dots, j_k}^{(k)}\big\vert\,\sigma_{j_1}\dots\sigma_{j_{k-1}}\big)\big(\sigma_{j_1}\vert\vert T^n \mathfrak{e}_{j_1}\vert\vert\dots\sigma_{j_{k-1}}\vert\vert T^n \mathfrak{e}_{j_{k-1}}\vert\vert\sigma_{j_k}^2\vert\vert T^n \mathfrak{e}_{j_k}\vert\vert\big).
\end{align*}
Then, the Cauchy-Schwarz inequality gives us that $\big\vert \mathcal{I}_n(f_k,g_k)\big\vert$ is less than
\begin{align*}
k!\, \vert\vert \mathcal{B}_{f_k}\vert\vert\,  \Big(\sum_{(j_1,\dots, j_k)\in (\mathbb{Z}^*)^k}\big\vert &\beta_{j_1,\dots,j_k}^{(k)}\big\vert^2\,\sigma_{j_1}^2\dots\sigma_{j_{k-1}}^2\Big)^{1/2}\\
&\bigg(\sum_{j\in \mathbb{Z}^*}\sigma_j^2\,\vert\vert T^n \mathfrak{e}_j\vert\vert^2\bigg)^{(k-1)/2}\bigg(\sum_{j\in \mathbb{Z}^*}\sigma_j^4\,\vert\vert T^n \mathfrak{e}_j\vert\vert^2\bigg)^{1/2}.
\end{align*}
We conclude the proof by using Corollary \ref{ecriture2}, and the estimates \eqref{finitesum1} and \eqref{finite2}:
\begin{align*}
\big\vert \mathcal{I}_n(f_k,g_k)\big\vert \le k!\,\frac{C(E)\,\pi^{\alpha}}{n^{\alpha}}\,\vert\vert E \vert\vert_2^{2k-1}\,\vert\vert \mathcal{B}_{f_k}\vert\vert\,\vert\vert \mathcal{B}_{g_k}\vert\vert.
\end{align*}
\end{proof}
\noindent With this result on the rate of mixing in each space $:\mathcal{G}^k:$, we can prove a general result on the rate of mixing for regular functions in $L^2(\mathcal{H},\mathcal{B},m)$ by considering the Wiener chaos decomposition \eqref{expansion} of our functions.
\subsection{The rate of mixing theorem}
It is now time to define the spaces of functions which will be used in our main theorem. We denote by $\mathcal{X}$ the space of real-valued functions $f$ in $L^2(\mathcal{H},\mathcal{B},m)$ such that the series
$\sum_{k\ge 0}\vert\vert \mathcal{B}_{f_k}\vert\vert^2$
is convergent, where $\mathcal{B}_{f_k}$ is the $k$-linear form \eqref{multilinearform} associated to the component $f_k$ of $f$ in the Wiener chaos decomposition \eqref{expansion} of $f$: $f=\sum_{k\ge 0}f_k$.
 We also introduce the subspace $\mathcal{Y}$ of $\mathcal{X}$ of real-valued functions $g$ such that all the multilinear forms $\mathcal{B}_{g_k}$ are bounded and such that the quantity
$\sup_{k\ge 0}k!\,\vert\vert \mathcal{B}_{g_k}\vert\vert$
is finite. We then endow these two spaces with the norms
\begin{equation*}
\vert\vert f \vert\vert_\mathcal{X}:=\bigg(\vert\vert f \vert\vert_{L^2(m)}^2+\sum_{k=0}^{+\infty}\vert\vert \mathcal{B}_{f_k}\vert\vert^2\bigg)^{1/2}\ \ \ \ \mathrm{and}\ \ \ \ \vert\vert g \vert\vert_{\mathcal{Y}}:=\vert\vert g \vert\vert_{L^2(m)} +\sup_{k\ge 0} \big(k!\, \vert\vert \mathcal{B}_{g_k}\vert\vert\big).
\end{equation*}

\begin{Prop}\label{banach}
\noindent $(i)$ The map $\vert\vert \cdot \vert\vert_\mathcal{X}$ defines a norm on the space $\mathcal{X}$ and $(\mathcal{X},\vert\vert\cdot\vert\vert_\mathcal{X})$ is a Banach space of functions which is contained in $L^2_{\mathbb{R}}(\mathcal{H},\mathcal{B},m)$.\\
\noindent $(ii)$ If $f$ is a real-valued infinitely differentiable function in $L^2(\mathcal{H},\mathcal{B},m)$ such that the series\newline $\sum_{k\ge 0}\frac{\vert\vert \int_\mathcal{H}D_kf(x)\,dm(x) \vert\vert^2}{(k!)^2}$ is convergent, then $f$ belongs to $\mathcal{X}$ and the norm of $f$ can also be written as
\begin{equation*}
\vert\vert f \vert\vert_\mathcal{X}=\bigg(\vert\vert f \vert\vert_{L^2(m)}^2+\sum_{k=0}^{+\infty}\frac{\big\vert\big\vert  \int_\mathcal{H}D^kf(x)\,dm(x)\big\vert\big\vert^2}{(k!)^2}\bigg)^{1/2}\cdot
\end{equation*}
$(iii)$ The map $\vert\vert \cdot \vert\vert_{\mathcal{Y}}$ defines a norm on the space $\mathcal{Y}$ and $(\mathcal{Y},\vert\vert \cdot \vert\vert_{\mathcal{Y}})$ is a Banach space. Furthermore, every real-valued infinitely differentiable function $g$ in $L^2(\mathcal{H},\mathcal{B},m)$ such that 
\begin{equation}\label{conditionintegraleY}
\sup_{k\ge 0}\int_\mathcal{H}\vert\vert D^k g(x)\vert\vert\,dm(x)<+\infty
\end{equation}
belongs to the space $\mathcal{Y}$.
\end{Prop}

\begin{proof}
$(i)$ It is straightforward to check that $\vert\vert \cdot \vert\vert_{\mathcal{X}}$ defines a norm on the space $\mathcal{X}$. We now prove that $(\mathcal{X},\vert\vert \cdot \vert\vert_\mathcal{X})$ is a Banach space. Let $(f_n)_{n\in \mathbb{N}}$ be a Cauchy sequence in the space $(\mathcal{X},\vert\vert \cdot \vert\vert_\mathcal{X})$ where the Wiener chaos decomposition of each function $f_n$ is written as $f_n=\sum_{k\ge 0}f_{n,k}$. We know that for all $\epsilon >0$, there is an integer $n_\epsilon\ge 1$ such that for any $n,m\ge n_\epsilon$, $\big\vert\big\vert f_n-f_m\big\vert\big\vert_{\mathcal{X}}\le \epsilon$ and then that
\begin{equation}\label{banach1} 
\big\vert\big\vert f_n-f_m\big\vert\big\vert_{L^2(m)}\le \epsilon\ \ \ \mathrm{and}\ \ \ \sum_{k\ge 0}\big\vert\big\vert \mathcal{B}_{f_{n,k}}-\mathcal{B}_{f_{m,k}}\big\vert\big\vert^2\le \epsilon^2\ \ \ \mathrm{for\ any\ } n, m \ge n_\epsilon.
\end{equation}
This shows that there is a function $\tilde{f}$ in $L^2_{\mathbb{R}}(\mathcal{H},\mathcal{B},m)$, with Wiener chaos decomposition $\tilde{f}=\sum_{k\ge 0}\tilde{f}_k$, such that $(f_n)_{n\in \mathbb{N}}$ is convergent to $\tilde{f}$ in $(L^2(\mathcal{H},\mathcal{B},m),\vert\vert \cdot \vert\vert_{L^2(m)})$. Furthermore, the second inequality of \eqref{banach1} gives us $\big\vert\big\vert \mathcal{B}_{f_{n,k}}-\mathcal{B}_{f_{m,k}}\big\vert\big\vert\le \epsilon$ for any $n,m\ge n_\epsilon$ and for all $k\ge 0$. Since the space of bounded $k$-linear forms is a Banach space for the norm $\vert\vert \cdot \vert\vert$, it follows that there exists a bounded $k$-linear form $\mathcal{B}_k$ such that $\lim_{n\to +\infty}\big\vert\big\vert \mathcal{B}_{f_{n,k}}-\mathcal{B}_k\big\vert\big\vert=0$. If we expand $\mathcal{B}_k$ as in \eqref{multilinearform}, we easily see that $\mathcal{B}_k=\mathcal{B}_{\tilde{f}_k}$ for any nonnegative integer $k$. We now see that for any $n\ge n_\epsilon$, we have
\begin{align*}
\big\vert\big\vert \tilde{f}-f_n\big\vert\big\vert_\mathcal{X}^2 &=\big\vert\big\vert \tilde{f}-f_n\big\vert\big\vert_{L^2(m)}^2+\sum_{k=0}^{+\infty}\big\vert\big\vert \mathcal{B}_{k}-\mathcal{B}_{f_{n,k}}\big\vert\big\vert^2\\
 &\le\lim_{m\to +\infty}\big\vert\big\vert f_m-f_n\big\vert\big\vert_{L^2(m)}^2+\lim_{m\to +\infty}\sum_{k=0}^{+\infty}\big\vert\big\vert \mathcal{B}_{f_{m,k}}-\mathcal{B}_{f_{n,k}}\big\vert\big\vert^2
\le 2\,\epsilon^2.
\end{align*}
This proves that the function $\tilde{f}$ belongs to $\mathcal{X}$ and that the sequence $(f_n)_{n\in \mathbb{N}}$ is convergent to $\tilde{f}$ in the space $(\mathcal{X},\vert\vert \cdot \vert\vert_{\mathcal{X}})$, that is $(\mathcal{X},\vert\vert\cdot\vert\vert_\mathcal{X})$ is a Banach space.\\
\noindent $(ii)$ is an immediate consequence of Theorem \ref{bounded}.\\
\noindent $(iii)$ In the same way as in the proof of $(i)$, we prove that any Cauchy sequence in $(\mathcal{Y},\vert\vert \cdot \vert\vert_\mathcal{Y})$ is convergent in this space, that is $(\mathcal{Y},\vert\vert \cdot \vert\vert_\mathcal{Y})$ is a Banach space. Furthermore, the fact that $\mathcal{Y}$ contains any real-valued infinitely differentiable function in $L^2(\mathcal{H},\mathcal{B},m)$ satisfying condition \eqref{conditionintegraleY} is a consequence of Theorem \ref{bounded}.
\end{proof}

Before stating our result about the rate of mixing, we point out some classes of functions which belong to the spaces $\mathcal{X}$ or $\mathcal{Y}$. Recall that an entire function $\phi : \mathbb{C}\longrightarrow \mathbb{C}$ is said to be of \textit{exponential type} if there exist constants $M$ and $\tau$ such that for every $r>0$ and $\theta$ in $\mathbb{R}$, $\vert \phi(re^{i\theta})\vert \le Me^{\tau r}$. Letting $\kappa$ stand for the infimum of all such $\tau$, we say that the function $\phi$ is of \textit{exponential type} $\kappa$ (see \cite{B} for more details on these functions).

\begin{Prop}\label{propositionexemples}
\noindent $(i)$ A polynomial in the elements $\mathfrak{Re}\langle \mathfrak{e}_k,\cdot\rangle$ belongs to $\mathcal{Y}$.\\
\noindent $(ii)$ A square-integrable function $f$ of the form
\begin{equation*}
f=\phi\big(\mathfrak{Re}\langle \mathfrak{e}_{-N},\cdot\rangle,\dots,\mathfrak{Re}\langle \mathfrak{e}_N,\cdot\rangle\big),
\end{equation*}
where $\phi : \mathbb{R}^{2N} \longrightarrow \mathbb{R}$ is a real-valued measurable function, belongs to $\mathcal{X}$.\\
\noindent $(iii)$ Let $\phi : \mathbb{C} \longrightarrow \mathbb{C}$ be an entire function of exponential type $\kappa$ such that $\kappa<(2\vert\vert E \vert\vert_2^2)^{-1}$. Then the function $f=\mathfrak{Re}\big(\phi\circ\vert\vert\cdot \vert\vert^2\big)$ belongs to $\mathcal{X}$.
\end{Prop} 

\begin{proof}
Assertion $(i)$ follows immediately from Proposition \ref{banach}. In order to prove $(ii)$, we show that $\vert\vert f \vert\vert_\mathcal{X}$ can be controlled by $\vert\vert f \vert\vert_{L^2(m)}$ in this case. The function $f$ can be written as $f=\sum_{k\ge 0}f_k$ where
$$
f_k=\sum_{(j_1,\dots, j_k)\in (\{-N,\dots, N\}\setminus\{0\})^k}\alpha_{j_1,\dots, j_k}^{(k)}:\mathfrak{Re}\langle \mathfrak{e}_{j_1},\cdot\rangle\dots\mathfrak{Re}\langle \mathfrak{e}_{j_k},\cdot\rangle:
$$
and with
$$
\vert\vert f \vert\vert_{L^2(m)}^2=\sum_{k=0}^{+\infty}\vert\vert f_k \vert\vert_{L^2(m)}^2<+\infty.
$$
We now compute the norm $\vert\vert f \vert\vert_{L^2(m)}$. For every positive integer $k$ we have
\begin{align*}
\notag &\vert\vert f_k \vert\vert_{L^2(m)}^2=\sum_{\substack{(i_1,\dots, i_k)\in (\{-N,\dots, N\}\setminus\{0\})^k\\(j_1,\dots, j_k)\in (\{-N,\dots, N\}\setminus\{0\})^k}}\alpha_{i_1,\dots, i_k}^{(k)}\alpha_{j_1,\dots, j_k}^{(k)}\\ &\qquad \qquad\times\int_{\mathcal{H}}:\mathfrak{Re}\langle \mathfrak{e}_{i_1},\cdot\rangle\dots\mathfrak{Re}\langle \mathfrak{e}_{i_{k}},\cdot\rangle:(x):\mathfrak{Re}\langle \mathfrak{e}_{j_1},\cdot\rangle\dots\mathfrak{Re}\langle \mathfrak{e}_{j_{k}},\cdot\rangle:(x)\,dm(x).
\end{align*}
By using the same method of computation as in Corollary \ref{sommeconvergente}, we obtain
\begin{align*}
\vert\vert f_k\vert\vert_{L^2(m)}^2=k!\sum_{(i_1,\dots, i_k)\in (\{-N,\dots, N\}\setminus\{0\})^k}\big\vert\alpha_{i_1,\dots ,i_k}^{(k)}\big\vert^2\sigma_{i_1}^2\dots \sigma_{i_k}^2.
\end{align*}
and then
\begin{equation}\label{normefexemple}
\vert\vert f\vert\vert_{L^2(m)}^2=\sum_{k=0}^{+\infty}k!\sum_{(i_1,\dots, i_k)\in (\{-N,\dots, N\}\setminus\{0\})^k}\big\vert\alpha_{i_1,\dots ,i_k}^{(k)}\big\vert^2\sigma_{i_1}^2\dots \sigma_{i_k}^2.
\end{equation}
Furthermore, the $k$-linear form $\mathcal{B}_{f_k}$ is defined by
$$
\mathcal{B}_{f_k}\big(x^{(1)},\dots, x^{(k)}\big)=\sum_{(i_1,\dots, i_k)\in (\{-N,\dots, N\}\setminus\{0\})^k}\alpha_{i_1,\dots, i_k}^{(k)}x_{i_1}^{(1)}\dots x_{i_k}^{(k)}
$$
for every vectors $x^{(1)},\dots, x^{(k)}$ in $\ell(\mathbb{Z}^*,\mathbb{R})$ and the Cauchy-Schwarz inequality gives us
$$
\vert\vert\mathcal{B}_{f_k}\vert\vert\le \Bigg(\sum_{(i_1,\dots, i_k)\in (\{-N,\dots, N\}\setminus\{0\})^k}\big\vert\alpha_{i_1,\dots, i_k}^{(k)}\big\vert^2\Bigg)^{1/2}.
$$
We can conclude that
$$
\sum_{k=0}^{+\infty}\vert\vert\mathcal{B}_{f_k}\vert\vert^2\le \sum_{k=0}^{+\infty}\sum_{(i_1,\dots, i_k)\in (\{-N,\dots, N\}\setminus\{0\})^k}\big\vert\alpha_{i_1,\dots, i_k}^{(k)}\big\vert^2.
$$
Let $\tilde{\sigma}=\min(\sigma_{-N},\dots,\sigma_{N})$. Then $\tilde{\sigma}^k\le \sigma_{i_1}\dots \sigma_{i_k}$ for every $k$-tuple $(i_1,\dots, i_k)$ belonging to $(\{-N,\dots, N\}\setminus\{0\})^k$ and the sequence $(k!\tilde{\sigma}^{2k})_{k\ge 0}$ tends to infinity. Then there exists a constant $C_N>0$ such that $k!\tilde{\sigma}^{2k}\ge C_N$ for every positive integer $k$. It follows that 
$$
\sum_{k=0}^{+\infty}\vert\vert\mathcal{B}_{f_k}\vert\vert^2 \le C_N^{-1}\vert\vert f \vert\vert_{L^2(m)}^2.
$$
This proves that $\vert\vert f \vert\vert_\mathcal{X}^2\le (1+C_N^{-1})\vert\vert f \vert\vert_{L^2(m)}^2$ and that $f$ belongs to the space $\mathcal{X}$.\newline
\noindent We now deal with assertion $(iii)$. There exist constants $\kappa<\tau<(2\vert\vert E \vert\vert_2^2)^{-1}$ and $M$ such that $\vert \phi(re^{i\theta})\vert\le Me^{\tau r}$ for any $r$ and $\theta$. In order to prove that $f$ belongs to $L^2(\mathcal{H},\mathcal{B},m)$, it suffices to show that the function $e^{\tau\vert\vert \cdot \vert\vert^2}$ belongs to $L^2(\mathcal{H},\mathcal{B},m)$ (recall that $\vert\vert E \vert\vert_2^2=2\sum_{k\ge 1}\sigma_k^2$). Since $(\vert \langle e_k,\cdot \rangle \vert)_{k\in \mathbb{N}}$ is a sequence of independant complex random variables, we have
\begin{align*}
\int_{\mathcal{H}}e^{2\tau\vert\vert x \vert\vert^2}\,dm(x)=\prod_{k=1}^{+\infty}\bigg(\int_\mathcal{H}e^{2\tau(\mathfrak{Re}\langle e_k,x\rangle)^2}\,dm(x)\bigg)^2 &=\prod_{k=1}^{+\infty}\bigg(\int_{\mathbb{R}}e^{-\big(\frac{1}{2\sigma_k^2}-2\tau\big)t^2}\,\frac{dt}{\sigma_k\sqrt{2\pi}}\bigg)^2\\
&=\prod_{k=1}^{+\infty}\bigg(\frac{1}{1-4\tau\sigma_k^2}\bigg)
\end{align*}
and this infinite product is convergent since the series $\sum_{k\ge 1}\sigma_k^2$ is convergent, which proves that $f$ belongs to $L^2(\mathcal{H},\mathcal{B},m)$. We are now going to prove that the infinitely differentiable function $g=\phi\circ\vert\vert\cdot\vert\vert^2$ is such that the series $\sum_{k\ge 1}(k!)^{-1}\big\vert\big\vert \int_\mathcal{H}D^k g(x)\,dm(x)\big\vert\big\vert$ is convergent (it will follow that $f$ belongs to the space $\mathcal{X}$). To do this, we expand our entire function $\phi$ as 
\begin{equation}\label{entire}
\phi(z)=\sum_{k\ge 0}\frac{a_k}{k!}z^k. 
\end{equation}
We will need a characterization of functions of exponential type in terms of the coefficients $a_k$. It is a well known result that an entire function written as in \eqref{entire} is of exponential type $\kappa$ if and only if $\overline{\lim}_{n\to +\infty}\vert a_n\vert^{1/n}=\kappa$ (see for instance \cite{B}). In particular, there exists a positive constant $C$ such that for any nonnegative integer $n$, we have
\begin{equation} \label{type}
\vert a_n \vert \le C \tau^n.
\end{equation}
In order to compute the derivatives of $g$, we introduce the symmetric bilinear function $A:\mathcal{H}\times \mathcal{H}\longrightarrow \mathbb{R}$ which is defined by $A(u,v)=\sum_{k\ge 1}\big(\mathfrak{Re}\langle e_k,u\rangle \mathfrak{Re}\langle e_k,v\rangle+\mathfrak{Im}\langle e_k,u\rangle \mathfrak{Im}\langle e_k,v\rangle\big)$. Then it is rather easy to see that the derivatives of $g$ can be written as follows:
\begin{equation*}
D^{2n}g(x)(h_1,\dots, h_{2n})=\sum_{j=0}^{n} 2^{n+j}\mathcal{S}_{2j}(x,h_1,\dots, h_{2n})\sum_{k=0}^{+\infty}\frac{a_{k+n+j}}{k!}\vert\vert x \vert\vert^{2k}
\end{equation*}
where for any $j\in \{0,\dots,n\}$, $\mathcal{S}_{2j}(x,h_1,\dots,h_{2n})$ is the sum of all the terms of the form 
\begin{equation*}
A(x,h_{i_1})\dots A(x,h_{i_{2j}})A(h_{i_{2j+1}},h_{i_{2j+2}})\dots A(h_{i_{2n-1}},h_{i_{2n}})
\end{equation*}
with $\{i_1,\dots,i_{2n}\}=\{1,\dots, 2n\}$ and
\begin{equation*}
D^{2n+1}g(x)(h_1,\dots,h_{2n+1})=\sum_{j=0}^{n}2^{n+j+1}\mathcal{S}_{2j+1}(x,h_1,\dots, h_{2n+1})\sum_{k=0}^{+\infty}\frac{a_{k+n+j+1}}{k!}\vert\vert x \vert\vert^{2k},
\end{equation*}
where for any $j\in \{0,\dots, n\}$, $\mathcal{S}_{2j+1}(x,h_1,\dots, h_{2n+1})$ is the sum of all the terms of the form
\begin{equation*}
A(x,h_{i_1})\dots A(x,h_{i_{2j+1}})A(h_{i_{2j+2}},h_{i_{2j+3}})\dots A(h_{i_{2n}},h_{i_{2n+1}})
\end{equation*}
with $\{i_1,\dots,i_{2n+1}\}=\{1,\dots, 2n+1\}$. In order to estimate $\vert D^n g(x)(h_1,\dots, h_n)\vert$, we need to compute the number of terms which appear in the sums 
$$
\mathcal{S}_{2j}(x,h_1,\dots,h_{2n})\qquad \textrm{and}\qquad \mathcal{S}_{2j+1}(x,h_1,\dots,h_{2n+1})
$$
Since the number of partitions by pairs of a set of $2p$ elements is equal to $\frac{(2p)!}{2^p p!}$, we easily see that there is exactly $\binom{2n}{2j}\frac{(2(n-j))!}{2^{n-j}(n-j)!}$ terms in the sum $\mathcal{S}_{2j}(x,h_1,\dots,h_{2n})$ and $\binom{2n+1}{2j+1}\frac{(2(n-j))!}{2^{n-j}(n-j)!}$ terms in the sum $\mathcal{S}_{2j+1}(x,h_1,\dots,h_{2n+1})$. Then, since $\vert A(u,v)\vert \le 2\, \vert\vert u \vert\vert\, \vert\vert v \vert\vert$, we get for any $h_1,\dots, h_{2n}$ in the closed unit ball $\mathbb{B}$ of $\mathcal{H}$, $\vert\mathcal{S}_{2j}(x,h_1,\dots,h_{2n})\vert \le \binom{2n}{2j}\frac{(2(n-j))!}{2^{n-j}(n-j)!}2^{n+j}\,\vert\vert x \vert\vert^{2j}$ and then by using \eqref{type}:
\begin{align*}
\big\vert D^{2n} g(x)(h_1,\dots, h_{2n})\big\vert &\le\sum_{j=0}^{n}2^{n+j}\binom{2n}{2j}\frac{(2(n-j))!}{2^{n-j}(n-j)!}2^{n+j}\vert\vert x \vert\vert^{2j}\sum_{k=0}^{+\infty}\frac{C\tau^{k+n+j}}{k!}\vert\vert x \vert\vert^{2k}\\
&\le C (2n)!\sum_{j=0}^{n}\frac{(4\tau)^{n+j}}{(n-j)!}\frac{\vert\vert x \vert\vert^{2j}}{(2j)!}e^{\tau \vert\vert x \vert\vert^2}.
\end{align*}
By using the same method, we find that for any $h_1,\dots, h_{2n+1}$ in $\mathbb{B}$,
\begin{equation*}
\big\vert D^{2n+1} g(x)(h_1,\dots,h_{2n+1}) \big\vert\le C (2n+1)!\sum_{j=0}^{n}\frac{(4\tau)^{n+j+1}}{(n-j)!}\frac{\vert\vert x \vert\vert^{2j+1}}{(2j+1)!}e^{\tau \vert\vert x \vert\vert^2}.
\end{equation*}
Then, we have
\begin{align*}
\sum_{n=0}^{+\infty}\frac{\big\vert\big\vert \int_\mathcal{H} D^{2n} g(x)\, dm(x)\big\vert\big\vert}{(2n)!}&\le C\int_\mathcal{H} e^{\tau \vert\vert x \vert\vert^2}\sum_{n=0}^{+\infty}\sum_{j=0}^{n}\frac{(4\tau)^{n+j}}{(n-j)!}\frac{\vert\vert x \vert\vert^{2j}}{(2j)!}\,dm(x)\\
&=C\int_\mathcal{H} e^{\tau \vert\vert x \vert\vert^2}\sum_{j=0}^{+\infty}\frac{(4\tau \vert\vert x \vert\vert)^{2j}}{(2j)!}\bigg(\sum_{n=j}^{+\infty}\frac{(4\tau)^{n-j}}{(n-j)!}\bigg)\,dm(x)\\
&=C e^{4\tau}\int_\mathcal{H}e^{\tau \vert\vert x \vert\vert^2}\sum_{j=0}^{+\infty}\frac{(4\tau \vert\vert x \vert\vert)^{2j}}{(2j)!}\,dm(x),
\end{align*}
and in the same way,
\begin{equation*}
\sum_{n=0}^{+\infty}\frac{\big\vert\big\vert \int_\mathcal{H} D^{2n+1} g(x)\, dm(x)\big\vert\big\vert}{(2n+1)!}\le C e^{4\tau}\int_\mathcal{H}e^{\tau \vert\vert x \vert\vert^2}\sum_{j=0}^{+\infty}\frac{(4\tau \vert\vert x \vert\vert)^{2j+1}}{(2j+1)!}\,dm(x).
\end{equation*}
We finally conclude that
\begin{equation*}
\sum_{n=0}^{+\infty}\frac{\big\vert\big\vert \int_\mathcal{H} D^{n}g(x)\, dm(x)\big\vert\big\vert}{n!}\le C e^{4\tau}\int_\mathcal{H}e^{\tau \vert\vert x \vert\vert^2 +4\tau \vert\vert x \vert\vert}\,dm(x).
\end{equation*}
According to the beginning of the proof, the integral $\int_\mathcal{H}e^{2\tau \vert\vert x \vert\vert^2}\,dm(x)$ is convergent and the function $f$ belongs to the space $\mathcal{X}$.
\end{proof}

\begin{Rem} It is not difficult to see that the space $\mathcal{X}$ is smaller than the whole space $L^2_{\mathbb{R}}(\mathcal{H},\mathcal{B},m)$. Indeed, we can take a function $f$ in $\mathcal{G}$ which is written as in Remark \ref{decompositiong}: $f=\sum_{k\in \mathbb{Z}^*}a_k\, \mathfrak{Re}\langle \mathfrak{e}_k,\cdot \rangle$ where $(a_k)_{k\in \mathbb{Z}^*}$ is a sequence of real numbers such that $\sum_{k\in \mathbb{Z}^*}a_k^2\,\sigma_k^2<+\infty$. It is clear that $f=f_1$ (in the Wiener chaos decomposition of $f$) and that $\vert\vert\mathcal{B}_{f_1}\vert\vert=\vert\vert (a_k)_{k\in \mathbb{Z}^*}\vert\vert_2$ which is not finite if the sequence $(a_k)_{k\in \mathbb{Z}^*}$ does not belong to $\ell_2(\mathbb{Z}^*,\mathbb{R})$. 
\end{Rem}

We finally state and prove our result on the decrease of correlations when we consider a function $f$ in the space $\mathcal{X}$ and a function $g$ in the space $\mathcal{Y}$.
\begin{Theo}\label{theoremefinal}
Let $T\in \mathcal{B}(\mathcal{H})$ be a bounded linear operator on $\mathcal{H}$ whose eigenvectors associated to unimodular eigenvalues are parametrized by a $\mathbb{T}$-eigenvector field $E$ which is $\mu$-spanning and $\alpha$-H\"olderian as in Assumption \ref{assumption}. Then, there exists a positive constant $C'(E)$ such that for any $f\in \mathcal{X}$ and $g\in\mathcal{Y}$, we have
\begin{equation*}
\big\vert \mathcal{I}_n(f,g) \big\vert \le \frac{C'(E)}{n^{\alpha}}\,\vert\vert f \vert\vert_{\mathcal{X}}\, \vert\vert g\vert\vert_{\mathcal{Y}}
\end{equation*}
for any positive integer $n$.
\end{Theo}
\begin{proof}
We consider the Wiener chaos decomposition of the functions $f\in \mathcal{X}$ and $g\in \mathcal{Y}$, that is $f=\sum_{k\ge 0}f_k$ and $g=\sum_{k\ge 0}g_k$, where $f_k, g_k$ belong to the space $:\mathcal{G}^k:$. Since the spaces $:\mathcal{G}^j:$ are orthogonal, we know that for any positive integer $n$, we have
\begin{equation*}
\mathcal{I}_n(f,g)=\sum_{k=1}^{+\infty}\mathcal{I}_n(f_k,g_k).
\end{equation*}
The functions $f_k$ and $g_k$ belong to the same space $:\mathcal{G}^k:$, so we can apply Theorem \ref{correlations}. Hence, by using the triangle inequality, we get
\begin{align*}
\big\vert \mathcal{I}_n(f,g) \big\vert &\le \sum_{k=1}^{+\infty}\big\vert \mathcal{I}_n(f_k,g_k)\big\vert \le \frac{C(E)\,\pi^{\alpha}}{n^{\alpha}}\sum_{k=1}^{+\infty}k!\,\vert\vert E \vert\vert_2^{2k-1}\,\vert\vert \mathcal{B}_{f_k}\vert\vert\,\vert\vert \mathcal{B}_{g_k}\vert\vert\\
&\le \frac{C(E)\,\pi^{\alpha}}{n^{\alpha}}\bigg(\sum_{k=1}^{+\infty}\vert\vert E \vert\vert_2^{4k-2}\bigg)^{1/2}\vert\vert f \vert\vert_\mathcal{X}\, \vert\vert g \vert\vert_\mathcal{Y}
\end{align*}
where the series $\sum_{k\ge0}\vert\vert E \vert\vert_2^{4k-2}$ can always be assumed to be convergent by taking, if necessary, a smaller H\"older constant $C(E)$. Then the result follows with the constant
$$
C'(E):=C(E)\pi^{\alpha}\Big(\sum_{k\ge 1}\vert\vert E \vert\vert_2^{4k-2}\Big)^{1/2}.
$$
\end{proof}

By using the examples of functions we find in the spaces $\mathcal{X}$ and $\mathcal{Y}$ in Proposition \ref{propositionexemples}, we can for instance deduce from Theorem \ref{theoremefinal} the following corollary.
\begin{Cor}
Let $T\in \mathcal{\mathcal{H}}$ whose eigenvectors associated to unimodular eigenvalues are parametrized by a $\mathbb{T}$-eigenvector field $E$ which is $\mu$-spanning and $\alpha$-H\"olderian as in Assumption \ref{assumption}. For any entire function $\phi:\mathbb{C}\longrightarrow \mathbb{C}$ of exponential type $\kappa<(2\vert\vert E \vert\vert_2^2)^{-1}$ and any polynomial $p$ in the variables $\mathfrak{Re}\langle \mathfrak{e}_i,\cdot\rangle$, we have
\begin{equation*}
\big\vert \mathcal{I}_n\big(\phi(\vert\vert \cdot \vert\vert^2),p\big)\big\vert \le \frac{C'(E)}{n^\alpha}\big\vert\big\vert \mathfrak{Re}\big(\phi\circ\vert\vert\cdot\vert\vert^2\big) \big\vert\big\vert_\mathcal{X}\,\vert\vert p \vert\vert_\mathcal{Y}
\end{equation*}
for any positive integer $n$, where the positive constant $C'(E)$ appears in Theorem \ref{theoremefinal}.
\end{Cor}

\begin{Rem}
In a more general situation, the $\mathbb{T}$-eigenvectors of $T\in \mathcal{B}(\mathcal{H})$ are parametrized by a countable family $(E_i)_{i\in I}$ of $\mathbb{T}$-eigenvector fields (see Remark \ref{plusieurs} for the definitions of the operators $K_{E_i}$ and $K$). Under the stronger assumption that the series $\sum_{i\in I}\alpha_i^2\vert\vert E_i\vert\vert_2$ is convergent and that the $\mathbb{T}$-eigenvector fields are H\"olderian with the same H\"older exponent $\alpha\in (0,1]$ (and with the same H\"older constant $C(E_i):=C$), one can easily proves that Proposition \ref{functional} remains true, that is the sequence $(\langle RT^{*n}x,y\rangle)_{n\in \mathbb{N}}$ is convergent to zero with speed $n^{-\alpha}$ (with the constant $C\pi^{\alpha}\sum_{i\in I}\alpha_i^2\vert\vert E_i\vert\vert_2$ instead of the constant $C(E,\alpha)$). Then, it is not difficult to prove that Lemma \ref{norm1} and Lemma \ref{seriemajoree} can be proved in this context and then that we have the same kind of result that Theorem \ref{theoremefinal}.
\end{Rem}

\subsection{Analytic $\mathbb{T}$-eigenvector fields}

In many cases (see for instance Example \ref{analytic} or various examples given in the book \cite{BM} on composition operators), a stronger regularity assumption on our $\mathbb{T}$-eigenvector field holds true: the $\mathbb{T}$-eigenvector field $E$ is a vector-valued analytic function in a neighbourhood of $\mathbb{T}$. In this case, the convergence to zero of the correlations $\mathcal{I}_n(f,g)$ is exponential when the functions $f$ and $g$ belong to the spaces $\mathcal{X}$ and $\mathcal{Y}$ respectively. More precisely:
\begin{Theo}\label{theoanalytic}
Let $T\in \mathcal{B}(\mathcal{H})$ be a bounded linear operator on $\mathcal{H}$ whose eigenvectors associated to unimodular eigenvalues are parametrized by a $\mathbb{T}$-eigenvector field $E$ which is $\mu$-spanning and analytic. Then there exists $0<t<1$ and a positive constant $D(E)$ such that for any $f\in \mathcal{X}$ and $g\in \mathcal{Y}$, 
$$
\vert \mathcal{I}_n(f,g)\vert\le D(E)\, t^n\, \vert\vert f \vert\vert_{\mathcal{X}}\, \vert\vert g \vert\vert_{\mathcal{Y}}
$$ 
for any positive integer $n$.
\end{Theo}

Indeed, the two only results where we use the fact the the $\mathbb{T}$-eigenvector field $E$ is \textit{regular} (H\"olderian or analytic) is Proposition \ref{functional} and Lemma \ref{norm1}. Assume that $E: \mathbb{T}\longrightarrow \mathcal{H}$ is a vector-valued analytic function (in a neighbourhood of $\mathbb{T}$). Then there exists a sequence of complex numbers $(c_n)_{n\in \mathbb{Z}_+}$ such that $E(\lambda)=\sum_{n\ge 0}c_n\lambda^n e_n$ and there exist two constants $t\in (0,1)$ and $M>0$ such that $\vert c_n\vert\le Mt^n$ for every nonnegative integer $n$. For instance, if one rewrite the proof of Proposition \ref{functional}, then he gets
\begin{align*}
\big\vert\langle RT^{*n}x,y\rangle\big\vert &=\bigg\vert\int_\mathbb{T}\lambda^n \langle x,E(\lambda)\rangle\overline{\langle y,E(\lambda)\rangle}\,d\mu(\lambda)\bigg\vert\\
&=\Bigg\vert\sum_{q=0}^{+\infty}c_{q+n}\overline{c_q}\langle x,e_{q+n}\rangle\overline{\langle y,e_q\rangle}\Bigg\vert\\
&\le M^2 t^n\vert\vert x\vert\vert\,\vert\vert y\vert\vert
\end{align*}
by using the Cauchy-Schwarz inequality and since $0<t<1$. Then we obtain the same kind of result as in Proposition \ref{functional} and the correlations decrease to zero with exponential speed. If we do the same thing with Lemma \ref{norm1}, we find that the correlations $\mathcal{I}_n\big(\vert\vert\cdot\vert\vert^2,\vert\vert\cdot\vert\vert^2\big)$ decrease to zero with exponential speed $t^n$. With these two results, we can prove Theorem \ref{theoanalytic} in the context of analytics $\mathbb{T}$-eigenvector fields.

\vspace*{0.5cm}

\noindent\emph{Acknowledgement}: I am grateful to the referee for many valuable suggestions on the presentation of the paper. In particular, I thank the referee for providing me with a complete answer to Question \ref{question}. I am also grateful to my advisor, Sophie Grivaux, for helpful discussions on this subject of speed of mixing.

\noindent \textit{Vincent Devinck}

\noindent \textit{Laboratoire Paul Painlev\'e}

\noindent \textit{UMR 8524}

\noindent \textit{Universit\'e des Sciences et Technologies de Lille}

\noindent \textit{Cit\'e Scientifique}

\noindent \textit{59655 Villeneuve d'Ascq cedex}

\noindent \textit{France}

\noindent \texttt{devinck.vincent@gmail.com}

\end{document}